\documentclass[3p]{elsarticle}

\usepackage{hyperref}

\journal{arXiv}

\usepackage{amsfonts}
\usepackage{amsmath, amscd,amssymb,bm,amsbsy,epsf}
    \usepackage{enumerate}
    \usepackage{paralist}
\usepackage{bbm, dsfont}
\usepackage{graphicx,color}
\usepackage{subfigure}
\usepackage{mathrsfs}
\usepackage{verbatim}
\usepackage{inputenc}
\usepackage{bbm}
\usepackage{multicol}
\usepackage{balance}
\usepackage{verbatim, booktabs}

\usepackage{multirow, url}
\usepackage[table]{xcolor}

\usepackage{sfmath} 

\DeclareMathAlphabet{\mathpzc}{OT1}{pzc}{m}{it}
\DeclareMathOperator*{\cl}{cl}

\newtheorem{theorem}{Theorem}
\newtheorem{lemma}[theorem]{Lemma}
\newtheorem{corollary}[theorem]{Corollary}
\newtheorem{proposition}[theorem]{Proposition}
\newdefinition{definition}{Definition}
\newdefinition{hypothesis}{Hypothesis}
\newdefinition{remark}{Remark}
\newproof{proof}{Proof}
\newdefinition{example}{Example}

\def\A{{\mathcal A}}
\def\B{{\mathcal B}}
\def\C{{\mathcal C}}
\def\V{\mathbf V}
\def\X{\mathbf X}
\def\Y{\mathbf Y}
\def\L{\mathbf L}
\def\0{\mathbf 0}
\def\g{\mathbf g}
\def\n{\bm{\widehat{ n} } }
\def\u{\mathbf u}
\def\uh{\mathbf {u}^{h}}
\def\ph{p^{h}}
\def\v{\mathbf v}

\def\x{\mathbf x}
\def\y{\mathbf y}
\def\div{\bm{\nabla} \cdot}
\def\grad{\bm{\nabla}}

\def\uone{\mathbf{u}_{1}}
\def\uoneh{\mathbf{u}_{1}^{h}}
\def\uonehp{\mathbf{u}_{1}^{h'}}
\def\utwo{\mathbf{u}_{2}}
\def\utwoh{\mathbf{u}_{2}^{h}}
\def\R{\bm{\mathbbm{R} } }

\def\Hdiv{\mathbf{H_{div}}}
\def\wone{\mathbf{w}_{1}}
\def\wtwo{\mathbf{w}_{2}}
\def\vone{\mathbf{v}_{1}}
\def\voneh{\mathbf{v}_{1}^{h}}
\def\vonehp{\mathbf{v}_{1}^{h'}}
\def\vtwo{\mathbf{v}_{2}}
\def\vtwoh{\mathbf{v}_{2}^{h}}
\def\qone{q_{1}}
\def\qoneh{q_{1}^{h}}
\def\qtwo{q_{2}}
\def\qtwoh{q_{2}^{h}}
\def\rone{r_{1}}
\def\rtwo{r_{2}}
\def\pone{p_{1}}
\def\poneh{p_{1}^{h}}
\def\ponehp{p_{1}^{h'}}
\def\ptwo{p_{2}}
\def\ptwoh{p_{2}^{h}}
\def\ptwohp{p_{2}^{h'}}
\def\outer{\bm {\widehat{\nu} } }
\def\flux{f_{\bm{\hat{n}} } }
\def\stress{f_{\Sigma } }

\def\defining{\overset{\mathbf{def}} =}

\def\triang{\mathcal{T}}
\def\tone{\mathcal{T}_{1} }
\def\ttwo{\mathcal{T}_{2}} 
\def\map{\mathcal{G} }
\def\mapone{\mathcal{G}_{1} }
\def\maptwo{\mathcal{G}_{2}} 
\def\Omeone{\Omega_{1}}
\def\Ometwo{\Omega_{2}}

\def\ind{\boldsymbol{\mathbbm{1}}}

\begin{document}

\begin{frontmatter}

\title{A Conforming Primal-Dual Mixed Formulation \\
for the 
2D Multiscale Porous Media Flow Problem}
\tnotetext[mytitlenote]{This material is based upon work supported by project HERMES 27798 from Universidad Nacional de Colombia,
Sede Medell\'in.}

\author[mymainaddress]{Fernando A Morales } 


\cortext[mycorrespondingauthor]{Corresponding Author}
\ead{famoralesj@unal.edu.co}

\address[mymainaddress]{Escuela de Matem\'aticas
Universidad Nacional de Colombia, Sede Medell\'in \\
Calle 59 A No 63-20 - Bloque 43, of 106,
Medell\'in - Colombia}


\begin{abstract}
In this paper a new primal-dual mixed finite element method is introduced, aimed to model multiscale problems with several geometric subregions in the domain of interest. In each of these regions porous media fluid flow takes place, but governed by physical parameters at a different scale; additionally, a fluid exchange through contact interfaces occurs between neighboring regions. The well-posedness of the primal-dual mixed finite element formulation on bounded simply connected polygonal domains of the plane is presented. Next, the convergence of the discrete solution to the exact solution of the problem is discussed, together with the convergence rate analysis. Finally, the numerical examples illustrate the method's capabilities to handle multiscale problems and interface discontinuities as well as experimental rates of convergence. 
\end{abstract}


\end{frontmatter}



%
%

%
%
%
%
%
\section{Introduction}   
%
%
%
%
Mixed variational formulations are a very important topic of research in applied mathematics. The Babuska-Brezzi theory (see \textsc{Theorem} \ref{Th well posedeness mixed formulation classic}) is remarkably powerful from the theoretical point of view, however it introduces high complexity in the discrete finite element spaces approximating the solution; this reflects in numerical stability problems  (see \cite{Hughes}). The achievements to overcome such difficulty can be on several directions. One of the streams seeks to stabilize the approximation by modifying the bilinear forms involved, namely using symmetric properties of the tensors as in  \cite{Gatica2}, \cite{Gatica1}, or including terms in the bilinear forms in a ``balanced" way as in\cite{FortinBrezzi2}, \cite{Hughes}. As this technique has proved to be fruitful and rich in terms of the possibilities to stabilize the forms of interest, some other aspects arise by itself, such as the discussion of minimal stabilisation procedures (see \cite{FortinBrezzi3}), or the a-priori, a-posteriori error analysis for these new scheme (see \cite{Gatica2}). A second approach uses discontinuous Galerkin finite elements (DG). The DG methods have several advantages and goals, some of these are: addressing non-conformality in a more flexible way, treating stability issues due to coupling constraints (demanding regularity in the discrete spaces), and computing in a more accurate way the physical quantity that is known to be predominant in specific subregions. The latter is attained in two ways, by local refining of the mesh and by approximating polynomial spaces; see \cite{Brezzi1}, \cite{Cockburn} for a unified vision of the DG Methods.   

All the aforementioned works, whichever the problem they may be analyzing (elasticity, heat diffusion, free flow, Darcy flow, etc), treat separately the primal and dual mixed formulations (see \cite{BrezziFortin}, \cite{GiraultRaviartNS}, \cite{RaviartThomas1977}). The present paper is focused on using simultaneously both fundamental versions for the treatment of multiscale problems in Darcy flow (see \textsc{Problem} \eqref{Eq porous media strong}), it is therefore a primal-dual mixed formulation; in a way this article is the numerical implementation of the formulation introduced in \cite{MoralesNaranjo} (see also \cite{MoralesShow2}, \cite{Morales2} for related formulations). The stability aspects become particularly critical when dealing with multiscale problems, as the presence of physical coefficients with different orders of magnitude adds up to the built-in complexity of the mixed variational formulations (coefficient $a(\cdot)$ in \textsc{Problem}  \eqref{Eq porous media strong}). The primal-dual mixed formulation tackles this issue by removing coupling constraints from the discrete trial spaces while satisfying them only on the solution i.e., the continuous formulation replaces \textit{strong coupling} conditions by \textit{weak coupling} conditions (see \textsc{Equations} \eqref{Eq interface balance conditions} and \textsc{Problem} \eqref{Pblm weak continuous solution}). Replacing the nature of the coupling conditions is a strategy already used in DG methods using penalization techniques; however, this is done only on the discrete version, while the continuous formulation still relies on strong coupling conditions. The latter is because, in the Darcy flow problem, while the primal mixed formulation can introduce weak coupling conditions on the normal flow exchange, the normal stress has to stay continuos. In contrast, the dual mixed formulation can introduce weak coupling conditions for the normal stress balance, but it requires the normal flow exchange to be continuous. The continuity constraints of the classical mixed formulations reflect later on, in the deep discussions of convergence present in the DG methods. 

Another advantage of the discrete primal-dual mixed formulation we are to introduce in this work is that, according to the regions, the predominant effect can be chosen to be modeled with the discrete space holding the sense of continuity, while the secondary effect is modeled with the discontinuous space. In the case of Darcy flow, the pressure is the dominant effect in regions of low permeability, while the flow velocity is the predominant one in regions of high permeability (see \textsc{Figures}  \ref{Fig Approximate Solution Numerical Example}, \ref{Fig Disc Approximate Solution Numerical Example} and  \ref{Fig Multiscale Jump Approximate Solution Numerical Example}).  This concept has already araised naturally in previous DG methods coupling advection with diffusion phenomena, due to the discrete spaces involved in the formulations, see \cite{Dawson}. To the author's best knowledge there is no precedent for having this level of flexibility in the analysis of coupling fluid flow phenomena, as the literature analyzing multiscale flow is mainly focused in coupling Stokes flow with Darcy flow, see \cite{ArbBrunson2007}, \cite{ArbLehr2006}, \cite{Gatica3}, \cite{Layton}, \cite{MoralesShow3}.

The proposed model is to analyze a variation of the classic porous media problem on a connected bounded open region $\Omega\subset \R^{2}$, i.e., 
\begin{subequations}\label{Eq porous media strong}
\begin{equation}\label{Eq Darcy Strong}
a(\cdot)\, \u + \grad p  + \g = 0\,,
\end{equation}
\begin{equation}\label{Eq conservative strong}
\div \u = F\,\quad \mathrm{in}\; \Omega .
\end{equation}
\begin{equation}\label{Eq Drained Condition}
p = 0 \,\quad \mathrm{on}\; \Gamma_{d} ,
\end{equation}
\begin{equation}\label{Eq Non-Flux Condition}
\u\cdot\n = 0 \quad \mathrm{on}\; \Gamma_{f}\defining\partial \Omega - \Gamma_{d} ;
\end{equation}
\end{subequations}
more specifically, when $\Omega$ is partitioned in two subdomains $\Omega_{1}, \Omega_{2}$ such that $a(\cdot)\big\vert_{\Omega_{1}} = O(1)$ and $a(\cdot)\big\vert_{\Omega_{2}} = O(\epsilon)$ for $\epsilon>0$ small (see \textsc{Figure} \ref{Fig Bipartite Map and Consistent Grid}). Recall that $ a(\cdot) $ is the flow resistance i.e., the viscosity times the inverse of permeability of the porous medium. Systems such as this, are suited for the modeling of oil reservoirs and subsurface water, where a network of thin channels, embedded in bedrock occurs, therefore the flow resistance coefficient changes its order of magnitude from one region $ \Omega_{1} $ to the other $ \Omega_{2} $. In this context the continuity of the solution $[\u, p]$ across the interface between $\Omega_{1}$ and $\Omega_{2}$ becomes a liability from the numerical point of view. Therefore, if it is possible to estimate a-priori, the magnitude of change that the solution will experience from one subdomain to the other (see \textsc{Example} \ref{Ex Multiscale Jump Example}), it is a more strategic approach to artificially introduce a discontinuity across the interface, and model it with a system, see \textsc{Equations} \eqref{Eq porous media strong decomposed}, satisfying a balance/coupling condition for both, normal flux and normal stress, see \textsc{Equations} \eqref{Eq interface balance conditions}. As mentioned above, these exchange conditions will be introduced weakly in the formulation allowing full decoupling of the underlying function spaces. Moreover, the trial spaces require that the pressure $ q $ is only square integrable $ L^{2} $ on one side of the interface, while it belongs to $ H^{1} $ on the other side of the interface (see \textsc{Figure}  \ref{Fig Approximate Solution Numerical Example} (a)); such discontinuity on the test spaces is ideal to handle discontinuities on the normal stress across the interface. The analogous takes place on the velocities modeling spaces, here the test functions $\v$ belong to $\Hdiv$ on one side of the interface while they are only square integrable $\mathbf{L}^{2}$ on the other (see \textsc{Figure}  \ref{Fig Approximate Solution Numerical Example} (b)). Again, this scenario will be ideal for discontinuities of normal flux across the interface. In summary, the primal-dual mixed formulation method will be able to capture interface discontinuities using uncoupled, conforming, finite dimensional spaces, presented in \textsc{Definitions} \ref{Def Global Finite Dimensional Spaces} and \ref{Def finite function spaces}.

We close this section introducing the general notation. In the present work vectors are denoted by boldface letters as are vector-valued functions and corresponding function spaces. The symbols $\grad$ and $\div$ represent the gradient and divergence operators respectively. The dimension is indicated by $N$ which will be equal to $2$ or $3$ depending on the context. Given a function $f: \R^{N}\rightarrow \R$ then  $\int_{\mathcal {M} } f\,dS$ denotes the integral on the $N-1$ dimensional manifold $\mathcal{M}\subseteq \R^{\! N}$. Analogously, $\int_{A} f\, d\x$ stands for the integral in the set $A\subseteq \R^{\! N}$; whenever the context is clear we simply write $\int_{A} f$. 
Given an open set $G$ of $\R^{N}$, the symbols $\Vert\cdot\Vert_{0,G}$, $\Vert\cdot\Vert_{1,G}$, $\Vert\cdot\Vert_{1/2, \partial G}$, $\Vert\cdot\Vert_{-1/2, \partial G}$ and $\Vert\cdot\Vert_{\Hdiv(G)}$ denote the $L^{2}(G)$, $H^{1}(G)$, $H^{1/2} (\partial G)$, $H^{-1/2} (\partial G)$ and $\Hdiv(G)$ norms respectively, while $\vert M\vert$ represents the Lebesgue measure of $G$ in $\R$, $\R^{2}$ or $\R^{3}$ depending on the context.
%
%
%
%
%
%
%
%
%
%
%
%
%
\section{Preliminaries}   
%
%
%
%
%
%
\subsection{Geometric Setting}
%
%
In this section we set the conditions on the domain of reference as well as its gridding.
\begin{definition}
Given a bounded open set $\omega$ in $\R^{2}$ we will say that a \textbf{bipartite map} is a finite collection of connected open subsets $\mathcal{G} = \{G_{n}: 1\leq n \leq N \}$ such that
\begin{enumerate}[(i)]
\item If $n\neq k$ then $G_{n}\cap G_{k} = \emptyset$.

\item The union satisfies $\Big\vert \omega - \bigcup\limits_{i\, = \, 1}^{N}  G_{n} \Big\vert = 0$ and $\cl (\omega )= \bigcup\limits_{i\, = \, 1}^{N} \cl (G_{n} )$.

\item The collection $\mathcal{G} =\big\{G_{n}: 1\leq n \leq N \big\}$ is partitioned in two subcollections  $\mathcal{G}_{1} =\big\{G^{\,1}_{i}: 1\leq i \leq I \big\}$ and $\mathcal{G}_{2} = \big\{G^{\,2}_{j}: 1\leq j \leq J \big\}$ such that 
\begin{enumerate}[a)]
\item $\big\{G_{n}: 1\leq n \leq N \big\} = 
\big\{G^{1}_{i}: 1\leq i \leq I \big\} \cup 
\big\{G^{2}_{j}: 1\leq j \leq J \big\}$.

\item If $i \neq k$ then $\big\vert\partial G^{1}_{i}\cap \partial G^{1}_{k}\big\vert = 0$.

\item If $j \neq \ell$ then $\big\vert\partial G^{\,2}_{j}\cap \partial G^{\,2}_{\ell}\ \big\vert = 0$.
\end{enumerate}
The collections $\mathcal{G}_{1}$ and $\mathcal{G}_{2}$ are said to be the \textbf{bipartition} or the \textbf{bi-coloring} of the map.
\end{enumerate}
%
\end{definition}
\begin{hypothesis}\label{Hyp Geometry of the Domain}
The domain of interest $\Omega$ is a polygonal, bounded, connected region of the plane and, it satisfies that 
\begin{enumerate}[(i)]
\item It has a bipartite map $\mathcal{G} = \{G_{n}: 1\leq n \leq N \}$ such that $G_{n}$ is a polygon for each $n = 1, \ldots, N$. 

\item If $\mapone, \maptwo$ is the bipartition of the map $\map$ then $\cl \big[\bigcup \{L:L\in \mapone\} \big]$ and $\cl \big[ \bigcup \{L:L\in \maptwo\} \big]$ are connected. 
\end{enumerate}
\end{hypothesis}
An example of bipartite map is depicted in \textsc{Figure} \ref{Fig Bipartite Map and Consistent Grid} (a), together with some other concepts introduced in the following definition.
\begin{definition}\label{Def geometric convention WVF}
Let $\Omega$ satisfy \textsc{Hypothesis} \ref{Hyp Geometry of the Domain} and let $\mathcal{G} = \{G_{n}: 1\leq n \leq N \}$ be its bipartite map with $\mathcal{G}_{1}$, $\mathcal{G}_{2}$ the map bipartition.
\begin{enumerate}[(i)]
\item
For each polygon $K\in \map$ denote by $\outer$ the outer normal vector to its boundary $\partial K$.

\item
For each polygon $K\in \map$ define $\n$ by
\begin{equation}\label{Def normal vector}
\n(\x) \defining\begin{cases}
\outer(\x) & K\in \mapone\; \text{and}\;\x \in \partial K  ,\\
-\outer(\x) & K \in \maptwo\; \text{and}\;\x \in \partial K\cap \Omega , \\
\outer(\x) & K \in \maptwo \, \text{and} \;\x \in \partial K\cap \partial \Omega .
\end{cases}
\end{equation}

\item
Define $\displaystyle \Omeone \defining \bigcup \{L:L\in \mapone\}$ and $\displaystyle\Ometwo \defining \bigcup \{M:M\in \maptwo\}$.

\item
Denote by $\displaystyle\Gamma\defining \bigcup\{\partial K: K\in \map\} - \partial \Omega$, the \textbf{interface} of the domain.

\end{enumerate}
\end{definition}
\begin{figure}[h] 
	\centering
	\begin{subfigure}[Bipartite Map $\map$ of region $\Omega$. ]
			{ \includegraphics[scale = 0.6]{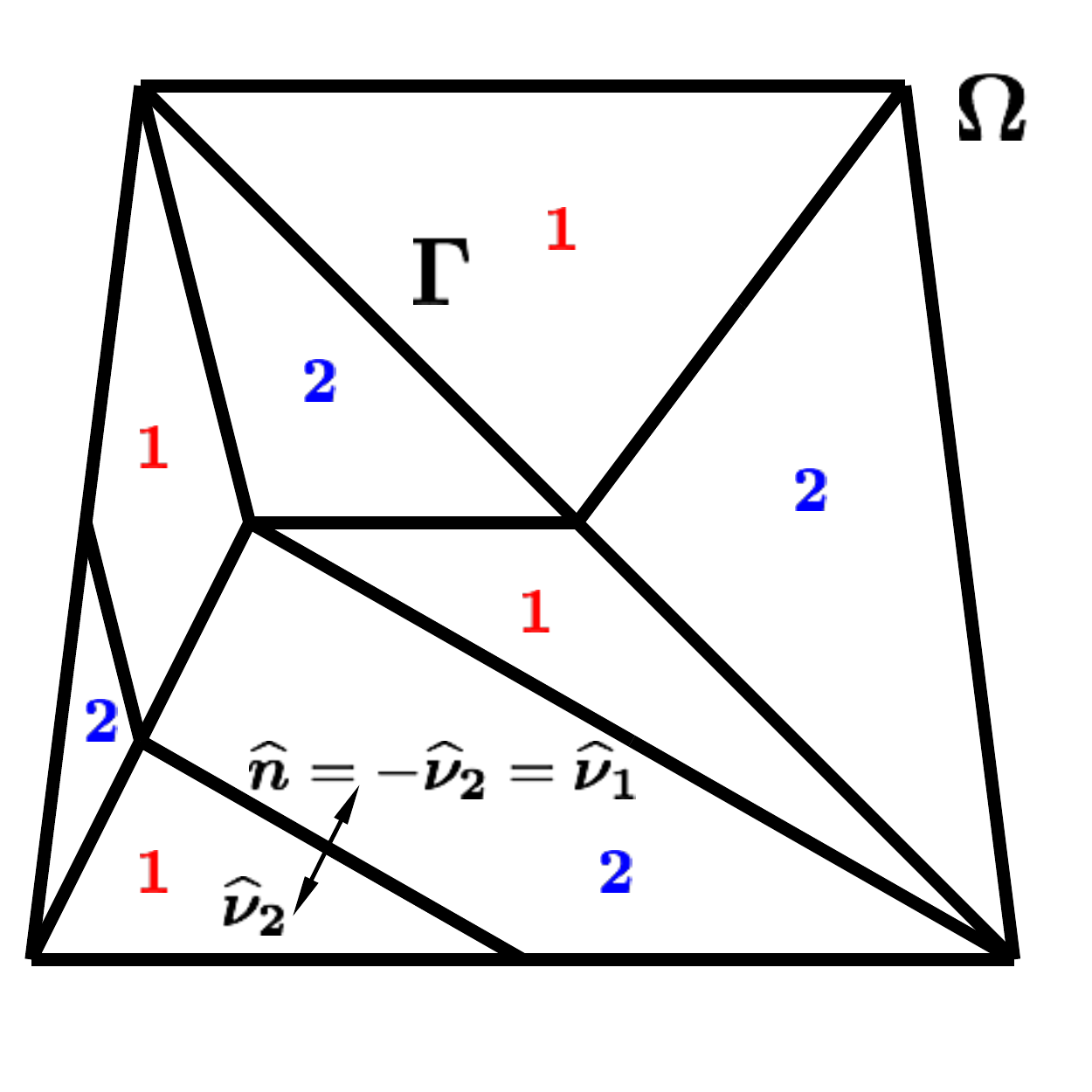}  }
	\end{subfigure} 
	~ 
	\begin{subfigure}[Grid $\triang$ consistent with $\map$.]
			{\includegraphics[scale = 0.6]{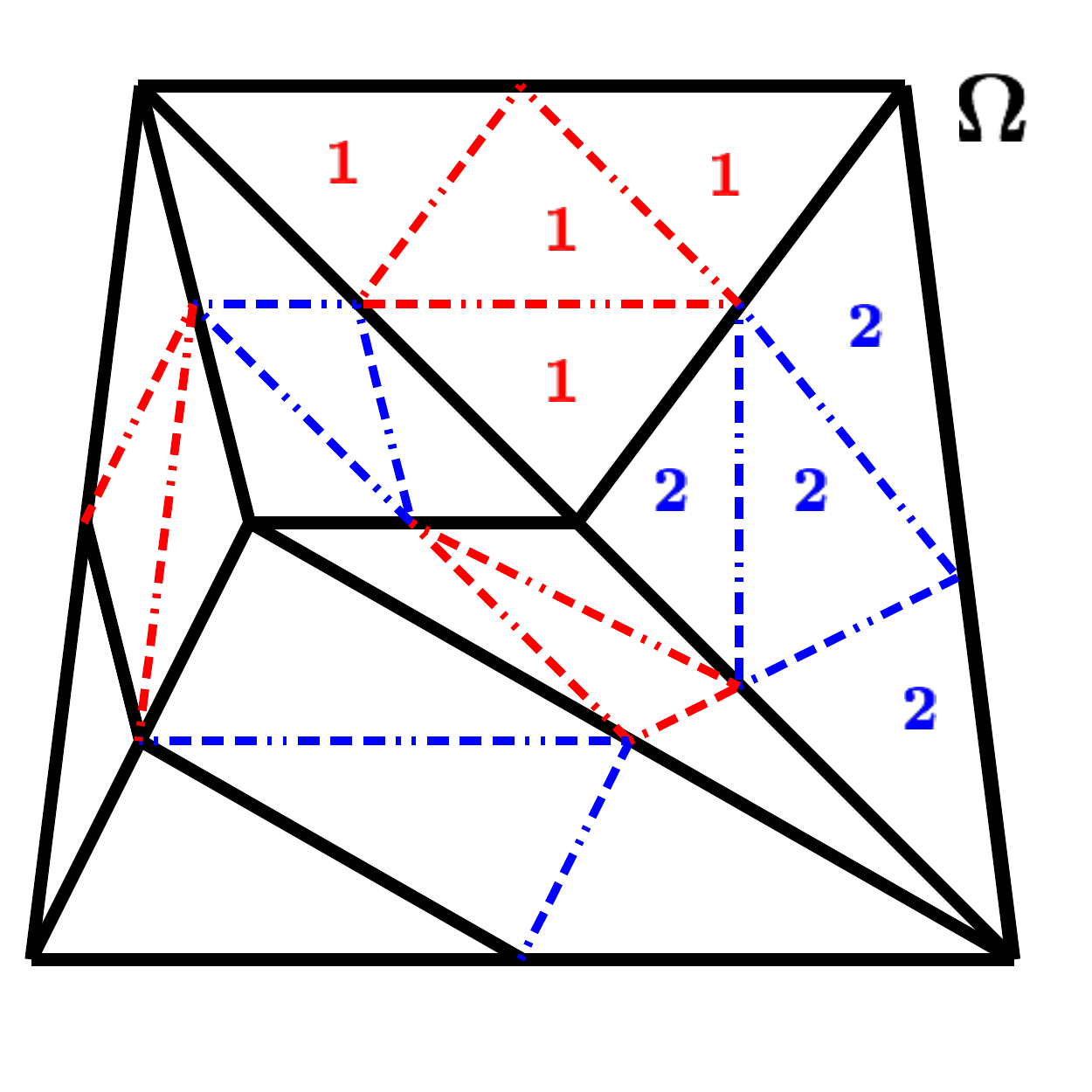} }                
	\end{subfigure} 
	%
	%
	%
	\caption{Figure (a) depicts a bipartite map $\map$ example for a given region $\Omega$. Subregions belonging to $\map_{i}$ have been labeled with $i$ for $i = 1,2$. The vector $\n$ and the outer normal vectors $\outer_{1}$, $\outer_{2}$ are illustrated for a couple of neighboring elements belonging to $\mapone$ and $\maptwo$ respectively. 
	Figure (b) depicts an example of grid $\triang$ consistent with the map $\map$. Some of the triangles belonging to $\tone$ and $\ttwo$ have been labeled with $1$ and $2$ respectively. \label{Fig Bipartite Map and Consistent Grid} }
\end{figure}
Next, we define the type of grids that will be considered in this work, see \textsc{Figure} \ref{Fig Bipartite Map and Consistent Grid} (b) for a simple example.
\begin{definition}\label{Def Consistent Grid}
Let $\Omega$ be as in \textsc{Definition} \ref{Def geometric convention WVF} above, then 
\begin{enumerate}[(i)]
   \item A triangulation $\triang$ of the domain $\Omega$ is said to be \textbf{consistent} with the map $\map = (\mapone, \maptwo)$ if for each triangle $K\in \triang$, it holds that $K\cap \Omeone = \emptyset$ or $K\cap \Ometwo = \emptyset$. Equivalently, $\ind_{K\cap \Omega_{1}}(\cdot) \ind_{K\cap \Omega_{2}}(\cdot) = 0$. 
   
   \item Given two triangulations $\triang'$ and $\triang$ of the domain $\Omega$, we say that $\triang '$ is a \textbf{refinement} of $\triang$, denoted by $\triang ' \leq \triang$, if for each element $K ' \in \triang '$ there exists a triangle $K\in \triang$ such that $K' \subseteq K$. 
   
   \item A sequence $\{\triang^{h}: h>0 \}$ is said to be \textbf{monotone} if $h' < h$ implies that $\triang^{h'}\leq \triang^{h}$.
\end{enumerate}
\end{definition}
%
%
%
%
%
%
\subsection{The Strong Problem and its Continuous Weak Formulation}
%
%
We begin this section recalling the general abstract setting to be used in this article. Let $\X$ and $\Y$ be Hilbert spaces and let $\A: \X\rightarrow \X $$'$, $\B: \X\rightarrow \Y $$'$ and $\C: \Y\rightarrow \Y $$'$ be continuous linear operators, we are to work on the following problem
\begin{equation}\label{Pblm operators abstrac system}
%
%
\begin{split}
\text{Find a pair}\; (\x, \y)\in \X\times \Y: \quad 
\A \x + \B '\y  = F_{1}\quad \text{in}\; \X ' , \\
%
- \B \x  + \C \y = F_{2} \quad \text{in}\; \Y ' ,
\end{split}
\end{equation}
where $F_{1}\in \X '$ and $F_{2} \in \Y '$. The following is a well-known result \cite{GiraultRaviartFEM}.
\begin{theorem}\label{Th well posedeness mixed formulation classic}
Assume that the linear operators $\A: \X\rightarrow \X'$, $\B: \X\rightarrow \Y '$, $\C: \Y\rightarrow \Y '$ are continuous and
\begin{enumerate}[(i)]
\item $\A$ is non-negative and $\X$-coercive on $\ker (\B)$. 

\item $\B$ satisfies the inf-sup condition 
\begin{equation}\label{Ineq general inf-sup condition}
   \inf_{\y \, \in \, \Y} \sup_{\x \, \in \, \X}
   \frac{\vert  \B\x(\y) \vert }{\Vert \x\Vert_{\X}\, \Vert \y \Vert_{\Y}}  >0 \, .
\end{equation}

\item $C$ is non-negative symmetric.
\end{enumerate}
Then for every $F_{1} \in \X '$ and $F_{2} \in \Y '$ the \textsc{Problem} \eqref{Pblm operators abstrac system} has a unique solution in $(\x, \y)\in \X \times \Y$; additionally it satisfies the estimate
\begin{equation} \label{mix-est}
\Vert\x\Vert_{\X} + \Vert \y\Vert_{\Y} \leq c\, (\Vert F_{1}\Vert_{\X '} + \Vert F_{2}\Vert_{\Y '}).
\end{equation}
%
\end{theorem}
Next, we present the strong problem to be approximated. Given a region $\Omega$ verifying \textsc{Hypothesis} \ref{Hyp Geometry of the Domain}, we introduce the following generalization of the Darcy flow \textsc{Problem} \eqref{Eq porous media strong}. 
\begin{multicols}{2}
\begin{subequations}\label{Eq porous media strong decomposed}
\begin{equation}\label{Eq Darcy Strong decomposed 0}
a(\cdot)\, \uone + \grad \pone 
+ \g 
= 0\,,
\end{equation}
%
%
\begin{equation}\label{Eq conservative strong decomposed 0}
\div \uone =  F
\,\quad \mathrm{in}\; \Omeone .
\end{equation}
%
%
\begin{equation}\label{Eq Drained Condition decomposed}
\pone = 0  \quad \mathrm{on}\; \partial \Omega_{1}\cap \partial \Omega  .
\end{equation}
\null \vfill
\columnbreak
\begin{equation}\label{Eq Darcy Strong decomposed 1}
a(\cdot) \, \utwo + \grad \ptwo 
+ \g 
= 0\,,
\end{equation}
\begin{equation}\label{Eq conservative strong decomposed 1}
\div \utwo = F\,\quad \mathrm{in}\; \Ometwo .
\end{equation}
\begin{equation}\label{Eq Non-Flux Condition decomposed}
\utwo\cdot\n = 0 \quad \mathrm{on}\; \partial \Omega_{2}\cap \partial \Omega  .
\end{equation}
\null \vfill
\end{subequations}
\end{multicols}
%
%
Endowed with the following interface exchange balance conditions
%
%
%
\begin{subequations}\label{Eq interface balance conditions}
\begin{equation}\label{Eq normal stress balance}
\ptwo - \pone = 
\stress\,, 
\end{equation}
\begin{equation}\label{Eq normal flux balance}
\uone \cdot\n - \utwo \cdot\n = 
\beta(\cdot)\, \ptwo + \flux   .
\end{equation}
\end{subequations}
%
%
%
\noindent The problem above, allows discontinuity jumps of discontinuity across the interface $\Gamma$, due to the forcing terms in the normal stress  \eqref{Eq normal stress balance} and normal flux balance conditions \eqref{Eq normal flux balance}, both relationships are nothing but statements normal stress and normal flux balance. The coefficients $a(\cdot)$, $\beta(\cdot)$ are nonnegative and they stand for the medium resistance to the fluid flow and the interface storage rate, respectively. The multiscaling of the coefficient $ a(\cdot) $will occur when modeling problems such as geological fissured systems (see \cite{Morales2}) where regions of high permeability have to be coupled with regions of low permeability. On the other hand, the coefficient $ \beta (\cdot) $ is meaningful in this context when one of the regions stores fluid and the other does not, due to the difference in the scaling of the problem between regions, its determination/measurement is an active research field, see \cite{BhunyaStorage} for an example of related work. Finally, recall that Darcy's law relates only pressure-velocity and that the pressure only acts in normal direction with respect to the physical object in contact, therefore the interface fluid exchange conditions \eqref{Eq interface balance conditions} can only be stated in the normal direction, while it is not possible to reconcile interface tangential velocity conditions with a Darcy system, see \cite{MoralesShow3} for an example.
\newline
\newline
In order to introduce the modeling spaces to be used in the weak variational formulation, first notice that $\big\{L: L\in \tone \big\}$, $\big\{M: M\in \ttwo \big\}$ are the simply connected components of $\Omega_{1}$ and $\Omega_{2}$ respectively. Then, 
\begin{align*} 
& \Hdiv(\Omega_{1}) = \bigoplus_{L \,\in \, \tone} \Hdiv(L)  , &
& H^{1}(\Omega_{2}) = \bigoplus_{M \,\in \, \ttwo} H^{1}(M) .
\end{align*}
The following space is introduced in order to couple adequately, the action of the pressure traces in the variational formulation
\begin{equation}\label{Eq Decoupling Trace Statement}
\begin{split}
E(\Omega_{2})  \defining & \big\{ q\in H^{1}(\Omega_{2}): 
q\ind_{\partial M \cap \partial L }\in H^{1/2}(\partial L ) \; \text{for all }\, 
(L, M) \in\mapone\times\maptwo\big\}\\
=  & \big\{ q\in H^{1}(\Omega_{2}): 
q\ind_{\Gamma }\in H^{1/2}(\Gamma ) \big\} .
\end{split} 
\end{equation}
We endow $E(\Omega_{2})$ with the $H^{1}(\Omega_{2})$ inner product. It is direct to see that $E(\Omega_{2})$ is a closed subspace of $H^{1}(\Omega_{2})$ and consequently a Hilbert space. Also define
\begin{equation}\label{Def Integrals on the trace}
\V(\Omega_{2}) \defining  \big\{ \v\in \L^{\! 2}(\Omega_{2}): \vtwo = \grad \qtwo\;
\text{for some}\; \qtwo\in E(\Omega_{2}) \big\} 
= \grad (E(\Omega_{2}) ) ,
\end{equation}
endowed with the $\L^{2}(\Omega_{2})$ inner product. Next we recall a necessary result.
\begin{lemma}\label{Th Completeness of V(Omega_2)}
Let $E(\Omega_{2})$ and $\V(\Omega_{2})$ be as defined in \eqref{Eq Decoupling Trace Statement},  \eqref{Def Integrals on the trace} respectively; define 
\begin{equation}\label{Def Space Generating with Isomorphism}
E_{0} (\Omega_{2})\defining \Big\{ \qtwo \in E(\Omega_{2}): \int_{\Omega_{2}}\qtwo = 0 \Big\} .
\end{equation}
Then,
\begin{enumerate}[(i)]
\item There exists a constant $C>0$ depending only on the domain $\Omega_{2}$ such that
\begin{align}\label{Ineq Control on Space Generating with Isomorphism}
& \Vert \rtwo \Vert_{1, \Omega_{2}} \leq C \, \Vert \grad \rtwo \Vert_{0, \Omega_{2}} , &
& \text{for all }\; \rtwo\in H\, .
\end{align}

\item The space $\V(\Omega_{2})$ is Hilbert.
\end{enumerate}

\end{lemma}
\begin{proof} 
See  \textsc{Lemma} 4.4 in \cite{MoralesNaranjo}.
\qed
\end{proof}
Now we are ready to introduce the functional setting of the problem, define
\begin{subequations}\label{Def function spaces}
\begin{equation}\label{Def spaces of velocities}
\X  \defining \Hdiv(\Omega_{1})\times E(\Omega_{2}).
\end{equation}
\begin{equation}\label{Def spaces of pressures}
\Y \defining \V(\Omega_{2})\times L^{2}(\Omega_{1}) .
\end{equation}
Endowed with their natural norms
\begin{equation}\label{Def norm space of velocities}
\big\Vert [\vone, \qtwo] \big\Vert_{\X}\defining
\big\{\Vert \vone \Vert_{\Hdiv(\Omega_{1})}^{2} 
+ \Vert \qtwo\Vert_{H^{1}(\Omega_{2})}^{2} \big\}^{\tfrac{1}{2}} , 
\end{equation}
\begin{equation}\label{Def norm space of pressures}
\big\Vert [\vtwo, \qone] \big\Vert_{\Y}\defining 
\big\{\Vert \vtwo\Vert_{\L^{2}(\Omega_{2})}^{2} 
+ \Vert  \qone \Vert_{ L^{2}(\Omega_{1})}^{2}\big\}^{\tfrac{1}{2}} .
\end{equation}
\end{subequations}
\begin{remark}\label{Rem notation of the duality product}
\begin{enumerate}[(i)]
\item 
Clearly $\X$ is a Hilbert space, in order to see that $\Y$ is a Hilbert space see \textsc{Lemma} \ref{Th Completeness of V(Omega_2)} above and/or \textsc{Lemma} 4.4 in \cite{MoralesNaranjo}.

\item 
In order to avoid heavy notation, from now on the following notational convention will be adopted
%
%
\begin{equation}\label{Eq duality notation full trace}
\int_{\Gamma} \left(\vone\cdot\n\right) \qtwo\, dS \defining 
\left\langle\vone\cdot\n, \, \qtwo\right\rangle_{\scriptscriptstyle H^{-1/2}(\Gamma), \, H^{1/2}(\Gamma)} \, .
\end{equation}
%
\end{enumerate}
\end{remark}
%
%
%
%
From now on, we assume that $F\in L^{2}(\Omega)$, $\g\in \L^{2}(\Omega_{2})$, $\stress\in H^{1/2}(\Gamma)$ and $\flux \in H^{-1/2}(\Gamma)$. Finally, the primal-dual mixed formulation for the \textsc{Problem} \eqref{Eq porous media strong decomposed} with interface balance conditions \eqref{Eq interface balance conditions} is given by 
\begin{subequations}\label{Pblm weak continuous solution}
%
\begin{multline}\label{Pblm weak continuous solution 1}
\text{Find}\;
\big([\uone, \ptwo],[\utwo, \pone]\big)\in \X\times \Y : \quad
\int_{\Omega_1}  a \, \uone \cdot \vone 
+ \int_{\Gamma}  \beta \, \ptwo \, \qtwo \, dS 
- \int_{\,\Gamma}\left(\uone\cdot\n\right) \qtwo\, d S \\
+\int_{\,\Gamma}\ptwo \left(\vone\cdot\n\right)\, d S 
- \int_{\Omega_1} \pone \,\div\vone\,
- \int_{\Omega_{2}} \utwo \cdot \grad \qtwo\, \\
= \int_{\Omega_{2}}F\, \qtwo - \int_{\Omega_{1}} \g \cdot \vone
+ \int_{\Gamma} \stress\, (\vone\cdot\n)\, dS 
- \int_{\Gamma} \flux\, \qtwo\, dS ,
\end{multline}
\begin{multline}\label{Pblm weak continuous solution 2}
\int_{\Omega_1}\div\uone\, \qone  
+ \int_{\Omega_2} \grad \ptwo \cdot\vtwo 
+ \int_{\Omega_2}  a \, \utwo \cdot \vtwo\,
= \int_{\Omega_{1}}F\, \qone 
- \int_{\Omega_{2}} \g \cdot \vtwo\, ,
\\
 \text{ for all } 
\big([\vone, \qtwo],[\vtwo, \qone]\big)\in \X\times \Y .
\end{multline}
\end{subequations}
Define the operators $\A: \X\rightarrow \X '$, $\B:\X\rightarrow \Y\,'$ and $\C:\Y\rightarrow \Y\,'$ by
\begin{subequations}\label{Def Operators Weak Continuous}
\begin{equation}\label{Def Regular Actions Operator }
\A [\vone, \qtwo],\big([\wone, \rtwo]\big)
\defining  
\int_{\Omega_1}  a \, \vone \cdot \wone 
+ \int_{\Gamma}  \beta \, \qtwo \, \rtwo \, dS 
- \int_{\,\Gamma}\left(\vone\cdot\n\right) \rtwo\, d S
+\int_{\,\Gamma}\qtwo \left(\wone\cdot\n\right)\, d S ,
\end{equation}
\begin{equation}\label{Def Mixed Operator Continuous}
\B[\vone,\,\qtwo], \big([\wtwo,\,\rone]\big)\defining 
 \int_{\Omega_1}  \div\vone \, \rone
+ \int_{\Omega_{2}} \grad\qtwo \cdot \wtwo ,\,
\end{equation}
\begin{equation}\label{Def Non Regular Actions non-negative Operator}
\C [\vtwo,\,\qone] \big([\wtwo,\,\rone]\big)\defining  \int_{\Omega_2} a \, \vtwo \cdot \wtwo .
\end{equation}
\end{subequations}
Hence, the \textsc{Problem} \eqref{Pblm weak continuous solution} is equivalent to
\begin{equation}\label{Pblm operators weak continuous solution}
\begin{split}
\text{Find a pair}\; 
\big([\uone,\,\ptwo], [\utwo,\,\pone]\big)\in \X\times \Y: \quad 
\A[\uone,\,\ptwo] + \B ' [\utwo,\,\pone]  = F_{1}\quad \text{in}\; \X ' ,\\
- \B [\uone,\,\ptwo]  + \C [\utwo,\,\pone] = F_{2} \quad \text{in}\; \Y ' .
\end{split}
\end{equation}
Here $F_{1}\in \X '$ and $F_{2} \in \Y '$ are the functionals defined by the right hand side of \eqref{Pblm weak continuous solution 1} and \eqref{Pblm weak continuous solution 2} respectively. 
In order to satisfy the required ellipticity conditions for the operator $\A$, some extra hypotheses on the coefficients become necessary.
\begin{hypothesis}\label{Hyp non null local storage coefficient}
   It will be assumed that coefficients of storage exchange $\beta: \Gamma\rightarrow [0, \infty)$ and porous medium resistance $a: \Omega\rightarrow (0,\infty)$, satisfy that $\beta\in L^{\infty}(\Gamma)$, $\Vert \beta \, \ind_{\Gamma} \Vert_{L^{1} (\Gamma) }  > 0$ and $a\in L^{\infty}(\Omega)$, $\big\Vert \dfrac{1}{a}\big\Vert_{L^{\infty}(\Omega)}>0 $ respectively.
\end{hypothesis}
\begin{theorem}\label{Th Well Posedness Problem}
Let $\Omega$ be a polygonal region and let $\map$ be a bipartite map, then if the \textsc{Hypothesis \ref{Hyp non null local storage coefficient}} is satisfied, the \textsc{Problem} \eqref{Pblm weak continuous solution} is well-posed.
\end{theorem}
\begin{proof}
See \textsc{Theorem 4.8} in \cite{MoralesNaranjo}.
\qed
\end{proof}
We close this section recalling the next result on recovering the strong problem from the weak variational formulation \eqref{Pblm weak continuous solution}. 
\begin{theorem}\label{Th weak formulation of continuos version}
The solution of the weak variational \textsc{Problem} \eqref{Pblm weak continuous solution} is a strong solution of the \textsc{Problem} \eqref{Eq porous media strong decomposed} with the forcing gravitation term $\g$ in the \textsc{Equation} \eqref{Eq Darcy Strong decomposed 0} replaced by $P\g$, which denotes its orthogonal projection onto the space $\V(\Omega_{2})$. In particular if $\g\ind_{\Omega_{2}} \in \V(\Omega_{2})$ the weak solution is exactly the strong solution.
\end{theorem}
\begin{proof} See \textsc{Theorem 4.9} in \cite{MoralesNaranjo}.
   \qed
\end{proof}
%
%
%
%
%
%
%
%
%
%
%
%
%
%
%
%
%
\section{The Discretization of the Problem }
%
%
%
%
In this section we present a viable discretization of the Problem \eqref{Pblm weak continuous solution} in the two dimensional case, from the theoretical point of view. We start introducing the discrete function spaces, we will denote by $P_{\ell}(K)$ the polynomials of order $\ell$ on the triangle $K$ and $\mathbf{P}_{\ell}(K) = (P_{\ell}(K))^{2}$. As usual, $\mathbf{RT}_{\ell}(K)$ indicates the Raviart-Thomas finite element of degree $\ell$ on the triangle $K$. From now on it will be assumed that the domain $\Omega$ satisfies \textsc{Hypothesis} \ref{Hyp Geometry of the Domain} and that any triangulation $\triang$ of analysis is consistent with the map $\map$, as introduced in \textsc{Definition} \ref{Def Consistent Grid}.  Hence, for a fixed consistent triangulation $\triang$ with size $h \defining \max \{\text{diameter}(K): T\in \triang \}$ we denote 
\begin{subequations}\label{Def Global Finite Dimensional Spaces}
\begin{equation}\label{Def Raviart-Thomas Global} 
\mathbf{RT}_{0}(\Omega_{1}, \triang)  \defining \big\{\v\in \Hdiv(\Omega_{1}): 
\v\vert_{K}\in   \mathbf{RT}_{0}(K) , \; \; \text{for all } \; K\in \triang , \; K\subseteq \Omega_{1} \big\}\, ,
\end{equation}
\begin{multline}\label{Def Piecewise Linear Pressure Global} 
\mathcal{Q}(\Omega_{ 2 }, \triang)  \defining \big\{\qtwo \in H^{1}(\Omega_{ 2 }): 
\qtwo = \xi\vert_{\Omega_{2}} \; \; \text{for some }\; \xi\in H^{1}(\Omega)\; \; \text{and }\\
\qtwo\vert_{K}\in   P_{1}(K) , \; \; \text{for all } \; \, K\in \triang , \; K\subseteq \Omega_{2} \big\} \, , 
\end{multline}
\begin{equation}\label{Def Piecewise Constant Velocity Global} 
\grad \mathcal{Q}(\Omega_{2 }, \triang) \defining 
\big\{\vtwo\in \mathbf{P}_{0}(\Omega_{2 }): 
\vtwo = \grad \qtwo \; \;  \text{for some } \; \qtwo\in \mathcal{Q}(\Omega_{2}, \triang) \big\} \, , 
\end{equation}
\begin{align}\label{Def Piecewise Constant Pressure Global} 
\mathcal{Q}(\Omega_{ 1 }, \triang) & \defining 
\big\{\qone \in L^{2}(\Omega_{ 1}): 
\qone \vert_{K}\in   P_{0}(K) , \; \; \text{for all } \; \, K\in \triang , \; K\subseteq \Omega_{1} \big\} \,  .
\end{align}
\end{subequations}
Whenever the triangulation $\triang$ is clear from the context, we simply write $\mathbf{RT}_{0}(\Omega_{1}) = \mathbf{RT}_{0}(\Omega_{1}, \triang)$, $\grad \mathcal{Q}(\Omega_{2 })  = \grad \mathcal{Q}(\Omega_{2 }, \triang)$ and $\mathcal{Q}(\Omega_{\ell})  = \mathcal{Q}(\Omega_{\ell}, \triang)$ for $\ell = 1, 2$. Notice that $\mathbf{RT}_{0}(\Omega_{1})  \times \mathcal{Q}(\Omega_{2}) \subseteq \X$ and $\grad \mathcal{Q}(\Omega_{2 })\times \mathcal{Q}(\Omega_{1} )\subseteq \Y$. Define the following discrete spaces
\begin{subequations}\label{Def finite function spaces}
\begin{equation}\label{Def finite spaces of velocities}
\X_{h}  \defining \mathbf{RT}_{0}(\Omega_{1})  \times \mathcal{Q}(\Omega_{2}) ,
\end{equation}
\begin{equation}\label{Def finite spaces of pressures}
\Y_{h} \defining \grad \mathcal{Q}(\Omega_{2})  \times \mathcal{Q}(\Omega_{1}) ,
\end{equation}
\end{subequations}
endowed $\X_{h}$, $\Y_{h}$ with the norms $\Vert \cdot \Vert_{\X}$ and $\Vert \cdot \Vert_{\Y}$ respectively. The discrete operators $\A_{h}:\X_{h}\rightarrow \X_{h}'$, $\B_{h}:\X_{h}\rightarrow \Y_{h}'$ and $\C_{h}:\Y_{h}\rightarrow \Y_{h}'$ are defined by the respective restriction of the operators $\A$, $\B$ and $\C$ introduced in \eqref{Def Regular Actions Operator }, \eqref{Def Mixed Operator Continuous} and \eqref{Def Non Regular Actions non-negative Operator} i.e., 
\begin{subequations}\label{Def FEM operators}
\begin{align}\label{Def FEM Regular Actions Operator }
& 
\A_{h}[\vone, \qtwo]([\wone, \rtwo]) \defining \A[\vone, \qtwo]([\wone, \rtwo]) ,
& 
& \text{for all }\, 
[\vone, \qtwo], [\wone, \rtwo] \in \X_{h}.
\end{align}
\begin{align}\label{Def FEM Mixed Operator Continuous}
& 
\B_{h}[\vone,\,\qtwo], \big([\wtwo,\,\rone]\big)\defining 
\B[\vone,\,\qtwo] \big([\wtwo,\,\rone]\big), &
&\text{for all }\, 
[\vone, \qtwo]\in \X_{h}, [\wtwo,\,\rone]\in \Y_{h}.
\end{align}
\begin{align}\label{Def FEM Non Regular Actions non-negative Operator}
& \C_{h} [\vtwo,\,\qone] \big([\wtwo,\,\rone]\big)\defining  
\C [\vtwo,\,\qone] \big([\wtwo,\,\rone]\big) , &
& \text{for all }\, 
[\vtwo, \qone], [\wtwo,\,\rone]\in \Y_{h}.
\end{align}
\end{subequations}
The discretization of \textsc{Problem} \eqref{Pblm operators weak continuous solution} is given by 
\begin{equation}\label{Pblm FEM operators weak continuous solution}
\begin{split}
\text{Find a pair}\; 
\big([\uoneh,\,\ptwoh], [\utwoh,\,\poneh]\big)\in \X_{h}\times \Y_{h}: \quad 
\A[\uoneh,\,\ptwoh] + \B ' [\utwoh,\,\poneh]  = F_{1}\quad \text{in}\; \X_{h} ' ,\\
- \B [\uoneh,\,\ptwoh]  + \C [\utwoh,\,\poneh] = F_{2} \quad \text{in}\; \Y_{h} ' .
\end{split}
\end{equation}
where $F_{1}\in \X_{h} '$ and $F_{2} \in \Y_{h} '$ are known functionals. We are to prove that the \textsc{Problem} \eqref{Pblm FEM operators weak continuous solution} above is well-posed, verifying that the operators $\A_{h}$, $\B_{h}$ and $\C_{h}$ satisfy the hypotheses of \textsc{Theorem} \ref{Th well posedeness mixed formulation classic}.
Before proving the inf-sup condition of the operator $\B_{h}$ we recall a well-known result
\begin{theorem}\label{Th Local Solution of the Dual Formulation}
Let $\mathbf{RT}_{0}(\Omega_{1}, \triang)$, $\mathcal{Q}(\Omega_{1}, \triang)$ be as defined in \eqref{Def Raviart-Thomas Global} and \eqref{Def Piecewise Constant Pressure Global} respectively. Then, for every $\qone \in \mathcal{Q}(\Omega_{1}, \triang)$ there exist $\vone\in \mathbf{RT}_{0}(\Omega_{1}, \triang)$ and a constant $C>0$ depending only on $\Omega_{1}$ such that $\div \vone = \qone$ and $\Vert \vone \Vert_{\Hdiv(\Omega_{1})} \leq C \Vert \qone \Vert_{L^{2}(\Omega_{1})}$.
\end{theorem}
\begin{proof}
See \textsc{Lemma} 5.4, Chapter III, pg 151 in \cite{BraessFEM}.
\qed
\end{proof}
\begin{lemma}\label{Th FEM inf sup condition}
The operator $\B_{h}: \X_{h} \rightarrow \Y_{h} '$ defined in \textsc{Equation} \eqref{Def FEM Mixed Operator Continuous} is continuous and satisfies the $\inf$-$\sup$ condition i.e., there exists a constant $C>0$ depending only on the map $\map$ such that for every $[\wtwo, \rone]\in\Y_{h} $ there exists $[\vone, \qtwo]\in\X_{h}$ satisfying
\begin{equation}\label{Ineq FEM existence of beta}
 \B_{h}\,[\vone, \qtwo] ([\wtwo, \rone])\geq 
C\, \big\Vert [\vone, \qtwo]\big\Vert_{\X_{h}} \big\Vert [\wtwo, \rone] \big\Vert_{\Y_{h}}
\, .
\end{equation}
Moreover, the constant $C>0$ is independent from $[\wtwo, \rone]$ and the triangulation $\triang$.
\end{lemma}
\begin{proof}
The continuity of $\B_{h}$ follows from the continuity of $\B$. Now fix $[\wtwo, \rone]\;\in \Y_{h}$, due to \textsc{Theorem} \ref{Th Local Solution of the Dual Formulation} there exists $\vone\in \mathbf{RT}_{0}(\Omega_{1})$ such that $\div \vone = \rone$ and $\Vert \vone \Vert_{\Hdiv(\Omega_{1})} \leq C\, \Vert \rone\Vert_{L^{2}(\Omega_{1})}$, with $C>0$ depending only on the domain $\Omega_{1}$.

Next, by definition of $\grad \mathcal{Q}(\Omega_{2})$ there must exist $\eta\in \mathcal{Q}(\Omega_{2})$ such that $\grad \eta = \wtwo$. Define $\qtwo \defining \eta - \frac{1}{\vert \Omega_{2}\vert} \int_{\Omega_{2}}\eta$, clearly $\qtwo \in E_{0}(\Omega_{2}) \cap \mathcal{Q}(\Omega_{2})$ and due to the \textsc{Inequality} \eqref{Ineq Control on Space Generating with Isomorphism}, it holds that $\Vert \qtwo \Vert_{1, \Omega_{2}} \leq C \Vert \wtwo \Vert_{0, \Omega_{2}}$ with $C>0$ depending only on the domain $\Omega_{2}$.

Then, the pair $[\vone, \qtwo]$ belongs to $\X_{h}$ and satisfies $\big\Vert [\vone, \qtwo] \Vert_{\X} \leq C  \, \big\Vert [\wtwo, \rone] \big\Vert_{\Y} $ with $C>0$ adequate depending only on the domain $\Omega$. Therefore,
\begin{equation*}
   \B_{h} [\vone, \qtwo]\big([\wtwo, \rone]\big) = \big\Vert [\wtwo, \rone] \big\Vert_{\Y}^{2} 
   \geq  \frac{1}{C}\,  
   \big\Vert [\vone, \qtwo] \big\Vert_{\X} \, 
   \big\Vert [\wtwo, \rone] \big\Vert_{\Y} \,  .
\end{equation*}
This completes the proof.
\qed
\end{proof}
\begin{lemma}\label{Th FEM coercivity of A on ker B}
If \textsc{Hypothesis \ref{Hyp non null local storage coefficient}} is satisfied then, the operator $\A_{h}:\X_{h}\rightarrow \X_{h} '$ defined by \eqref{Def FEM Regular Actions Operator } is continuous and $\X_{h}$-coercive on $\X_{h}\cap \ker (\B_{h})$ i.e., 
\begin{equation}\label{Ineq FEM charaterization of alpha}
\A_{h}[\vone, \qtwo] \big([\vone, \qtwo]\big)\geq  C \,
\big\Vert[\vone, \qtwo] \big\Vert_{\X_{h}}^{2} \,,\quad 
\text{for all}\; [\vone, \qtwo]\in \X_{h}\cap \ker (\B_{h}) .
\end{equation}
Where $C>0$ is an adequate constant depending only on the domain $\Omega$.
\end{lemma}
\begin{proof} The continuity of the operator $\A_{h}$ follows from the continuity of the operator $\A$. For the coerciveness of the operator, let $[\vone, \qtwo]\in \X_{h}\cap \ker(\B_{h})$ then 
   \begin{equation}\label{Def FEM testing the kernel of B}
      \B_{h}[\vone, \qtwo]\big([\wtwo, \rone]\big) = 0 \quad \text{for all}\; [\wtwo, \rone]\in \Y_{h} .
   \end{equation}
Notice that $\div\vone\vert_{L}$ is constant for each $L\in \triang$ contained in $\Omega_{1}$ and that $\grad \qtwo$ belongs to $\grad \mathcal{Q}(\Omega_{2})$ by definition. Therefore, $[\grad \qtwo, \div\vone] \in \Y_{h}$, in particular, testing \eqref{Def FEM testing the kernel of B} with $[\bm{0}, \rone]\in \Y_{h}$ we conclude that $\div \vone = 0$ since $\rone$ is an arbitrary element in $L^{2}(\Omega_{1})$. On the other hand, clearly $\grad \qtwo \in \grad \mathcal{Q}(\Omega_{2})$ and the pair $[\grad \qtwo, 0]\in \Y$ is eligible for testing \eqref{Def FEM testing the kernel of B}; which yields $\grad\qtwo = \bm{0}$, i.e. $\qtwo$ is constant inside $\Omega_{2}$. Hence 
\begin{equation*}
   \int_{\Gamma} \beta \, \qtwo^{2} = 
   \frac{\Vert \beta \, \ind_{\partial \Gamma}\Vert_{L^{1}(\Gamma)}}{\vert  \Omega_{2} \vert} \, 
   \Vert \qtwo \Vert_{0, \Omega_{2}}^{2} 
   = \frac{\Vert \beta \, \ind_{\Gamma}\Vert_{L^{1}(\Gamma)}}{\vert  \Omega_{2} \vert} \, 
   \Vert \qtwo \ind_{\Omega_{2}} \Vert_{1, \Omega_{2}}^{2} .
\end{equation*}
Using the previous observations we get that
\begin{equation*}
\begin{split}
\A_{h}[\vone, \qtwo] \big([\vone, \qtwo]\big) & = 
\int_{\Omega_1}  a\vone \cdot \vone 
+ \int_{\Gamma}  \beta \, \qtwo^{2}  \, dS \\
& \geq
\Big\Vert\frac{1}{a} \Big\Vert^{-1}_{L^{\infty}(\Omega)}\Vert \vone \Vert^{2}_{\Hdiv(\Omega_{1} ) } 
+ \frac{\Vert \beta \, \ind_{\Gamma}\Vert_{L^{1}(\Gamma)}}{\vert  \Omega_{2} \vert} \, 
   \Vert \qtwo  \Vert_{1, \Omega_{2} }^{2} \\
 &  \geq C \, \big\Vert [\vone, \qtwo]\big\Vert_{\X}^{2} \, ,
   \end{split}
\end{equation*}
where $C = \min \big\{\big\Vert\dfrac{1}{a} \big\Vert^{-1}_{L^{\infty}(\Omega)}, \, 
\vert  \Omega_{2} \vert^{-1}\, \Vert \beta \, \ind_{\Gamma}\Vert_{L^{1}(\Gamma)} \,  \big\}$. This completes the proof. 
\qed
\end{proof}
\begin{theorem}\label{Th FEM Well Posedness Problem}
Let $\Omega$ be a polygonal region and let $\triang$ be a triangulation, then if the \textsc{Hypothesis \ref{Hyp non null local storage coefficient}} is satisfied, the \textsc{Problem} \eqref{Pblm FEM operators weak continuous solution} is well-posed.
\end{theorem}
\begin{proof}
It is direct to see that the operator $\C_{h}$ is non-negative and symmetric. Due to this fact, \textsc{Lemma} \ref{Th FEM inf sup condition} and \textsc{Lemma} \ref{Th FEM coercivity of A on ker B}, the hypotheses of \textsc{Theorem} \ref{Th well posedeness mixed formulation classic} are satisfied and the result follows. 
\qed
\end{proof}
%
%
%
%
%
%
%
%
%
\subsection{Strong Convergence}
%
%
In this section we prove rigorously, under mild hypotheses on a sequence of triangulations $\{\triang^{h}: h > 0\}$, the strong convergence of discrete solutions to the continuous one i.e., $\big(\uh , \ph \big) \rightarrow \big(\u, p\big)$, when $h\rightarrow 0$. In order to attain a-priori estimates some previous results are necessary.
\begin{proposition}\label{Th L^2 Norm Control}
Let $\Omega$ be a domain satisfying \textsc{Hypothesis \ref{Hyp Geometry of the Domain}} then, there exists $C > 0$ depending only on the map $\map = (\mapone, \maptwo) $ such that
\begin{align}\label{Ineq Control With Gradient and Trace}
%
&\Vert \xi \Vert^{2}_{H^{1}(\Omega_{i})}
\leq   C^{\,2}\Big(\Vert\grad  \xi \Vert^{2}_{L^{2}(\Omega_{i})} + 
\Vert \sqrt{\beta} \, \xi \Vert^{2}_{L^{2}(\Gamma)} 
\Big) , & 
& \text{for all }\, \xi \in H^{1}(\Omega_{i}) \; \text{and }\, i = 1,2.
%
\end{align}
\end{proposition}
\begin{proof}
Notice that the map $\displaystyle \xi \mapsto \Big(\Vert\grad  \xi \Vert^{2}_{L^{2}(\Omega_{i})} + 
\Vert \sqrt{\beta}  \, \xi \Vert^{2}_{L^{2}(\Gamma)} 
\Big)^{1/2}$ is a norm, for if it is equal to zero it follows that $\xi$ is constant, therefore 
\begin{equation*}
0 = \int_{\Gamma} \beta \, \vert \xi\vert^{2} =\vert \xi\vert^{2}\, \Vert \beta \Vert_{L^{1}(\Gamma)}.
\end{equation*} 
Due to the \textsc{Hypothesis \ref{Hyp non null local storage coefficient}}, this implies that $\xi = 0$. From here, a standard application of the Rellich-Kondrachov Theorem delivers the result.
\qed 
\end{proof}
\begin{proposition}
   Let $\triang^{h}$ be a consistent triangulation of $\Omega$ and let $\big([\uoneh, \ptwoh],[\utwoh, \poneh]\big)\in \X_{h}\times \Y_{h}$ be the solution of \textsc{Problem} \eqref{Pblm FEM operators weak continuous solution}, then there exists $C > 0$ depending only on the domain $\Omega$ such that
   \begin{equation}\label{Ineq Estimate of Pressure One}
      \Vert \poneh\Vert_{0, \Omega_{1}}
      \leq C\, \big(\Vert \uoneh \Vert_{0, \Omega_{1}}^{2}
      + \Vert \ptwoh \Vert_{1, \Omega_{2}}^{2}
      + \Vert \g \Vert_{1, \Omega_{2}}^{2}
      + \Vert \stress \Vert_{1, \Omega_{2}}^{2}
      \big)^{1/2} .
   \end{equation}
\end{proposition}
\begin{proof}
Test \textsc{Problem} \eqref{Pblm FEM operators weak continuous solution} with $\big([\vone, 0],[\0, 0]\big)\in \X_{h} \times \Y_{h}$ and add both equations, this gives
\begin{equation}\label{Eq Test for Isolating pone}
\int_{\Omega_1} \poneh \,\div\vone\,
=
-\int_{\Omega_1}  a \, \uoneh \cdot \vone 
-\int_{\,\Gamma}\ptwoh \left(\vone\cdot\n\right)\, d S 
+ \int_{\Omega_{1}} \g \cdot \vone
- \int_{\Gamma} \stress\, (\vone\cdot\n)\, dS  .
\end{equation}
Applying the CBS inequality to each summand we get
\begin{equation*}
\begin{split}
\Big\vert\int_{\Omega_1} \pone \,\div\vone\,\Big\vert
& \leq
C \, \big(\Vert \uoneh \Vert_{0, \Omega_{1}} \Vert \vone \Vert_{0, \Omega_{1}} 
+ \Vert \ptwoh \Vert_{1/2, \Gamma} \, \Vert \vone \Vert_{-1/2, \Gamma}  \\
& + \Vert \g \Vert_{0, \Omega_{1}} \Vert \vone \Vert_{0, \Omega_{1}}  
+ \Vert \stress \Vert_{1/2, \Gamma} \, \Vert \vone \Vert_{-1/2, \Gamma}  
\big)
\\
& \leq 
C \, \big(\Vert \uoneh \Vert_{0, \Omega_{1}} 
+ \Vert \ptwoh \Vert_{1, \Omega_{2}} 
+ \Vert \g \Vert_{0, \Omega_{1}} 
+ \Vert \stress \Vert_{1/2, \Gamma} 
\big) \Vert \vone \Vert_{\Hdiv(\Omega_{1})} \\
& \leq 2 \,  C\, \big(\Vert \uoneh \Vert_{0, \Omega_{1}}^{2}
      + \Vert \ptwoh \Vert_{1, \Omega_{2}}^{2}
      + \Vert \g \Vert_{1, \Omega_{2}}^{2}
      + \Vert \stress \Vert_{1/2, \Gamma}^{2}
      \big)^{\tfrac{1}{2}} \, \Vert \vone \Vert_{\Hdiv(\Omega_{1})} .  
\end{split}
\end{equation*}
Here the generic constant of the second line is large enough. 
Due to \textsc{Theorem \ref{Th Local Solution of the Dual Formulation}} there exists $\vone\in \mathbf{RT}_{0}(\Omega_{1})$ such that $\div \vone = \poneh$ and $\Vert \vone \Vert_{\Hdiv(\Omega_{1})} \leq C\, \Vert \poneh \Vert_{0, \Omega_{1}}$, where the generic bound $C > 0$, depends only on the domain $\Omega_{1}$. Testing the expression above with this function, the \textsc{Inequality} \eqref{Ineq Estimate of Pressure One} follows.
   \qed
\end{proof}
Now we are ready to present an a-priori estimate.
\begin{theorem}\label{Th a-priori Estimates}
Let $\{\triang^{h}: h>0\}$ be a monotone sequence of consistent triangulations of $\Omega$. Denote by $\big([\uoneh, \ptwoh],[\utwoh, \poneh]\big)\in \X_{h}\times \Y_{h}$ the solution of \textsc{Problem} \eqref{Pblm FEM operators weak continuous solution} associated to the triangulation $\triang^{h}$ with the fixed forcing terms $F$, $\g$, $\stress$, $\flux$. Then, there exists $C>0$ such that 
\begin{align}\label{Ineq Uniform Boundedness}
& \Vert[\uoneh, \ptwoh]\Vert_{\X}+\Vert [\utwoh, \poneh]\Vert_{\Y} \leq C \, ,&
& \text{for all }\, h > 0 .
\end{align}
\end{theorem}
\begin{proof}
Test \textsc{Problem} \eqref{Pblm FEM operators weak continuous solution} with $\big(\uh, \ph\big)$ and add both equations, this gives
\begin{equation}\label{Eq Test on the FEM diagonal}
\int_{\Omega_1}  a \, \vert \uoneh \vert^{2} 
+ \int_{\Omega_2}  a \, \vert \utwoh \vert^{2} 
+ \int_{\Gamma}  \beta \, \vert \ptwoh \vert^{2} \, dS 
= \int_{\Omega }F\, \ph - \int_{\Omega } \g \cdot \uh
+ \int_{\Gamma} \stress\, (\uoneh\cdot\n)\, dS 
- \int_{\Gamma} \flux\, \ptwoh\, dS .
\end{equation}
On the right hand side term, we apply first the usual duality bounds and next the CBS inequality for vectors in $\R^{4}$, this gives
\begin{equation}\label{Ineq First a-priori Estimate}
\begin{split}
C_{0} \Big[ \Vert \uoneh \Vert^{2}_{0,\Omeone}  & 
+ \Vert \utwoh \Vert^{2}_{0, \Ometwo}  
+ \Vert \sqrt{\beta} \, \ptwoh \Vert^{2} _{0, \Gamma} \Big] \\
\leq \Vert F\Vert_{ 0, \Omega }\, & \Vert \ph\Vert_{ 0, \Omega } 
+ \Vert \g\Vert_{L^{2}(\Omega)}\, \Vert \uh\Vert_{0, \Omega } 
%
+ \Vert \stress \Vert_{1/2, \, \Gamma}\, \Vert \uoneh\Vert_{\Hdiv(\Omeone)} 
+ \Vert \flux \Vert_{ 0, \Gamma }\, \Vert \ptwoh\Vert_{ 0, \Gamma }  \\
\leq \sqrt{2} \, \Big[ \Vert F &  \Vert_{0, \Omega}^{2}
+ \Vert \g\Vert_{0, \Omega}^{2}
+ \Vert \stress \Vert_{1/2, \, \Gamma}^{2}
+ \Vert \flux \Vert_{0, \Gamma}^{2}\Big]^{1/2} \\
& \Big[\Vert \ph \Vert_{0, \Omega}^{2} 
+ \Vert \ptwoh\Vert_{0, \Gamma}^{2}
+ \Vert \uoneh\Vert_{\Hdiv(\Omeone)}^{2} 
+ \Vert \utwoh\Vert_{0, \Ometwo}^{2}\Big]^{1/2} .
\end{split}
\end{equation}
In the expression above, the constant $\sqrt{2}$ appears due to the estimate $\Vert \uoneh\Vert_{0, \Omeone}^{2}  + \Vert \uoneh\Vert_{\Hdiv(\Omeone)}^{2} \leq 2 \,  \Vert \uoneh\Vert_{\Hdiv(\Omeone)}^{2} $.
Next, we focus on giving estimates to the second factor of the right hand side. In order to bound the pressure, first split it in two pieces $\Vert \ph\Vert_{0, \Omega}^{2}  = \Vert \poneh\Vert_{0, \Omega_{1} }^{2}  + \Vert \ptwoh\Vert_{0, \Omega_{2} }^{2} $, now due to \textsc{Proposition} \ref{Th L^2 Norm Control}, there exists $C > 0$ depending only on the map $\map$ such that
\begin{equation}\label{Ineq Estimate of Pressure Two}
\begin{split} 
\frac{1}{C} \, \Vert \poneh \Vert_{0, \Omega_{1}}^{2} 
& \leq 
\Vert \grad \poneh \Vert_{0, \Omega_{1}}^{2}
+ \Vert \sqrt{\beta} \, \poneh\Vert_{0, \Gamma}^{2} \\
& \leq
2 \, \Vert \uoneh \Vert_{0, \Omega_{1}}^{2} 
+ 2 \, \Vert \g^{h}\Vert_{0, \Omega_{1}}^{2} 
+ 2 \, \Vert \sqrt{\beta} \, \ph_{2}\Vert_{0, \Gamma}^{2}
+ 2 \, \Vert \sqrt{\beta} \, \stress \Vert_{1/2, \, \Gamma}^{2} \\
& \leq
2 \, \Vert \uoneh \Vert_{0, \Omega_{1}}^{2} 
+ 2 \, \Vert \g\Vert_{0, \Omega_{1}}^{2} 
+ 2 \, \Vert \sqrt{\beta} \, \ph_{2}\Vert_{0, \Gamma}^{2}
+ 2 \, \big\Vert \sqrt{\beta} \, \big\Vert_{L^{\infty}(\Gamma)} \, \Vert  \stress \Vert_{1/2, \, \Gamma}^{2} . 
\end{split}
\end{equation}
The second inequality holds due to the strong discretized Darcy equation \eqref{Eq Darcy Strong decomposed 0} i.e, $\uoneh + \grad \poneh = \g^{h}$, with $\g^{h}$ denoting the orthogonal projection of $\g$ onto $\grad \mathcal{Q}(\Omega_{2}, \triang^{h})$. In addition, $\Vert \g^{h} \Vert_{0, \Omega_{1}} \leq \Vert \g \Vert_{0, \Omega_{1}}$ which gives the third inequality. On the other hand, combining the estimates \eqref{Ineq Estimate of Pressure One} and \eqref{Ineq Control With Gradient and Trace} with \eqref{Ineq Estimate of Pressure Two} gives
   \begin{equation}\label{Ineq Second Estimate of Pressure One}
      \Vert \poneh\Vert_{0, \Omega_{1}}^{2}
      \leq C\, \big(\Vert \uoneh \Vert_{0, \Omega_{1}}^{2}
      + \Vert \sqrt{\beta} \, \ptwoh \Vert_{0, \Gamma}^{2}
      + \Vert \g \Vert_{0, \Omega_{1} }^{2}
      + \Vert \stress \Vert_{1/2, \Gamma}^{2}
      \big) ,
   \end{equation}
for $C  > 0$ large enough, which depends only on $\Omega$. Next, 
due to \eqref{Eq conservative strong decomposed 0} it holds that $\Vert \uoneh\Vert_{\Hdiv(\Omeone)}^{2} 
= 
\Vert \uoneh \Vert_{0, \Omeone}^{2} 
+ \Vert F\Vert_{0,\Omeone}^{2}$. Denoting 
$\kappa = \max \big\{1, 2 \, \big\Vert \sqrt{\beta} \, \big\Vert_{L^{\infty}(\Gamma)}  \big\}
\big(\Vert F\Vert_{0, \Omega}^{2}
+ \Vert \g\Vert_{0, \Omega}^{2}
+ \Vert \stress \Vert_{1/2,\,\Gamma}^{2}
+ \Vert \flux \Vert_{0, \Gamma}^{2}\big)$ and introducing these observations in \eqref{Ineq First a-priori Estimate} we get
\begin{equation*} 
C^{2}_{0} \Big[\Vert \uoneh \Vert^{2}_{0, \Omeone} 
+ \Vert \utwoh \Vert^{2}_{0, \Ometwo}  
+ \Vert \sqrt{\beta} \, \ptwoh \Vert^{2} _{0, \Gamma}
\Big]^{2} 
\leq 2 \max\{   C ,  1 \}\, \kappa
\Big[
\Vert \uoneh\Vert_{0, \Omeone}^{2} 
+ \Vert \utwoh\Vert_{0, \Ometwo}^{2}
+ \Vert \sqrt{\beta} \, \ptwoh\Vert_{0, \Gamma}^{2} \Big]
+ \kappa^{2}.
\end{equation*}
The expression above shows that for all $h > 0$, a square function is controlled by a linear function of the same argument, therefore, there must exist yet another constant still denoted by $C>0$, such that
\begin{align*} 
& \Vert \uoneh \Vert^{2}_{0, \Omeone} 
+ \Vert \utwoh \Vert^{2}_{0,\Ometwo}  
+ \Vert \sqrt{\beta} \, \ptwoh \Vert^{2} _{0, \Gamma}
 \leq C \,,  &
& \text{for all }\, h> 0 .
\end{align*}
From here, the strong Darcy equation \eqref{Eq Darcy Strong decomposed 1}, the \textsc{Inequality} \eqref{Ineq Control With Gradient and Trace}, the \textsc{Inequality} \eqref{Ineq Second Estimate of Pressure One} and the conservation \textsc{Statement} \eqref{Eq conservative strong decomposed 0} give the result.
\qed
\end{proof}
From the standard theory of general Hilbert spaces the following result is trivial.
\begin{corollary}\label{Th Weak Convergence of Solutions}
Assuming the hypotheses of \textsc{Theorem \ref{Th a-priori Estimates}} hold, there exist an element $\big([\uone', \ptwo'],[\utwo', \pone']\big)\in \X\times \Y$ and a subsequence, still denoted the same, such that $\big\{\big([\uoneh, \ptwoh],[\utwoh, \poneh]\big): h>0\}$ is weakly convergent to $ \big([\uone', \ptwo'],[\utwo', \pone']\big) $.
\end{corollary}
Before proving the strong convergence of the full sequence of solutions, we recall a standard finite element theory result.
\begin{proposition}\label{Th Density of FEM spaces}
Let $\Omega$ be an open polygonal domain of $\R^{2}$ satisfying \textsc{Hypothesis \ref{Hyp Geometry of the Domain}} and let $\{\triang^{h}: h > 0\}$ be a monotone sequence of consistent triangulations with size $h\rightarrow 0$, then 
\begin{subequations}
\begin{equation}\label{Eq Velocity One FEM Closure}
\cl \big\{ \mathbf{RT}_{0}(\Omega_{1}, \triang^{h})   : h > 0 \big\} 
= \Hdiv(\Omega_{1}) , 
\end{equation}
\begin{equation}\label{Eq Pressure Two FEM Closure}
\cl \big\{ \mathcal{Q}(\Omega_{2}, \triang^{h})   : h > 0 \big\} 
 = E(\Omega_{2}) , 
\end{equation}
\begin{equation}\label{Eq Velocity Two FEM Closure}
\cl \big\{ \grad \mathcal{Q}(\Omega_{2}, \triang^{h})   : h > 0 \big\} 
= \V(\Omega_{2}) , 
\end{equation}
\begin{equation}\label{Eq Pressure One FEM Closure}
\cl \big\{ \mathcal{Q}(\Omega_{1}, \triang^{h})   : h > 0 \big\} 
 = L^{2}(\Omega_{2}) .
\end{equation}
\end{subequations}
\end{proposition}
\begin{proof}
The identities \eqref{Eq Velocity One FEM Closure} and \eqref{Eq Pressure One FEM Closure} are standard conformal finite element results. For the identity \eqref{Eq Pressure Two FEM Closure} it is enough to extend, in a continuous and linear fashion, the elements of $\mathcal{Q}(\Omega_{2}, \triang^{h})$ to polynomials of degree one in the whole domain $\Omega$. This extension yields the classic FEM space of continuous, piecewise linear affine functions (on the whole domain $\Omega$) associated to $\triang^{h}$, which we denote by $\mathcal{Q}^{1}(\Omega, \triang^{h})$. From the standard theory of conformal finite elements, we know that $\cl \big\{ \mathcal{Q}^{1}(\Omega, \triang^{h})   : h > 0 \big\} 
 = H^{1}(\Omega) $, in particular, the statement \eqref{Eq Pressure Two FEM Closure} holds. Finally, the identity \eqref{Eq Velocity Two FEM Closure} follows trivially from \eqref{Eq Pressure Two FEM Closure}.
 \qed
\end{proof}
Next we prove the convergence of the solutions and identify the limiting problem.
\begin{theorem}\label{Th Convergence of the Solutions}
Let $\{\triang^{h}: h>0\}$, $F$, $\g$, $\stress$, $\flux$, $\big([\uoneh, \ptwoh],[\utwoh, \poneh]\big)$ be as in \textsc{Theorem \ref{Th a-priori Estimates}}. Then, the element $\big([\uone', \ptwo'],[\utwo', \pone']\big) $ given by \textsc{Corollary \ref{Th Weak Convergence of Solutions}} is the unique solution to \textsc{Problem} \eqref{Pblm operators weak continuous solution}. Moreover, the whole sequence converges to this point i.e., 
\begin{align}\label{Stmt Global Weak Convergence}
& \big([\uoneh, \ptwoh],[\utwoh, \poneh]\big)\xrightarrow[h\rightarrow 0]{} 
\big([\uone, \ptwo],[\utwo, \pone]\big) ,&
& \text{weakly in } \X\times \Y.
\end{align}
\end{theorem}
\begin{proof}
In order to prove the result, it is enough to show that $\big([\uone', \ptwo'],[\utwo', \pone']\big)$ satisfies the variational \textsc{Statement} \eqref{Pblm weak continuous solution}. Let $\big([\vone, \qtwo],[\vtwo, \qone]\big)$ be an arbitrary element of $\X\times \Y$ and let $\big([\vone^{h}, \qtwo^{h}],[\vtwo^{h}, \qone^{h}]\big)$ be its orthogonal projection onto $\X_{h} \times \Y_{h}$. Due to \textsc{Proposition} \ref{Th Density of FEM spaces} the sequence $\big\{ \big([\vone^{h}, \qtwo^{h}],[\vtwo^{h}, \qone^{h}]\big): h> 0 \big\}$ converges strongly to $\big([\vone, \qtwo],[\vtwo, \qone]\big)$. Now test the variational formulation associated to \textsc{Problem} \eqref{Pblm FEM operators weak continuous solution}, this gives
\begin{subequations} 
\begin{multline*} 
\int_{\Omega_1}  a \, \uoneh \cdot \vone^{h} 
+ \int_{\Gamma}  \beta \, \ptwoh \, \qtwo^{h}  \, dS 
- \int_{\,\Gamma}\left(\uoneh\cdot\n\right) \qtwo^{h} \, d S
+\int_{\,\Gamma}\ptwoh \left(\vone^{h} \cdot\n\right)\, d S \\
- \int_{\Omega_1} \poneh \,\div\vone^{h} \,
- \int_{\Omega_{2}} \utwoh \cdot \grad \qtwo^{h} \, 
= \int_{\Omega_{2}}F\, \qtwo^{h}  - \int_{\Omega_{1}} \g \cdot \vone^{h} 
+ \int_{\Gamma} \stress\, (\vone^{h} \cdot\n)\, dS 
- \int_{\Gamma} \flux\, \qtwo^{h} \, dS ,
\end{multline*}
\begin{equation*} 
\int_{\Omega_1}\div\uoneh\, \qone^{h}   
+ \int_{\Omega_2} \grad \ptwoh \cdot\vtwo^{h}  
+ \int_{\Omega_2}  a \, \utwoh \cdot \vtwo^{h}  \,
= \int_{\Omega_{1}}F\, \qone^{h}  
- \int_{\Omega_{2}} \g \cdot \vtwo^{h} \, .
\end{equation*}
\end{subequations}
Notice that in both expressions above each summand of the left hand side converges since one of the factors is weakly convergent, while the other is strongly convergent. The right hand side also converges due to the strong convergence of the quantifiers. Consequently, the element $\big([\uone', \ptwo'],[\utwo', \pone']\big)$ satisfies the variational \textsc{Statement} \eqref{Pblm weak continuous solution} for any arbitrary test function $\big([\vone, \qtwo],[\vtwo, \qone]\big)$. It follows that $\big([\uone', \ptwo'],[\utwo', \pone']\big)$ is a solution of \textsc{Problem} \eqref{Pblm operators weak continuous solution} , this concludes the first part of the theorem.

For the second part, the well-posedness of \textsc{Problem }\eqref{Pblm operators weak continuous solution} gives the uniqueness of its solution.  Consequently, due to the \textsc{Estimate} \eqref{Ineq Uniform Boundedness}, any subsequence of $\big\{\big([\vone^{h}, \qtwo^{h}],[\vtwo^{h}, \qone^{h}]\big): h > 0\big\}$ would have yet another subsequence weakly convergent  to the solution of \textsc{Problem} \eqref{Pblm operators weak continuous solution}. Hence, the \textsc{Statement} \eqref{Stmt Global Weak Convergence} follows and the proof is complete. 
\qed
\end{proof}
Finally, we have
\begin{theorem}\label{Th Strong Convergence of the Solutions}
Let $\{\triang^{h}: h>0\}$, $F$, $\g$, $\stress$, $\flux$, $\big([\uoneh, \ptwoh],[\utwoh, \poneh]\big)$ be as in \textsc{Theorem \ref{Th a-priori Estimates}} above and let $\big([\uone, \ptwo],[\utwo, \pone]\big)$ be the solution to \textsc{Problem} \eqref{Pblm operators weak continuous solution}. Then
\begin{align}\label{Stmt Global Strong Convergence}
& \big([\uoneh, \ptwoh],[\utwoh, \poneh]\big)\xrightarrow[h\rightarrow 0]{} 
\big([\uone, \ptwo],[\utwo, \pone]\big) ,&
& \text{strongly in } \X\times \Y.
\end{align}
\end{theorem}
\begin{proof}
We use the standard approach. Test \textsc{Problem} \eqref{Pblm operators weak continuous solution} with $\big(\u, p\big)$ and add both equations, this yields
\begin{equation}\label{Eq Test on the diagonal}
\int_{\Omega_1}  a \, \vert \uone \vert^{2} 
+ \int_{\Omega_2}  a \, \vert \utwo \vert^{2} 
+ \int_{\Gamma}  \beta \, \vert \ptwo \vert^{2} \, dS 
= \int_{\Omega }F\, p - \int_{\Omega } \g \cdot \u
%
+ \int_{\Gamma} \stress\, (\uoneh\cdot\n)\, dS 
- \int_{\Gamma} \flux\, \ptwo\, dS .
\end{equation}
On the other hand, taking $\limsup$ in the \textsc{Identity} \eqref{Eq Test on the FEM diagonal} we get 
\begin{equation*} 
\begin{split}
\limsup_{h\rightarrow 0} \Big[\int_{\Omega_1}  a \, \vert \uoneh \vert^{2} 
+ \int_{\Omega_2}  a \, \vert \utwoh \vert^{2} 
& + \int_{\Gamma}  \beta \, \vert \ptwoh \vert^{2} \, dS\Big] \\
& = \int_{\Omega }F\, p - \int_{\Omega } \g \cdot \u
%
+ \int_{\Gamma} \stress\, (\uone\cdot\n)\, dS 
- \int_{\Gamma} \flux\, \ptwo\, dS \\
& = \int_{\Omega_1}  a \, \vert \uone \vert^{2} 
+ \int_{\Omega_2}  a \, \vert \utwo \vert^{2} 
+ \int_{\Gamma}  \beta \, \vert \ptwo \vert^{2} \, dS \\
& \leq \liminf_{h\rightarrow 0} \Big[\int_{\Omega_1}  a \, \vert \uoneh \vert^{2} 
+ \int_{\Omega_2}  a \, \vert \utwoh \vert^{2} 
+ \int_{\Gamma}  \beta \, \vert \ptwoh \vert^{2} \, dS\Big] .
\end{split}
\end{equation*}
In the expression above the equality of the second line holds due to the \textsc{Identity} \eqref{Eq Test on the diagonal} and the inequality of the third line holds due to the weak convergence \textsc{Statement} \eqref{Stmt Global Weak Convergence}. From here, due to standard Hilbert space theory, it follows that 
\begin{subequations}\label{Stmt First Stage Strong Convergence}
\begin{equation}\label{Stmt Strong Ltwo Convergence uone}
\Vert \uoneh - \uone \Vert_{0, \Omega} \xrightarrow[h \, \rightarrow \, 0]{}  0 \, ,  
\end{equation}
\begin{equation}\label{Stmt Strong Convergence utwo}
\Vert \utwoh - \utwo \Vert_{0, \Omega}  \xrightarrow[h \, \rightarrow \, 0]{}  0 \, , 
\end{equation}
\begin{equation}\label{Stmt Strong Convergence Trace ptwo}
\Vert \sqrt{\beta} (\ptwoh - \ptwo) \Vert_{0, \Gamma} 
\xrightarrow[h \, \rightarrow \, 0]{}  0 .
\end{equation}
\end{subequations}
On the other hand, the solution $\uoneh$ satisfies the discretization of the \textsc{Equation} \eqref{Eq conservative strong decomposed 0}. Therefore, it holds that $\div \uoneh = F^{h}$, where $F^{h}$ is the orthogonal projection of $F$ on the space $\mathcal{Q}(\Omega_{1}, \triang^{h})$. Since $\Vert F^{h} - F\Vert_{0, \Omega_{1}}\rightarrow 0$ it follows that $\Vert \div \uoneh - \div \uone\Vert_{0, \Omega_{1}}\rightarrow 0$ which, combined with the \textsc{Statement} \eqref{Stmt Strong Ltwo Convergence uone} yields
\begin{equation}\label{Stmt Strong Hdiv Convergence uone}
\Vert \uoneh - \uone \Vert_{\Hdiv(\Omega_{1})} 
\xrightarrow[h \, \rightarrow \, 0]{} 0 \, . 
\end{equation}
Next, the solution $\ptwoh$ satisfies the discretized version of Darcy's law \eqref{Eq Darcy Strong decomposed 1} i.e., $\utwoh + \grad \ptwoh = \g^{h}$. Again, $\g^{h}$ indicates the orthogonal projection of $\g$ onto $\grad \mathcal{Q}(\Omega_{2}, \triang^{h})$ and due to the strong convergence of the orthogonal projections it follows that $\Vert \grad \ptwoh - \grad \ptwo\Vert_{0, \Omega_{2}}\rightarrow 0$. The latter, combined with the \textsc{Statement} \eqref{Stmt Strong Convergence Trace ptwo} and the \textsc{Inequality} \eqref{Ineq Control With Gradient and Trace} implies
\begin{equation}\label{Stmt Strong Convergence ptwo}
\Vert \ptwoh - \ptwo \Vert_{1, \Omega_{2}} 
\xrightarrow[h \, \rightarrow \, 0]{} 0 .
\end{equation}
Finally, for the strong convergence of $ \big\{ \poneh: h > 0\big\} $, let $ \big\{ \voneh: h > 0\big\} \subseteq \Hdiv(\Omega_{1}) $ be a sequence of functions such that $ \div \voneh = \poneh $ and $ \Vert \voneh \Vert \leq C $; which exists because of \textsc{Theorem} \ref{Th Local Solution of the Dual Formulation}. It will be shown that any any subsequence of $ \big\{ \poneh: h > 0\big\} $ has another subsequence denoted with the index $ h' $ such that 
\begin{equation}\label{Eq Norms Limit Pressure One}
\lim\limits_{h'\,\rightarrow\, 0}\Vert \ponehp \Vert_{0, \Omega_{1}} = \Vert \pone\Vert_{0, \Omega_{1}} .
\end{equation}
Take a subsequence of $ \big\{ \poneh: h > 0\big\} $ still denoted the same. Due to the boundedness of $ \big\{ \voneh: h > 0\big\} $, there must exist a convergent subsequence denoted with the index $ h' $ which is weakly convergent in $ \Hdiv(\Omega_{1}) $ to an element $ \vone $. The \textsc{Identity} \eqref{Eq Test for Isolating pone} holds for any element in $ \X_{h'} $ in particular  
\begin{equation*}
\begin{split}
\Vert \ponehp \Vert_{0, \Omega_{1}}^{2} & = \int_{\Omega_1} \ponehp \,\ponehp\, \\
& = 
\int_{\Omega_1} \ponehp \,\div\vonehp\, \\
& =
-\int_{\Omega_1}  a \, \uonehp \cdot \vonehp 
-\int_{\,\Gamma}\ptwohp \left(\vonehp\cdot\n\right)\, d S 
+ \int_{\Omega_{1}} \g \cdot \vonehp
- \int_{\Gamma} \stress\, (\vonehp\cdot\n)\, dS  .
\end{split}
\end{equation*}
All the summands of the right hand side converge since one of the factors converges strongly while the other converges weakly, then the left hand side also converges, i.e.,
\begin{equation*}
\lim\limits_{h'\,\rightarrow\, 0}\Vert \ponehp \Vert_{0, \Omega_{1}}^{2}
=
-\int_{\Omega_1}  a \, \uone \cdot \vone 
-\int_{\,\Gamma}\ptwo \left(\vone\cdot\n\right)\, d S 
+ \int_{\Omega_{1}} \g \cdot \vone
- \int_{\Gamma} \stress\, (\vone\cdot\n)\, dS  .
\end{equation*}
Observe that, $ \div \vone = \pone $ since $ \big\{ \poneh: h > 0\big\} $ converges weakly to $ \pone $, now test the \textsc{Statement} \eqref{Pblm weak continuous solution 1} with $ [\vone, 0]\in \X $ to get
\begin{equation*}
\begin{split}
\Vert \pone\Vert_{0, \Omega_{1}}^{2} & = 
\int_{\Omega_1} \pone \,\pone \\
& = \int_{\Omega_1} \pone \,\div\vone\,\\
& =
-\int_{\Omega_1}  a \, \uone \cdot \vone 
-\int_{\,\Gamma}\ptwo \left(\vone\cdot\n\right)\, d S 
+ \int_{\Omega_{1}} \g \cdot \vone
- \int_{\Gamma} \stress\, (\vone\cdot\n)\, dS  .
\end{split}
\end{equation*}
Equating both expressions above we get \textsc{Equation} \ref{Eq Norms Limit Pressure One}. From elementary real analysis it follows that the full sequence of $L^{2}(\Omega_{1}) $-norms converges to $ \Vert \pone \Vert_{0, \Omega_{1}} $. Finally, from standard Hilbert space theory it follows that $ \Vert \ponehp -  \pone\Vert_{0, \Omega_{1}} \xrightarrow[h\,\rightarrow\, 0]{} 0 $ 
and the proof is complete.
\qed
\end{proof}
%
%
%
%
%
%
%
%
%
%
%
%
%
\subsection{Rate of Convergence}\label{Sec Rate of Convergence}
%
%
%
%
In this section the rate of convergence analysis is presented. It will be done assuming \textsc{Hypothesis \ref{Hyp non null local storage coefficient}} is satisfied. We proceed in the standard way, see \cite{GaticaMixed} 
\begin{definition}\label{Def Projection Operator}
Given $ h > 0 $ fixed, define the operator $ \Lambda_{h}: \X\times \Y \rightarrow \X\times \Y $ such that $ \big([\vone, \ptwo], [\vtwo, \pone]\big)  $ is mapped to the unique solution $ \big([\voneh, \ptwoh], [\vtwoh, \poneh]\big) \in \X\times \Y $ of the problem 
\begin{equation}\label{Eq Projection Operator}
\begin{split}
\A[ \uoneh,\,\ptwoh] + \B ' [ \utwoh,\, \poneh]  
= \A[ \uone,\,\ptwo] + \B ' [ \utwo,\, \pone]  \quad \text{in}\; \X_{h} ' ,\\
- \B [ \uoneh,\, \ptwoh]  + \C [ \utwoh,\, \poneh] 
= - \B [ \uone,\, \ptwo]  + \C [ \utwo,\, \pone] \quad \text{in}\; \Y_{h} ' ,
\end{split}
\end{equation}
followed by the canonical embedding $ j: \X_{h}\times \Y_{h} \hookrightarrow \X\times \Y  $.
\end{definition}
\begin{remark}\label{Rem Projection Operator}
\begin{enumerate}[(i)]
\item The operator $ \Lambda_{h} $ above is well-defined due to \textsc{Theorem} \ref{Th FEM Well Posedness Problem}.

\item Due to \textsc{Theorem} \ref{Th FEM Well Posedness Problem}, the operator is $ \Lambda_{h} $ linear, continuous and idempotent.
\end{enumerate}
\end{remark} 
We have the following result
\begin{theorem}\label{Th Bounded by the Infimum}
\begin{enumerate}[(i)]
\item There exists $ M > 0 $ such that 
\begin{align}\label{Stmt Projections Uniform Boundedness}
& \Vert \Lambda_{h} \Vert \leq M , &
& \text{for all } h > 0 ,
\end{align}
 i.e., the family  $ \big\{ \Lambda_{h}: h > 0 \big\} $ is globally bounded.

\item Let $ \big( [\uone, \ptwo], [\utwo, \pone]\big) \in \X\times \Y $, $ \big( [\uoneh, \ptwoh], [\utwoh, \poneh]\big) \in \X_{h}\times \Y_{h} $ be the unique solutions to \textsc{Problem} \eqref{Pblm operators weak continuous solution} and \eqref{Pblm FEM operators weak continuous solution} respectively, then
\begin{multline}\label{Ineq Discrete Solution as Infimum}
\big\Vert \big( [\uone, \ptwo], [\utwo, \pone]\big) 
- \big( [\uoneh, \ptwoh], [\utwoh, \poneh]\big)\big\Vert_{\X\times \Y} \leq 
\\
(1 + M ) \inf\limits_{ ( [\voneh, \qtwoh], [\vtwoh, \qoneh] )\in \X_{h} \times \Y_{h}}
\big\Vert 
\big( [\uone, \ptwo], [\utwo, \pone]\big)
- \big( [\voneh, \qtwoh], [\vtwoh, \qoneh]\big)
\big\Vert_{\X\times \Y} ,
\end{multline}
for every $ h >0 $.
\end{enumerate}
\end{theorem}
\begin{proof}
\begin{enumerate}[(i)]
\item Given an arbitrary element $ \big( [\vone, \qtwo], [\vtwo, \qone]\big) \in \X\times \Y $, defining $ \widetilde{F}_{1} \defining  \A[ \vone,\,\qtwo] + \B ' [ \vtwo,\, \qone] $ and $ \widetilde{F}_{2} = - \B [\vone,\,\qtwo]  + \C [\vtwo,\,\qone] $, it is clear due to \textsc{Theorem} \ref{Th weak formulation of continuos version} that $ \big( [\vone, \qtwo], [\vtwo, \qone]\big) $ is the unique solution to \textsc{Problem} \ref{Pblm operators weak continuous solution} with $ F_{i} $ replaced by $ \widetilde{F}_{i} $ for $ i = 1, 2 $. Recalling Definition \ref{Def Projection Operator}, it is clear that $ \Lambda_{h} \big( [\vone, \qtwo], [\vtwo, \qone]\big) $ is the unique solution to \textsc{Problem} \eqref{Pblm FEM operators weak continuous solution} and due to the strong convergence analysis, \textsc{Theorem} \ref{Th Strong Convergence of the Solutions}, it holds that $ \big\Vert \big( [\vone, \qtwo], [\vtwo, \qone]\big) - \Lambda_{h}\big( [\vone, \qtwo], [\vtwo, \qone]\big) \big\Vert \xrightarrow[h\,\rightarrow\, 0]{} 0 $. In particular, the sequence $ \big\{ \Lambda_{h}\big( [\vone, \qtwo], [\vtwo, \qone]\big) : h > 0 \big\} $ is bounded i.e., the family of operators $ \big\{ \Lambda_{h} : h > 0 \big\} $ is bounded pointwise; due to the Banach-Steinhaus Uniform Boundedness Principle (from standard Functional Analysis theory), the Statement \eqref{Stmt Projections Uniform Boundedness} holds.

\item Since $ \Lambda_{h} $ is idempotent, observe that 
\begin{equation*}
\big( [\uone, \ptwo], [\utwo, \pone]\big) 
- \big( [\uoneh, \ptwoh], [\utwoh, \poneh]\big) =
\big( I - \Lambda_{h} \big) 
\Big( 
\big( [\uone, \ptwo], [\utwo, \pone]\big) 
- \big( [\voneh, \qtwoh], [\vtwoh, \qoneh]\big)
\Big).
\end{equation*}
From \textsc{Statement} \ref{Stmt Projections Uniform Boundedness} and the expression above, \textsc{Inequality} \eqref{Ineq Discrete Solution as Infimum} follows trivially.
\end{enumerate}
\qed
\end{proof}
Finally, we have the rate of convergence result
\begin{theorem}\label{Th Rate of Convergence}
Let $ \big( [\uone, \ptwo], [\utwo, \pone]\big) \in \X\times \Y $, $ \big( [\uoneh, \ptwoh], [\utwoh, \poneh]\big) \in \X_{h}\times \Y_{h} $ be the unique solutions to \textsc{Problem} \eqref{Pblm operators weak continuous solution} and \eqref{Pblm FEM operators weak continuous solution} respectively, then
\begin{subequations}\label{Eq Rates of Convergence}
\begin{equation}\label{Eq Velocity One Rate of Convergence}
\big\Vert \uone - \uoneh \big\Vert_{\Hdiv(\Omega_{1})} =
\mathcal{O}(h),
\end{equation}
\begin{equation}\label{Eq Velocity Two Rate of Convergence}
\big\Vert \utwo - \utwoh \big\Vert_{0, \Omega_{2}} = 
\mathcal{O}(h),
\end{equation}
\begin{equation}\label{Eq Pressure One Rate of Convergence}
\big\Vert \pone - \poneh \big\Vert_{0, \Omega_{1}} = 
\mathcal{O}(h),
\end{equation}
\begin{align}\label{Eq Presure Two Rate of Convergence}
& \big\Vert \ptwo - \ptwoh \big\Vert_{0, \Omega_{2}} = 
\mathcal{O}(h^{2}), &
& \big\Vert \ptwo - \ptwoh \big\Vert_{1, \Omega_{2}} = 
\mathcal{O}(h) .
\end{align}
\end{subequations}
\end{theorem}
\begin{proof}
In order to prove \textsc{Inequality} \eqref{Eq Velocity One Rate of Convergence} define $ \widetilde{F}_{1} \defining  \A[ \uone,\,0] + \B ' [ \boldsymbol{0},\, 0] $ and $ \widetilde{F}_{2} = - \B [\uone,\,0]  + \C [\boldsymbol{0},\,0] $. Again, due to \textsc{Theorem} \ref{Th weak formulation of continuos version} $ \big( [\uone, 0], [\boldsymbol{0}, 0]\big) $ is the unique solution to \textsc{Problem} \ref{Pblm operators weak continuous solution} with $ F_{i} $ replaced by $ \widetilde{F}_{i} $ for $ i = 1, 2 $. Applying \textsc{Inequality} \eqref{Ineq Discrete Solution as Infimum} yields
\begin{multline*} 
\big\Vert \big( [\uone, 0], [\boldsymbol{0}, 0]\big) 
- \big( [\uoneh, 0], [\boldsymbol{0}, 0]\big)\big\Vert_{\X\times \Y} \leq 
\\
(1 + M ) \inf\limits_{ ( [\voneh, \qtwoh], [\vtwoh, \qoneh] )\in \X_{h} \times \Y_{h}}
\big\Vert 
\big( [\uone, 0], [\boldsymbol{0}, 0]\big)
- \big( [\voneh, \qtwoh], [\vtwoh, \qoneh]\big)
\big\Vert_{\X\times \Y} .
\end{multline*}
Let $ \Pi_{h} $ be the global Raviart-Thomas interpolation operator, then $ \big( [\Pi_{h}\uone, 0], [\boldsymbol{0}, 0]\big) \in \X_{h}\times \Y_{h} $; recalling the inequality above it follows
\begin{equation*} 
\big\Vert \uone - \uoneh \big\Vert_{ \Hdiv(\Omega_{1}) } \leq 
(1 + M ) \big\Vert \uone - \Pi_{h}\uoneh \big\Vert_{ \Hdiv(\Omega_{1}) } 
\leq \mathcal{O}(h) .
\end{equation*}
In the expression above, the last inequality follows from standard finite element theory for interpolation operators, see \cite{GaticaMixed}. The remaining statements in \eqref{Eq Rates of Convergence} are shown using the same scheme.
\qed
\end{proof}
\begin{remark}
Observe that the rates of convergence summarized in \eqref{Eq Rates of Convergence} are all the standard ones, no gain or deterioration has been added by the scheme. This is because there were no strong coupling conditions in building the spaces, neither the continuous $ \X, \Y $, nor the discrete ones $ \X_{h}, \Y_{h} $. The interface exchange conditions are satisfied weakly, i.e., only by the solution of the problems \eqref{Pblm weak continuous solution} and \eqref{Pblm FEM operators weak continuous solution} respectively.
\end{remark}
%
%
%
%
%
%
%
%
%
\section{Numerical Examples}\label{Sec Numerical Example}
%
%
\begin{figure}[h] 
	\centering
	\begin{subfigure}
	[Basic experimentation domain $\Omega$ and map $\map$. ]
			{\includegraphics[scale = 0.6]{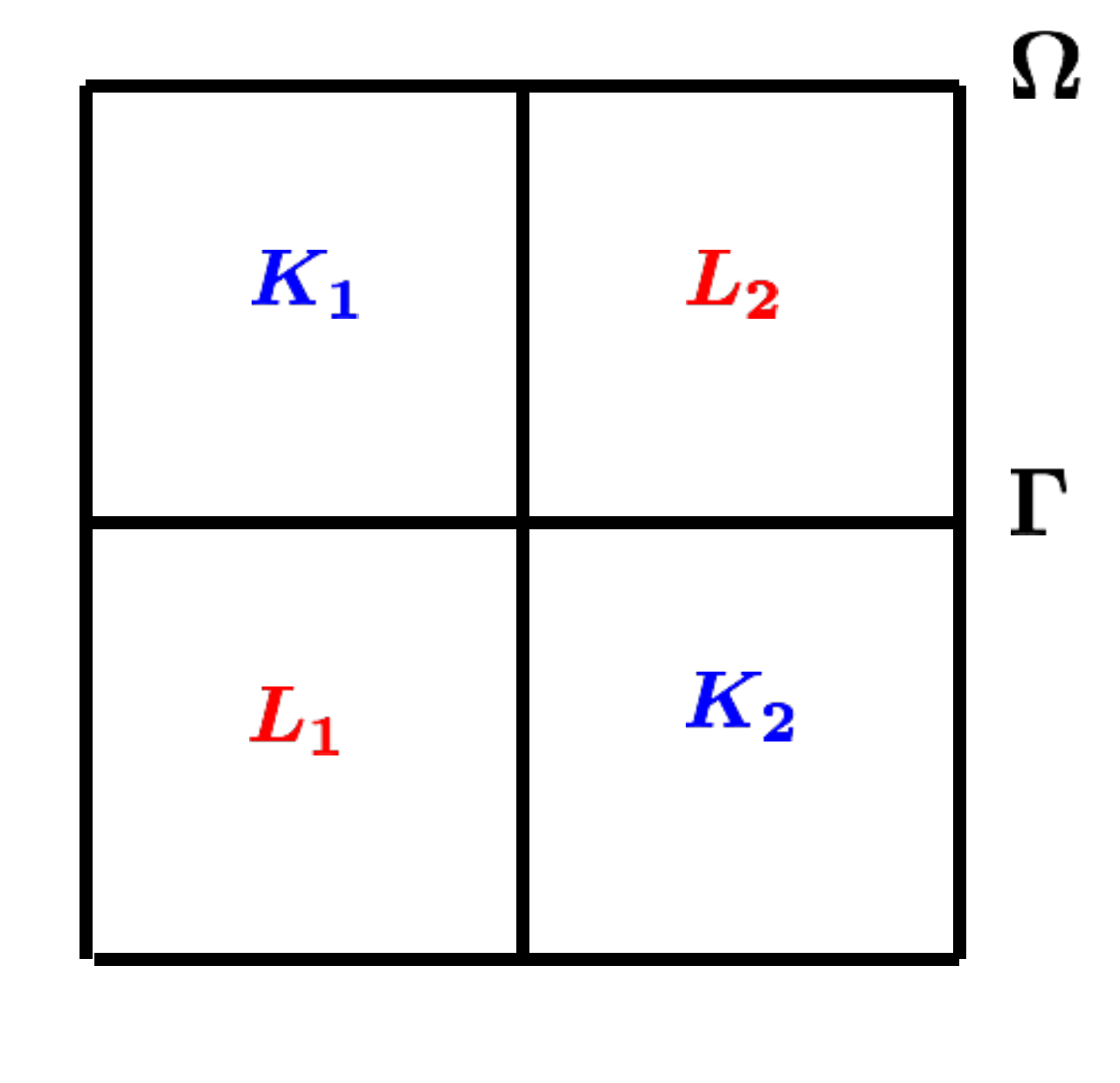} }
	\end{subfigure} 
	~ 
	\begin{subfigure}[Consistent grid example.]
			{\includegraphics[scale = 0.6]{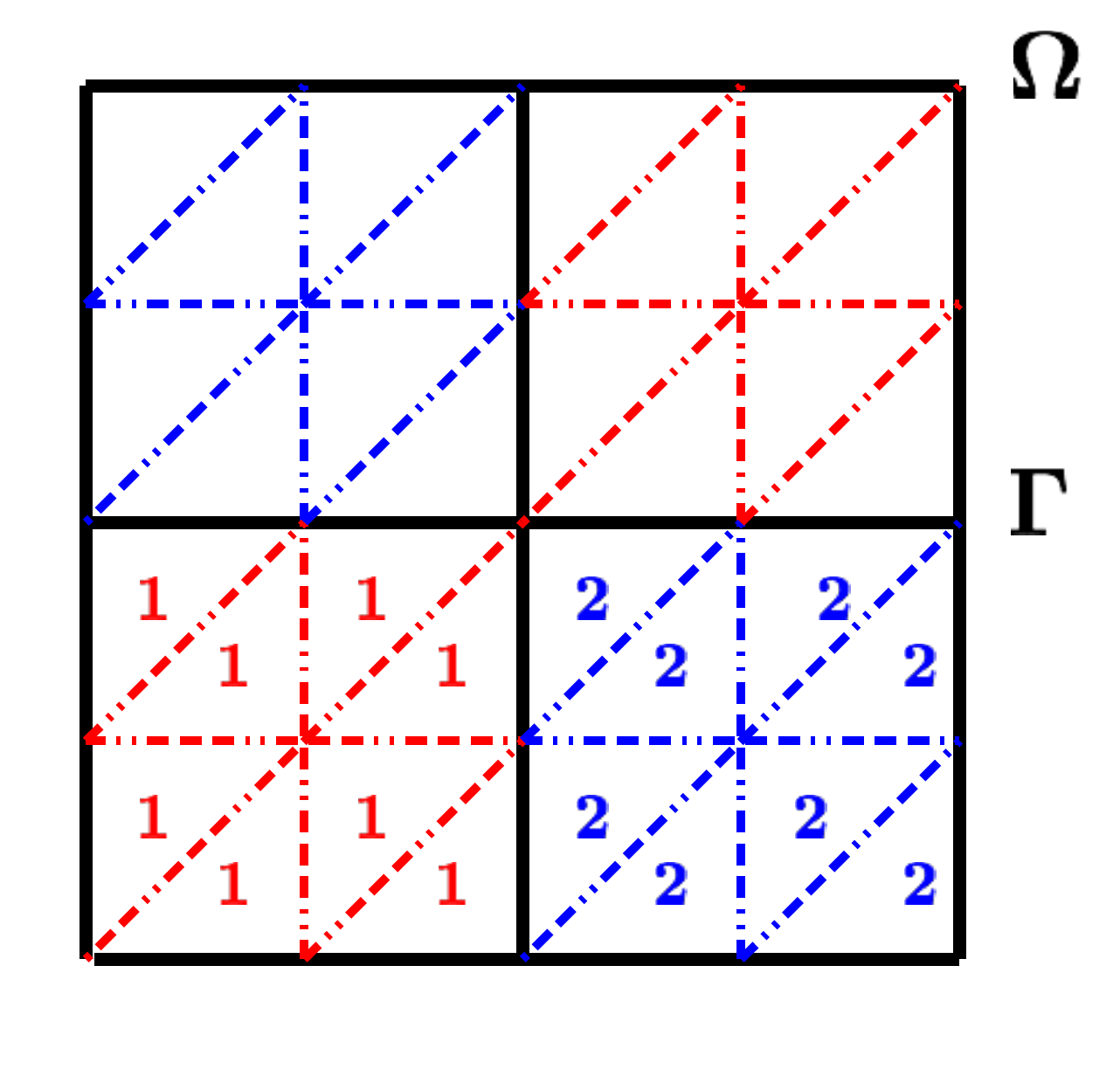} } 
	\end{subfigure} 
	%
	%
	%
	\caption{Figure (a) depicts a bipartite map $\map = (\mapone, \maptwo)$ example for the region $\Omega = [0,1] \times [0,1]$. Subregions belonging to $\map_{1}$ are red-colored and the subregions belonging to  $\map_{2}$ are blue-colored. 
	Figure (b) depicts an example of a grid $\triang$ consistent with the map $\map$. Some of the triangles belonging to $\tone$ and $\ttwo$ have been labeled with $1$ and $2$ respectively. \label{Fig Bipartite Map Grid Numerical Example} }
\end{figure}
In this section we present two numerical examples to illustrate the method, the first showing a case of continuity, the second a slight perturbation of the first to illustrate how the method handles discontinuities across interfaces. The numerical examples use the finite dimensional spaces $\X_{h}$, $\Y_{h}$ introduced in \eqref{Def finite function spaces}. The experiments are executed in a MATLAB script using adaptations of the codes \textbf{EBmfem.m} (see, \cite{Carstensen1}, \cite{CarstensenBahariawati}) and \textbf{fem2d.m} (see, \cite{CarstensenAlbertyFunken}, \cite{Carstensen2}).
\newline
\newline
For the sake of clarity, we adopt the domain $\Omega$, the interface $\Gamma$ and the subdomains $\Omega_{1}$, $\Omega_{2}$ as follows (see \textsc{Figure} \ref{Fig Bipartite Map Grid Numerical Example} (a))  
\begin{equation}\label{Def Geometric Parameters}
\begin{split} 
& \Omega \defining (-1,1) \times (-1,1) , 
\qquad \qquad\qquad \quad\quad
 \Gamma \defining (-1,1) \times \{0\} \cup  \{0\} \times  (-1, 1) , \\
& \Omega_{1} \defining (-1,0) \times  (-1, 0) \cup  (0, 1) \times  (0, 1)  ,\quad
\Omega_{2} \defining  (-1, 0) \times (0, 1)\cup  (0, 1) \times  (-1, 0) .
\end{split}
\end{equation}
Again, for simplicity, all the experiments run on the uniform Cartesian grid, see \textsc{Figure} \ref{Fig Bipartite Map Grid Numerical Example} (b). The sequence of grids $\{\triang^{i}: 0\leq i\leq 5 \}$ has correspongind sizes $ h_{i}^{-1} = 2^{i} $ for $0 \leq i \leq 5$; consequently it is a monotone sequence as described in \textsc{Definition} \ref{Def Consistent Grid}. The experimental computation for the order of convergence $r$, uses the standard approach. Assuming that the error satisfies $e = \mathcal{O}(h^{r})$, we approximate $r$ by
\begin{align*} 
& r \sim \frac{\log e_{k + 1} - \log e_{k}}{\log h_{k + 1} - \log h_{k} } 
=  \frac{\log e_{k} - \log e_{k + 1}}{\log 2 } ,&
& \text{for all }\, 0\leq k \leq 4 .
\end{align*}
In the expression above, the last equality holds due to the particular nature of the grids' size. 
\begin{example}\label{Ex Continuous Example}
The purpose of the present example is to illustrate how the method handles problems free of discontinuities across the interfaces. The exact solution in this case is given by
\begin{subequations}\label{Eqn Cont Exact Solutions}
\begin{align}\label{Eqn Exact Pressure}
& p: \Omega \rightarrow \R \, ,&
& p(x, y) = x \, y \, (x -1)^{2} (y - 1)^{2}(x + 1)^{2}(y + 1)^{2},  
\end{align}
\begin{align}\label{Eqn Exact Velocity}
& \u: \Omega\rightarrow \R^{2} \, , &
& \u(x, y) = -\grad p(x,y),
\end{align}
\end{subequations}
see \textsc{Figure} \ref{Fig Exact Solution Numerical Example}. The forcing terms are
\begin{subequations}\label{Eqn Forcing Terms}
\begin{equation}\label{Eqn Int Forcing Terms}
\begin{split}
\g: \Omega \rightarrow \R^{2}, \qquad\qquad\qquad &
\g = \boldsymbol{0} ,  \\
F: \Omega \rightarrow \R , \qquad\qquad\qquad &
F = -\div \grad p ,  
\end{split}
\end{equation}
\begin{align}\label{Eqn Interface Forcing Terms}
%
\stress, \flux: \Gamma \rightarrow \R , \qquad\qquad\qquad & 
\stress = 0 , 
& \flux  = 0 . 
%
\end{align}
\end{subequations}
\begin{figure}[t] 
	\centering
	\begin{subfigure}
	[Pressure Exact Solution.]
		{\resizebox{7.8cm}{8.0cm}
			{\includegraphics[scale = 0.33]{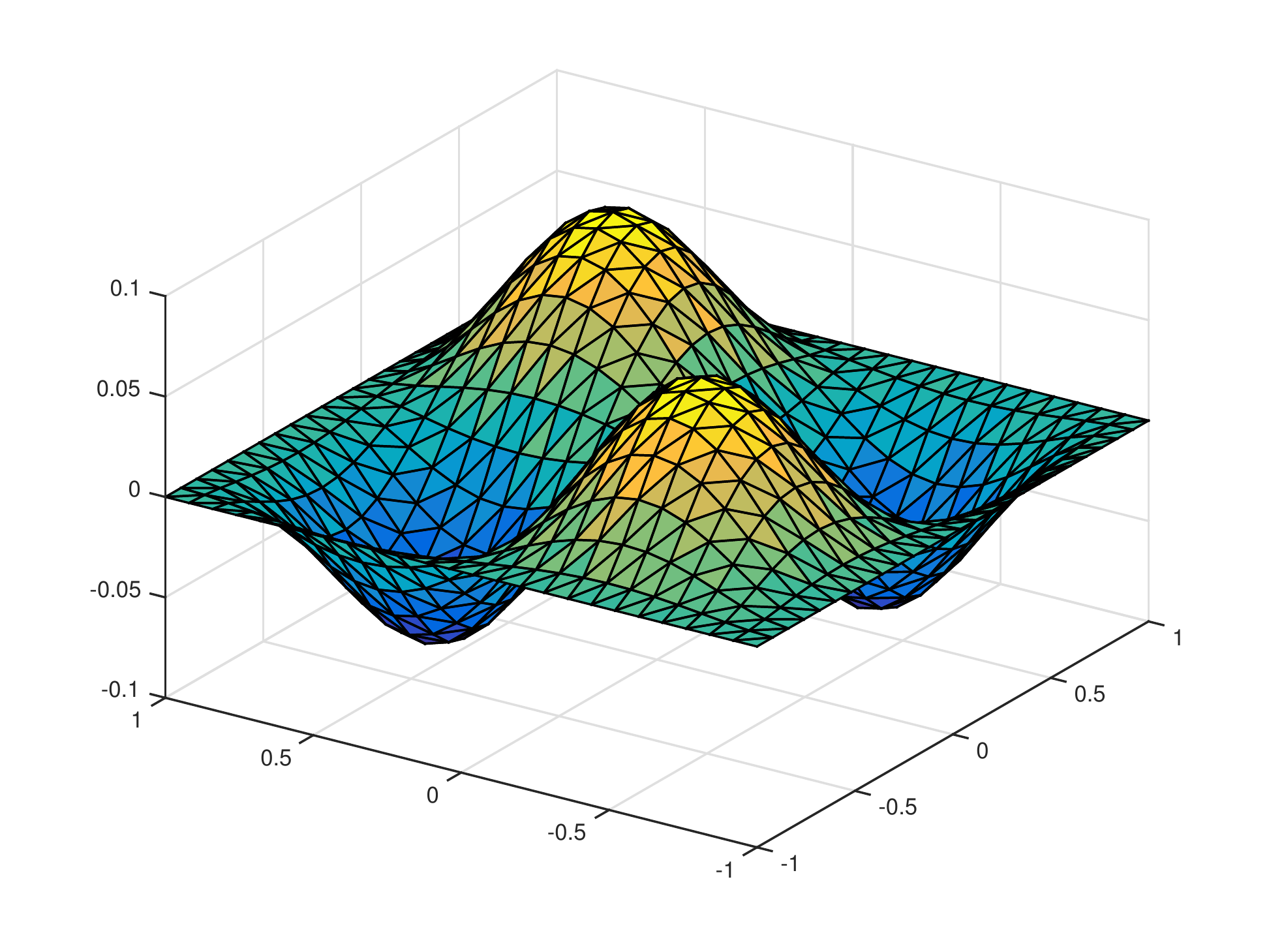} } 
			}
	\end{subfigure} 
	~ 
	\begin{subfigure}[Flux Exact Solution.]
		{\resizebox{7.8cm}{8.0cm}
			{\includegraphics{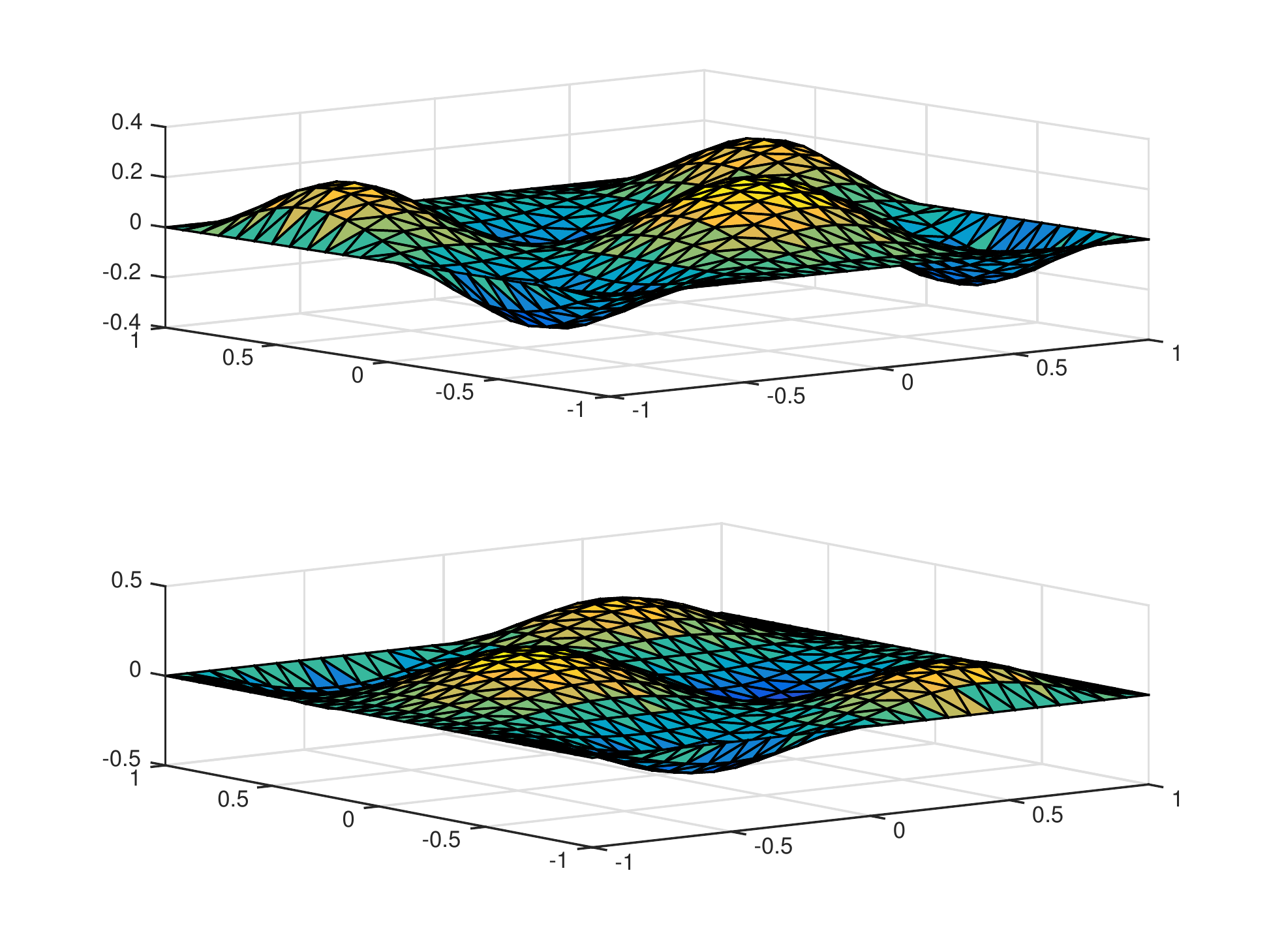} } 
			}                
	\end{subfigure} 
	%
	%
	%
	\caption{\textsc{Example} \ref{Ex Continuous Example}. Figure (a) depicts the pressure of the exact solution $p(x, y) =  x \, y \, (x -1)^{2} (y - 1)^{2}(x + 1)^{2}(y + 1)^{2}$, see \textsc{Equation} \eqref{Eqn Exact Pressure}.
	Figure (b) depicts the flux of the exact solution 
	$ \u = - \grad p 
	$, see \textsc{Equation} \eqref{Eqn Exact Velocity}. On the upper right corner is depicted the $\boldsymbol{x}$-component while the lower right corner displays the $\boldsymbol{y}$-component. \label{Fig Exact Solution Numerical Example} }
\end{figure}
\begin{figure}[h] 
	\centering
	\begin{subfigure}
	[Pressure Approximate Solution. ]
		{\resizebox{7.8cm}{8.0cm}
			{\includegraphics{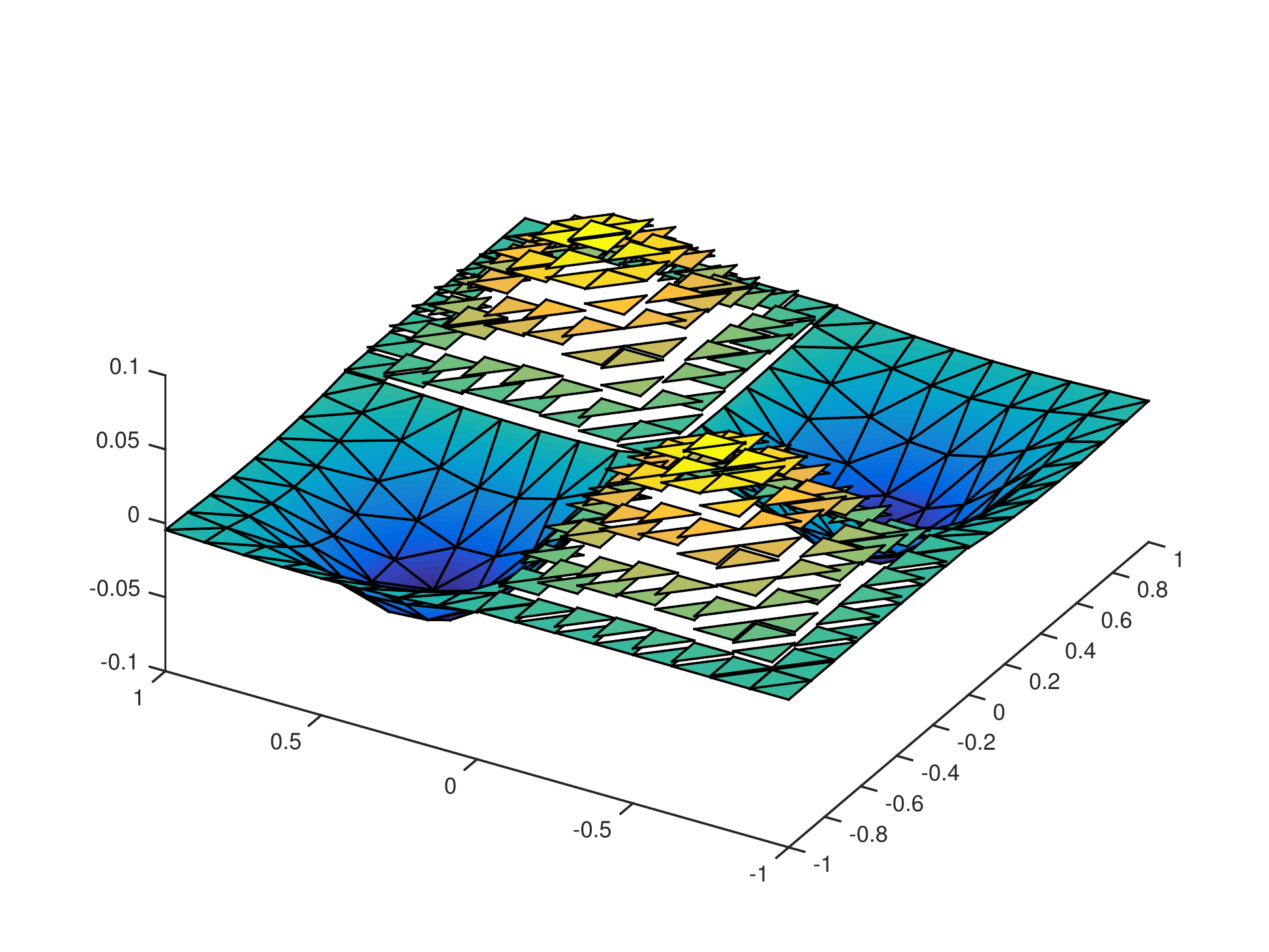} } 
			}
	\end{subfigure} 
	~ 
	\begin{subfigure}[Flux Approximate Solution.]
		{\resizebox{7.8cm}{8.0cm}
			{\includegraphics[scale = 0.33]{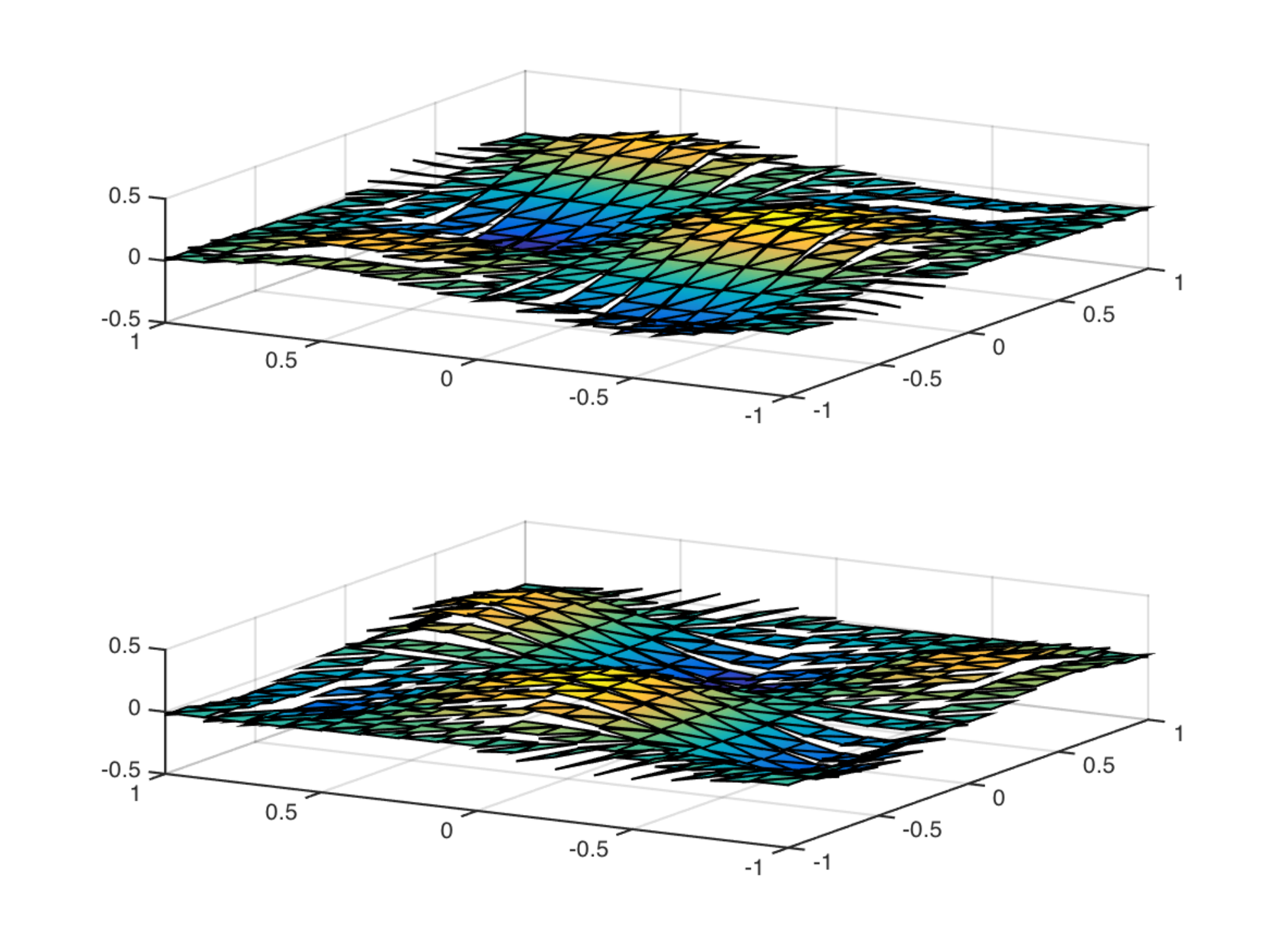} } 
			}                
	\end{subfigure} 
	%
	%
	%
	\caption{\textsc{Example} \ref{Ex Continuous Example}, approximate solution for a mesh of size $ h^{-1} = 8 $. The sub-domains are $\Omega_{1} = (-1, 0) \times (-1, 0)\cup  (0, 1) \times  (0, 1)$ and $\Omega_2 =  (-1, 0) \times (0, 1)\cup  (0, 1) \times  (-1, 0)$, see \textsc{Identity} \eqref{Def Geometric Parameters}.  Figure (a) depicts the pressure $p^{h}$ of the approximate solution, it is piecewise constant on the domain $\Omega_{1} $ and piecewise linear affine on the domain $\Omega_2$.
	Figure (b) depicts the flux of the approximate solution $\u^{h}$. On the upper right corner is depicted the $\boldsymbol{x}$-component of the flux 
	, which is continuous across \textbf{horizontal edges} of $\Omega_{1}$ and piecewise constant on the domain $\Omega_2$.	
	On the lower right corner we display 
	the $\boldsymbol{y}$-component of the flux, which is continuous across \textbf{vertical edges} of $\Omega_{1}$ and piecewise constant on the domain $\Omega_2$. \label{Fig Approximate Solution Numerical Example} }
\end{figure}
It is direct to see that $[\u, p]$ defined by \eqref{Eqn Cont Exact Solutions}
is the exact solution of the \textsc{Problem} \eqref{Eq porous media strong decomposed} on the geometric domain described by \eqref{Def Geometric Parameters} with the forcing terms defined in \eqref{Eqn Forcing Terms}. In particular, the boundary conditions \eqref{Eq Drained Condition decomposed}, \eqref{Eq Non-Flux Condition decomposed} and the interface exchange conditions \eqref{Eq normal flux balance}, \eqref{Eq normal stress balance} are satisfied.

The convergence results are displayed in the \textsc{Tables} \ref{Table Pressure Approximation} and \ref{Table Velocity Approximation} below, the convergence rate behaves as expected, except for $ \pone $, we have
\begin{subequations}\label{Stmt Numerical Rates of Convergence Ex 1}
\begin{align}\label{Stmt Pressure Numerical Rates of Convergence Ex 1}
& \Vert \poneh - \pone \Vert_{0, \Omega_{1}} = \mathcal{O}(h^{1.8}) ,&
& \Vert \ptwoh - \ptwo \Vert_{0, \Omega_{2}} = \mathcal{O}(h^{2}) , &
& \Vert \ptwoh - \ptwo \Vert_{1, \Omega_{2}} = \mathcal{O}(h) . 
\end{align}
\begin{align}\label{Stmt Velocity Numerical Rates of Convergence Ex 1}
& \Vert \uoneh - \uone \Vert_{0, \Omega_{1}} = \mathcal{O}(h) ,&
& \Vert \uoneh - \uone \Vert_{\Hdiv(\Omega_{1})} = \mathcal{O}(h) , &
& \Vert \utwoh - \utwo \Vert_{0, \Omega_{2}} = \mathcal{O}(h) . 
\end{align}
\end{subequations}
Finally, the numerical solution for $ h^{-1} = 8 $ is depicted in \textsc{Figure} \ref{Fig Approximate Solution Numerical Example}; the choice of the grid was based on optical clarity to illustrate both: the nature of discrete solution and its convergence to the continuous solution. 
%
\begin{table}[h!]
\caption{Pressures Convergence Table, \textsc{Example} \ref{Ex Continuous Example}}\label{Table Pressure Approximation}
\def\arraystretch{1.4}
\rowcolors{2}{gray!25}{white}
\begin{center}
\begin{tabular}{ c c c c c c c }
    \hline
    \rowcolor{gray!50}
$h ^{-1}$  
& $\Vert  \poneh- p_{1}  \Vert_{ 0, \Omega_{1} } $ 
& $r$  
& $\Vert  \ptwoh- p_{2}  \Vert_{0, \Omega_{2} }$ 
& $r$
& $\Vert  \ptwoh- p_{2}  \Vert_{1, \Omega_{2} }$ 
& $r$ \\ 
    \toprule
$1$ &   0.1836  &   &  2.8144  &  &   0.8643 &   \\
$ 2 $  &  0.0261  &  0.4383   &  0.0721  & 5.2867   & 0.1976  &  2.1289   \\
$ 4 $ &   0.0091  &  1.5201 &  0.0226  & 1.6737    &  0.0887  &   1.1556   \\
$ 8 $  &  0.0026  &    1.8074 &  0.0062  & 1.8660   &  0.0422  &  1.0717    \\ 
$ 16 $  &  0.0007  &     1.8931 &  0.0016  & 1.9542    &  0.0209  &  1.0137 \\ 
$ 32 $  &  0.0002  &     1.807  &  0.0004  & 2.0000   &  0.0104   &   1.0069  \\ 
    \hline
\end{tabular}
\end{center}
\end{table}
%
%
\begin{table*}[h!]
\caption{Velocities Convergence Table, \textsc{Example} \ref{Ex Continuous Example} }\label{Table Velocity Approximation}
\def\arraystretch{1.4}
\rowcolors{2}{gray!25}{white}
\begin{center}
\begin{tabular}{ c c c c c c c }
    \hline
    \rowcolor{gray!50}
$ h^{-1} $  
& $\Vert  \uone^{h}- \uone  \Vert_{ 0, \Omega_{1} } $ 
& $r$  
& $\Vert  \uone^{h}- \uone  \Vert_{\Hdiv(\Omega_{1})}$ 
& $r$
& $\Vert  \utwo^{h}- \utwo  \Vert_{ 0, \Omega_{2} } $ 
& $r$ \\ 
    \toprule
$ 1 $ &   0.8264 &    &   0.8264  &  &   0.9184  & \\
$ 2 $  &   0.1409  & 2.5522   &  0.1409  &  2.5522   &  0.1840   &  2.3194 \\
$ 4 $ &   0.0617  &   1.1913  &  0.0617  &   1.1913  &  0.0857   &  1.1023  \\
$ 8 $  &  0.0302  &   1.0307  &  0.0302  &   1.0307   &   0.0417 &  1.0392   \\ 
$ 16 $  &  0.0150  &  1.0096  &  0.0150  &  1.0096   &  0.0208  &  1.0035   \\ 
$ 32 $  &  0.0075  & 1.0000  &  0.0075   &  1.0000   &  0.0104  &  1.0000   \\ 
    \hline
\end{tabular}
\end{center}
\end{table*}
\end{example}
\begin{example}\label{Ex Discontinuous Example}
The present example is a perturbation of the previous one, in order to illustrate how the method handles problems with simultaneous discontinuities across the interfaces in both: the normal flux and the normal stress. The perturbation is localized on the fourth quadrant of the domain $(0, 1)\times (0, -1)$. The analytic solution in this case is given by
\begin{subequations}\label{Eqn Disc Exact Solutions}
\begin{equation}\label{Eqn Disc Exact Pressure} 
\begin{split}
p : \Omega \rightarrow & \R \, ,\\
p(x, y)  = \, & x \, y \, (x -1)^{2} (y - 1)^{2}(x + 1)^{2}(y + 1)^{2} \\
& + \frac{1}{20}\,\big( (x - 1)^2 - (y + 1)^2 \big)\ind_{(1,0)\times (0, -1)}(x,y)
, 
\end{split}
\end{equation}
\begin{align}\label{Eqn Disc Exact Velocity}
& \u: \Omega\rightarrow \R^{2} \, , &
& \u(x, y) = -\grad p(x,y), 
\end{align}
\end{subequations}
see \textsc{Figure} \ref{Fig Disc Exact Solution Numerical Example}. The forcing terms are acting inside the domains are identical to the previous example,
\begin{subequations}\label{Eqn Disc Forcing Terms}
\begin{equation}\label{Eqn Disc Interior Forcing Terms}
\begin{split}
\g: \Omega \rightarrow \R^{2}, \qquad\qquad\qquad &
\g = \boldsymbol{0} ,  \\
F: \Omega \rightarrow \R , \qquad\qquad\qquad &
F = -\div \grad p ,  
\end{split}
\end{equation}
In this case, the interface forcing terms account for the jumps of the solution across the interface, i.e., according to the interface exchange conditions \eqref{Eq normal flux balance}, \eqref{Eq normal stress balance} , $\stress$ and $\flux$ are given by 
\begin{equation}\label{Eqn Disc Interface Forcing Terms}
\begin{split}
\stress & :  \Gamma \rightarrow \R ,  \\
\stress (x, y) & =  \frac{1}{20} \big(  (x - 1)^2 - 1 \big)\ind_{(0,1)\times \{0\}} (x, y)
+ \frac{1}{20} \big( 1 - (y + 1)^2  \big)\ind_{ \{0\} \times (-1,0) } (x, y), \\
\flux & : \Gamma \rightarrow \R , \\
\flux  (x, y)  & =  \frac{1}{20} \big(  x - 4 \big)\ind_{(0,1)\times \{0\}} (x, y)
+ \frac{1}{20} \big( 4 - y )  \big)\ind_{ \{0\} \times (-1,0) } (x, y) . 
\end{split}
\end{equation}
\end{subequations}
\begin{figure}[t] 
	\centering
	\begin{subfigure}
	[Discontinuous Pressure Exact Solution. ]
		{\resizebox{7.8cm}{8.0cm}
			{\includegraphics[scale = 0.15]{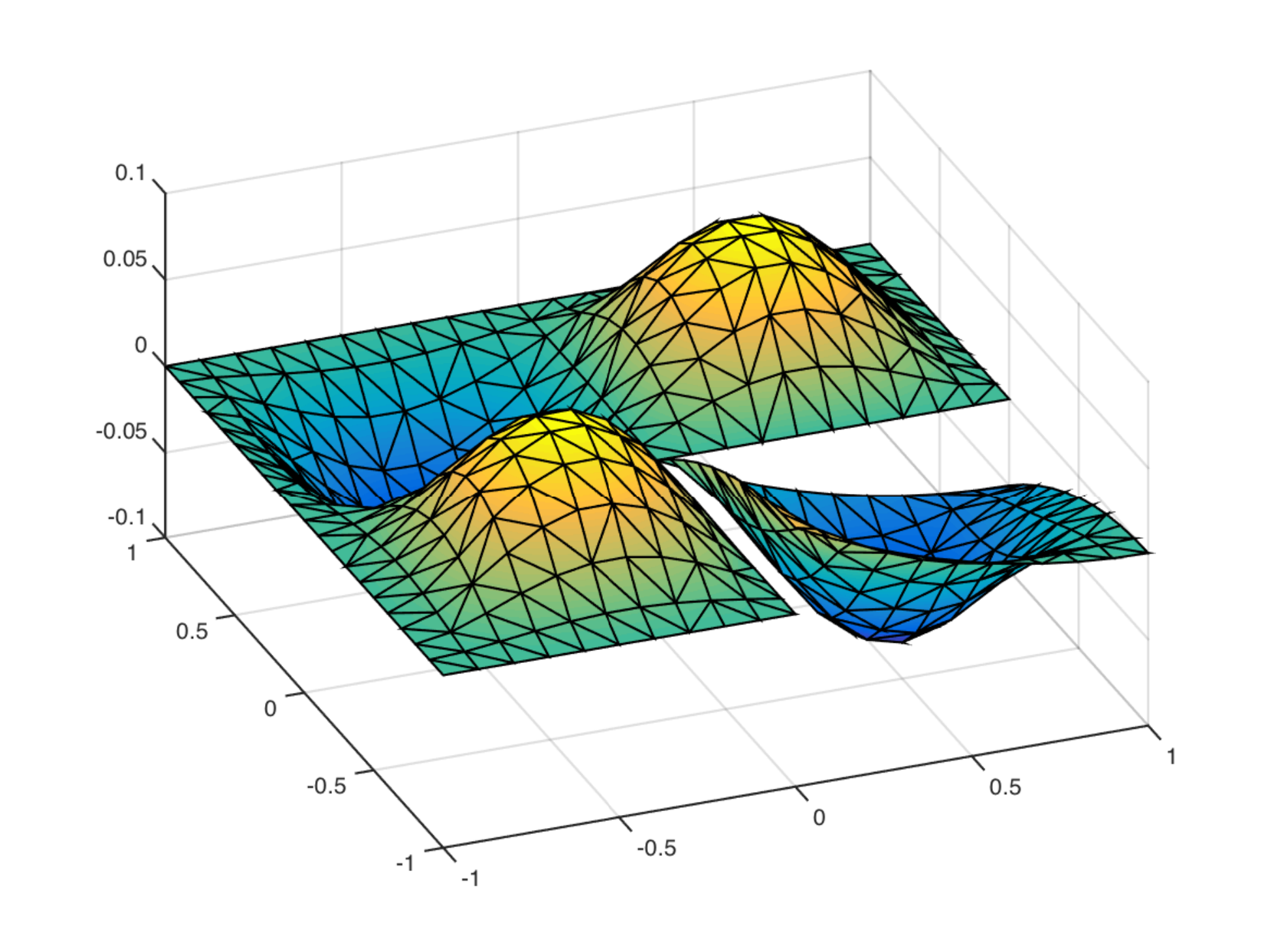} } 
			}
	\end{subfigure} 
	~ 
	\begin{subfigure}[Discontinuous Flux Exact Solution.]
		{\resizebox{7.8cm}{8.0cm}
			{\includegraphics{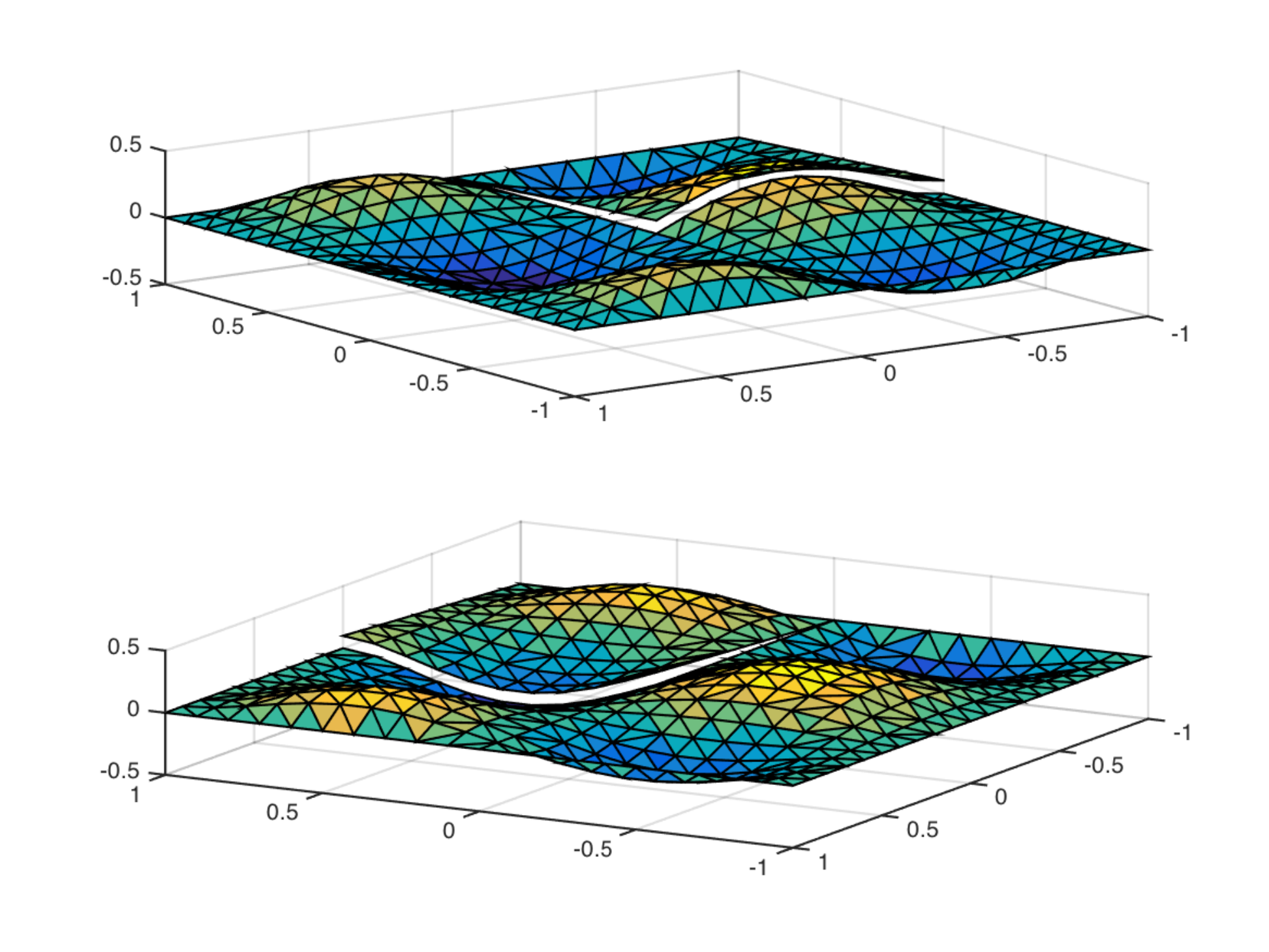} } 
			}                
	\end{subfigure} 
	%
	%
	%
	\caption{\textsc{Example} \ref{Ex Discontinuous Example}, discontinuous exact solution. The visualization angles are different for the pressure and the flux; the choice is made focusing on the jumps of discontinuity. The discontinuities take place on the interface subset $\{0\}\times (-1,0) \cup (0,1)\times \{0\} $ for the pressure $ p $ as well as both components of the velocity $ \u_{x} $, $ \u_{y} $.
	Figure (a) depicts the pressure of the exact solution $p$, see \textsc{Equation} \eqref{Eqn Disc Exact Pressure}.
	Figure (b) depicts the flux of the exact solution 
	$ \u = - \grad p 
	$, see \textsc{Equation} \eqref{Eqn Disc Exact Velocity}. On the upper right corner is depicted the $\boldsymbol{x}$-component while the lower right corner displays the $\boldsymbol{y}$-component. \label{Fig Disc Exact Solution Numerical Example} }
\end{figure}
\begin{figure}[h] 
	\centering
	\begin{subfigure}
	[Pressure Approximate Solution. ]
		{\resizebox{7.8cm}{8.0cm}
			{\includegraphics{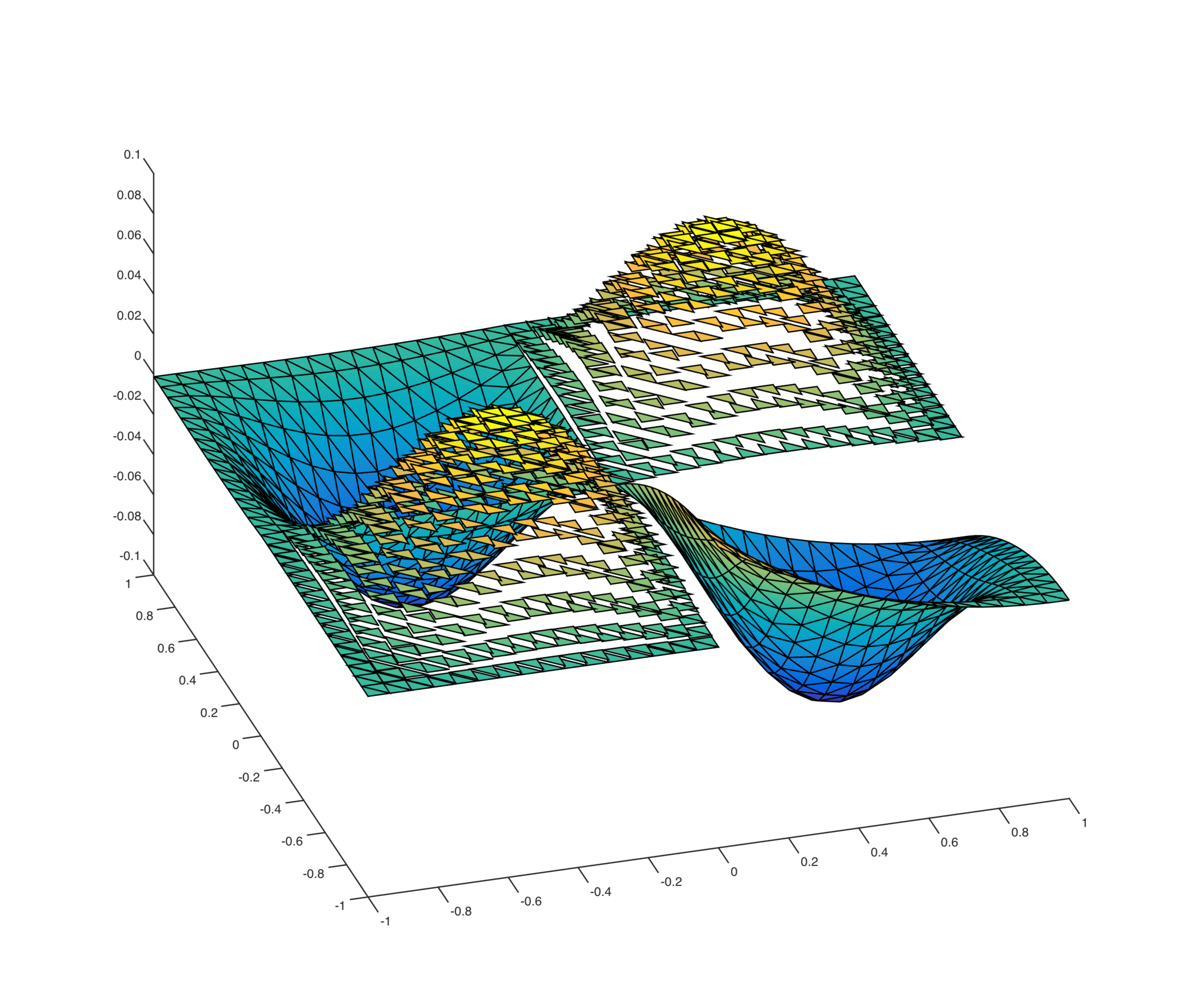} } 
			}
	\end{subfigure} 
	~ 
	\begin{subfigure}[Flux Approximate Solution.]
		{\resizebox{7.8cm}{8.0cm}
			{\includegraphics{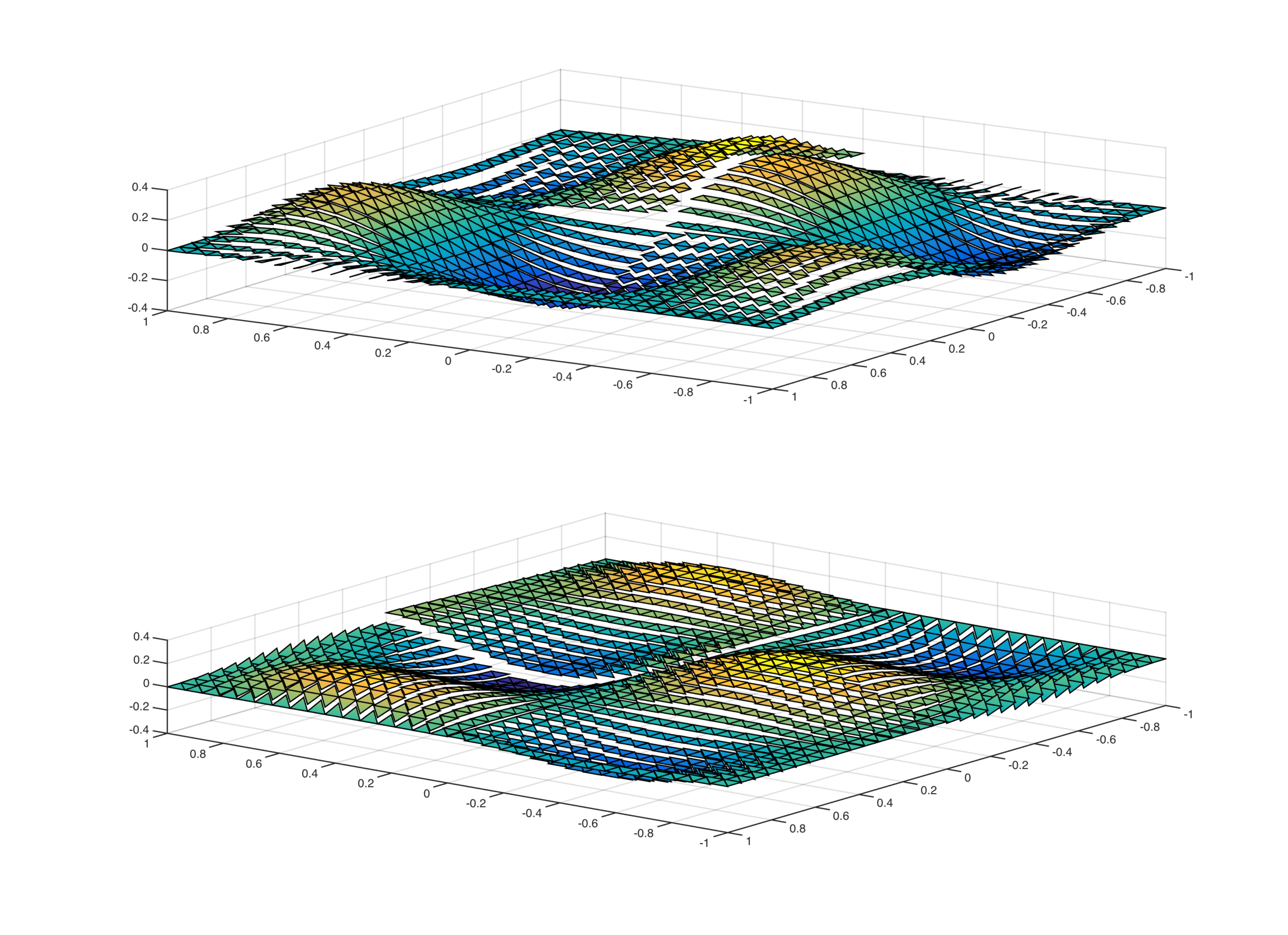} }
			 }                
	\end{subfigure} 
	%
	%
	%
	\caption{\textsc{Example} \ref{Ex Discontinuous Example}, approximate solution for a mesh of size $ h^{-1} = 16 $. The visualization angles are different for the pressure and the flux; the choice is made focusing on the jumps of discontinuity. Notice how the method captures the discontinuities for the pressure $ p $ and both components of the velocity $ \u_{x} $, $ \u_{y} $. The jumps take place on the interface subset $\{0\}\times (-1,0) \cup (0,1)\times \{0\} $.  Figure (a) depicts the pressure $p^{h}$ of the approximate solution.
	Figure (b) depicts the flux of the approximate solution $\u^{h}$. On the upper right corner is depicted the $\boldsymbol{x}$-component of the flux and, on the lower right corner we display 
	the $\boldsymbol{y}$-component of the flux. 
	\label{Fig Disc Approximate Solution Numerical Example} }
\end{figure}
It is direct to see that $[\u, p]$ defined by \eqref{Eqn Disc Exact Solutions}
is the exact solution to the \textsc{Problem} \eqref{Eq porous media strong decomposed} on the geometric domain described by \eqref{Def Geometric Parameters} with the forcing terms defined in \eqref{Eqn Disc Forcing Terms}. Again, the boundary and interface conditions are satisfied.

The convergence results are displayed in the \textsc{Tables} \ref{Table Disc Pressure Approximation} and \ref{Table Disc Velocity Approximation} below. The convergence behavior is virtually identical to the continuous case with observable differences (\textsc{Tables} \ref{Table Pressure Approximation} and \ref{Table Velocity Approximation}) only for the first stages. Consequently, the convergence rate agree with those presented in \textsc{Equation} \eqref{Stmt Numerical Rates of Convergence Ex 1}.
Finally, the numerical solution for $ h^{-1} = 16 $ is depicted in \textsc{Figure} \ref{Fig Disc Approximate Solution Numerical Example}; the choices of grid as well as display angle, were based on optical clarity for the jumps across the interface.
%
\begin{table}[h!]
\caption{Pressures Convergence Table, \textsc{Example} \ref{Ex Discontinuous Example}}\label{Table Disc Pressure Approximation}
\def\arraystretch{1.4}
\rowcolors{2}{gray!25}{white}
\begin{center}
\begin{tabular}{ c c c c c c c }
    \hline
    \rowcolor{gray!50}
$h^{-1}$  
& $\Vert  \poneh- p_{1}  \Vert_{ 0, \Omega_{1} } $ 
& $r$  
& $\Vert  \ptwoh- p_{2}  \Vert_{0, \Omega_{2} }$ 
& $r$
& $\Vert  \ptwoh- p_{2}  \Vert_{1, \Omega_{2} }$ 
& $r$ \\ 
    \toprule
$ 1 $ &   0.9984  & &  1.6520  & &  2.5140  & \\
$ 2 $  &  0.0261  & 5.2575 &  0.0721  & 4.5181   &  0.1980  & 3.6664  \\
$ 4 $ &   0.0091  & 1.5201 &  0.0226  & 1.6737    &  0.0889  & 1.1552  \\
$ 8 $  &  0.0026  & 1.8074 &  0.0062  & 1.8660    &  0.0423  & 1.0715  \\ 
$ 16 $  &  0.0007  & 1.8931 &  0.0016  & 1.9542   &  0.0209 & 1.0172   \\ 
$ 32 $  &  0.0002  &  1.8074  &  0.0004  & 2.0000  &  0.0104   & 1.0069  \\ 
    \hline
\end{tabular}
\end{center}
\end{table}
%
%
\begin{table*}[h!]
\caption{Velocities Convergence Table, \textsc{Example} \ref{Ex Discontinuous Example}}\label{Table Disc Velocity Approximation}
\def\arraystretch{1.4}
\rowcolors{2}{gray!25}{white}
\begin{center}
\begin{tabular}{ c c c c c c c }
    \hline
    \rowcolor{gray!50}
$ h^{-1} $  
& $\Vert  \uone^{h}- \uone  \Vert_{ 0, \Omega_{1} } $ 
& $r$  
& $\Vert  \uone^{h}- \uone  \Vert_{\Hdiv(\Omega_{1})}$ 
& $r$
& $\Vert  \utwo^{h}- \utwo  \Vert_{ 0, \Omega_{2} } $ 
& $r$ \\ 
    \toprule
$ 1 $ &   19.8825 & &  19.8825  & &   9.8017  & \\
$ 2 $  &   0.1409  & 7.1407  &  0.1409  & 7.1407  &   0.1844 & 5.7321 \\
$ 4 $ &   0.0617  & 1.1913  &  0.0617  & 1.1913  &  0.0860  & 1.1004 \\
$ 8 $  &  0.0302  & 1.0307  &  0.0302  & 1.0307  &  0.0418 & 1.0408 \\ 
$ 16 $  &  0.0150 & 1.0096  &   0.0150  & 1.0096  &  0.0208 & 1.0069 \\ 
$ 32 $  &  0.0075  & 1.0000  &  0.0075   & 1.0000 &  0.0104  & 1.0000 \\ 
    \hline
\end{tabular}
\end{center}
\end{table*}
\end{example}
\begin{example}\label{Ex Multiscale Continuous Example}
The purpose of the present example is to illustrate how the method handles problems with flux discontinuities across the interfaces. Such discontinuities occur because the flow resistance coefficient $ a(\cdot) $, has different orders of magnitude within regions $ \Omega_{1} $ and $ \Omega_{2} $. 
For clarity of exposition we use the same pressure as in \textsc{Example} \ref{Ex Continuous Example}, i.e., the exact solution, see \textsc{Figure} \ref{Fig Multiscale Exact Solution Numerical Example}, is given by
\begin{subequations}\label{Eqn Mult Cont Exact Solutions}

\begin{align}\label{Eqn Mult Cont Exact Pressure}
& p: \Omega \rightarrow \R \, ,&
& p(x, y) = x \, y \, (x -1)^{2} (y - 1)^{2}(x + 1)^{2}(y + 1)^{2},  
\end{align}
\begin{align}\label{Eqn Mult Cont Exact Velocity}
& \u: \Omega\rightarrow \R^{2} \, , &
& \u(x, y) = -\frac{1}{a(x, y)}\,\grad p(x,y).
\end{align}
Here, the flow resistance coefficient is defined as
\begin{equation}\label{Eqn Mult Cont Coefficient}
a(x, y) \defining \ind_{ \Omega_{1} } (x, y) + 5 \ind_{ \Omega_{2} }(x, y) ,
\end{equation}
\end{subequations}
in particular, it satisfies \textsc{Hypothesis} \ref{Hyp non null local storage coefficient}. The forcing terms are
\begin{subequations}\label{Eqn Mult Cont Forcing Terms}
\begin{equation}\label{Eqn Int Mult Cont Forcing Terms}
\begin{split}
\g: \Omega \rightarrow \R^{2}, \qquad\qquad\qquad &
\g = \boldsymbol{0} ,  \\
F: \Omega \rightarrow \R , \qquad\qquad\qquad &
F = -\div \frac{1}{a}\, \grad p ,  
\end{split}
\end{equation}
\begin{equation}\label{Eqn Interface Mult Cont Forcing Terms}
\begin{split}
\stress &: \Gamma \rightarrow \R , 
\quad
\stress (x, y) = 0 ,\\ 
\flux &: \Gamma \rightarrow \R , 
\quad
\flux (x,y)  = 
\frac{4}{5}\,x\,(x^{2} - 1)^{2}\ind_{ (-1,1)\times\{0\} }  
+  \frac{4}{5}\,y\,(y^{2} - 1)^{2} \ind_{\{0\} \times (-1,1)}  
\end{split}
\end{equation}
\end{subequations}
\begin{figure}[t] 
	\centering
	\begin{subfigure}
	[Pressure Exact Solution.]
		{\resizebox{7.8cm}{8.0cm}
			{\includegraphics[scale = 0.33]{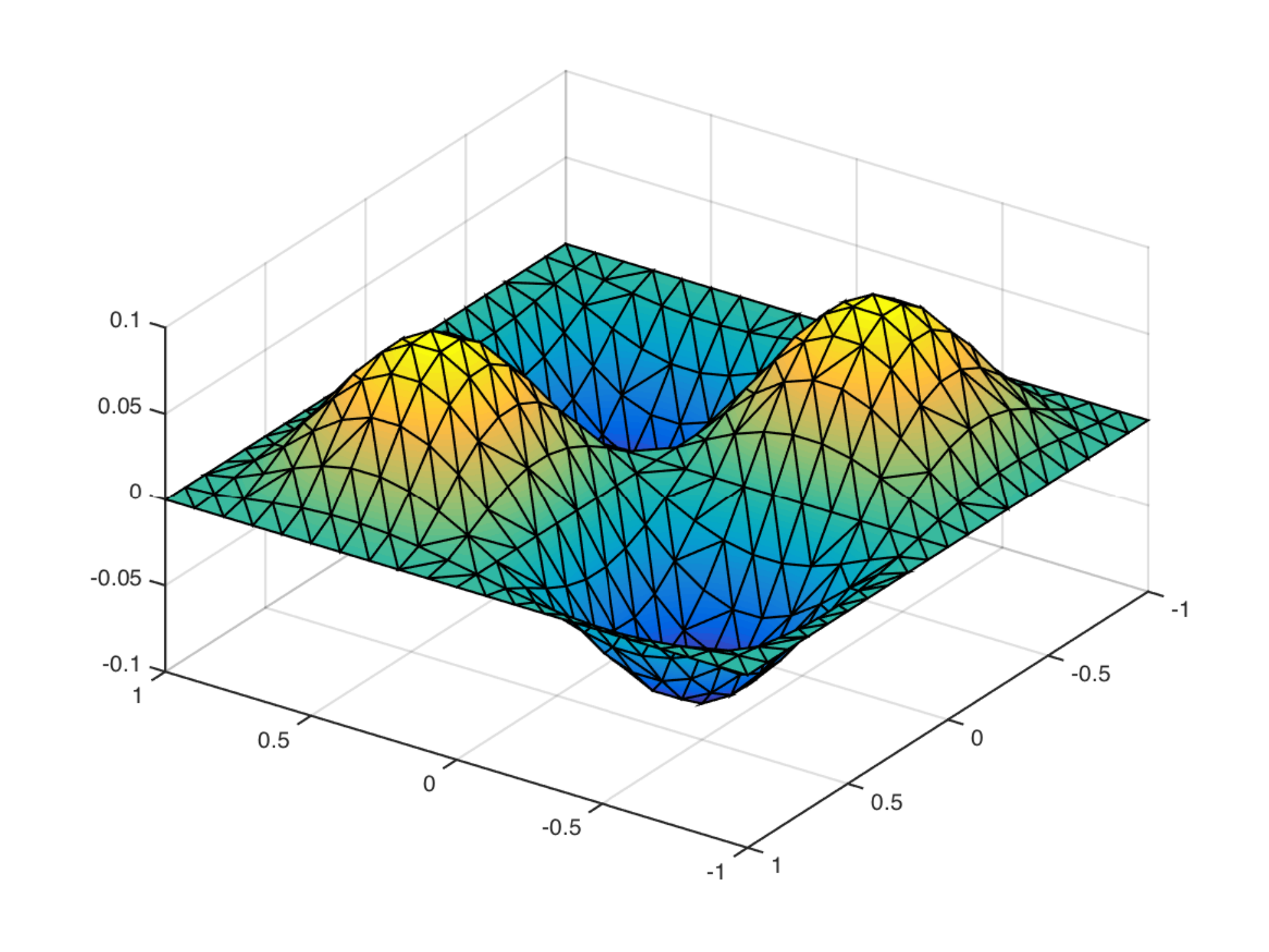} } 
			}
	\end{subfigure} 
	~ 
	\begin{subfigure}[Flux Exact Solution.]
		{\resizebox{7.8cm}{8.0cm}
			{\includegraphics{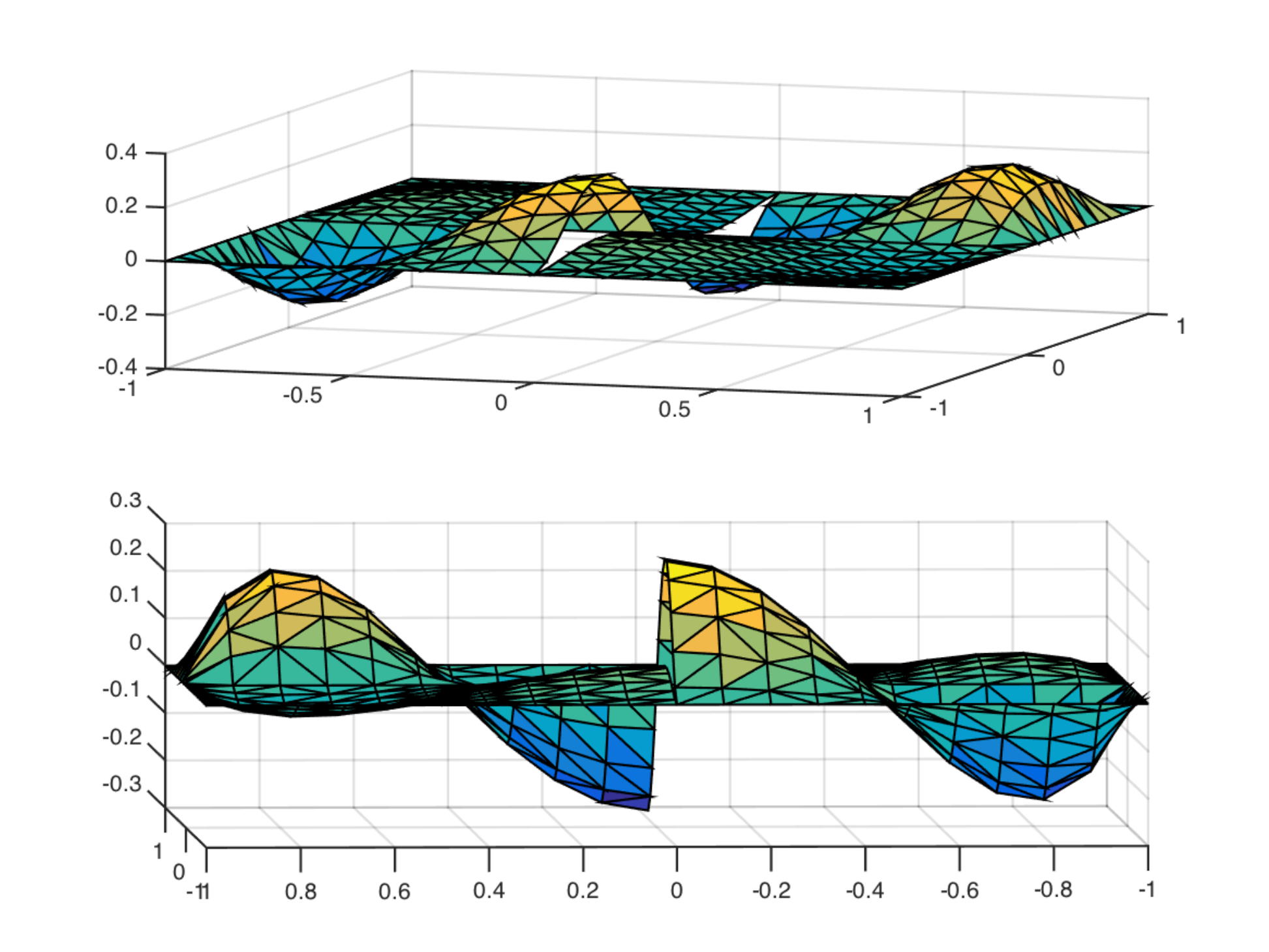} } 
			}                
	\end{subfigure} 
	%
	%
	%
	\caption{\textsc{Example} \ref{Ex Multiscale Continuous Example}. Figure (a) depicts the pressure of the exact solution $ p(x, y) =  x \, y \, (x -1)^{2} (y - 1)^{2}(x + 1)^{2}(y + 1)^{2} $, see \textsc{Equation} \eqref{Eqn Exact Pressure}.
	Figure (b) depicts the flux of the exact solution 
	$ \u = -a^{-1}\,\grad p 
	$, see \textsc{Equation} \eqref{Eqn Exact Velocity}. On the upper right corner is depicted the $\boldsymbol{x}$-component while the lower right corner displays the $\boldsymbol{y}$-component. Here, discontinuities occur only for the velocity due to the flow resistance coefficient $a(\cdot) $, see \textsc{Equation} \eqref{Eqn Mult Cont Coefficient}. The velocity's $ x $-component $ \u_{x} $ has a jump across $ \{0\} \times(-1, 1) $ while the $ y $-component $ \u_{y} $ jumps across $ (-1, 1) \times \{0\} $, see \textsc{Equation} \eqref{Eqn Interface Mult Cont Forcing Terms} for the jumps' exact algebraic expression. \label{Fig Multiscale Exact Solution Numerical Example} }
\end{figure}
\begin{figure}[h] 
	\centering
	\begin{subfigure}
	[Pressure Approximate Solution. ]
		{\resizebox{7.8cm}{8.0cm}
			{\includegraphics{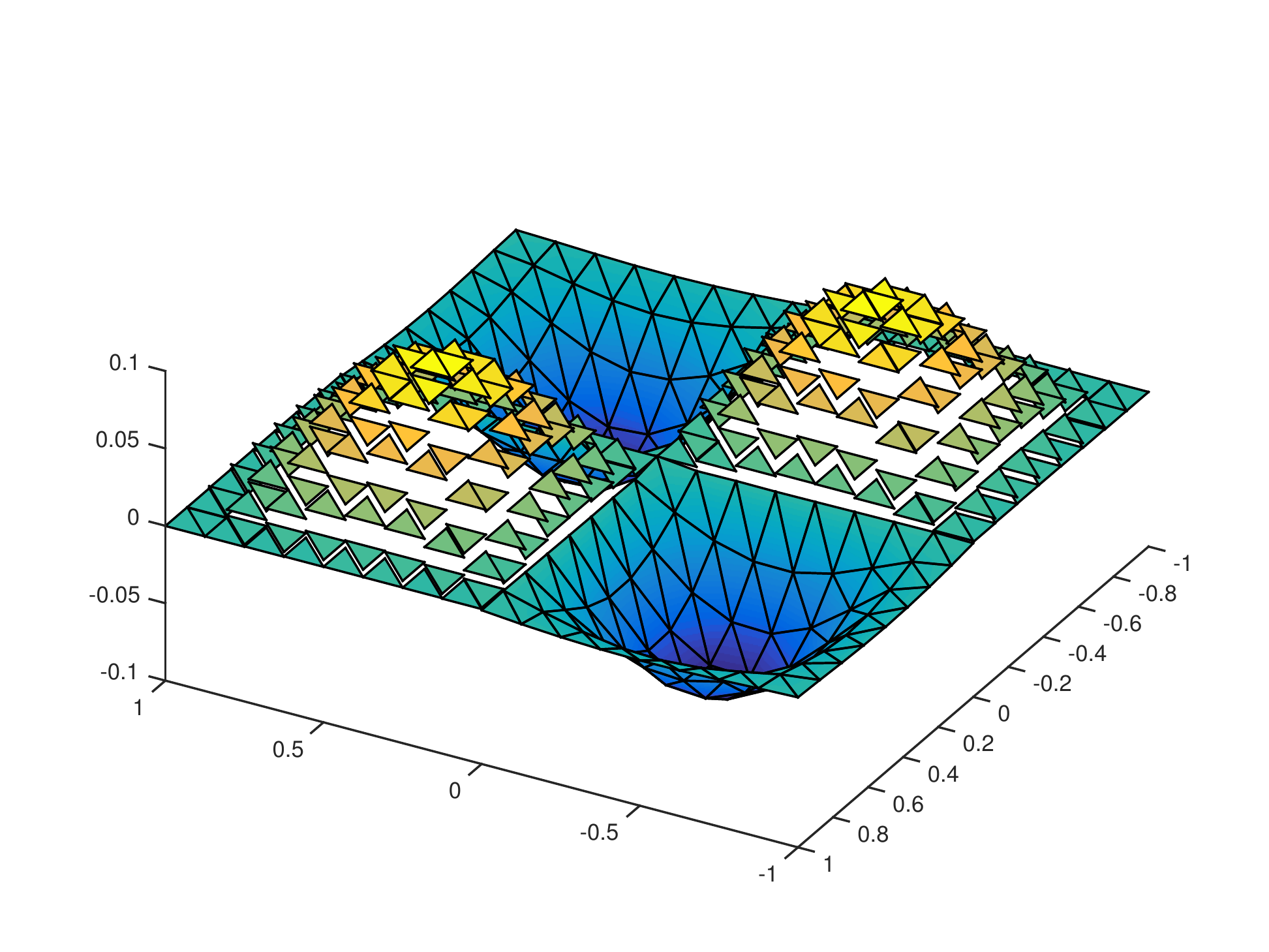} } 
			}
	\end{subfigure} 
	~ 
	\begin{subfigure}[Flux Approximate Solution.]
		{\resizebox{7.8cm}{8.0cm}
			{\includegraphics[scale = 0.33]{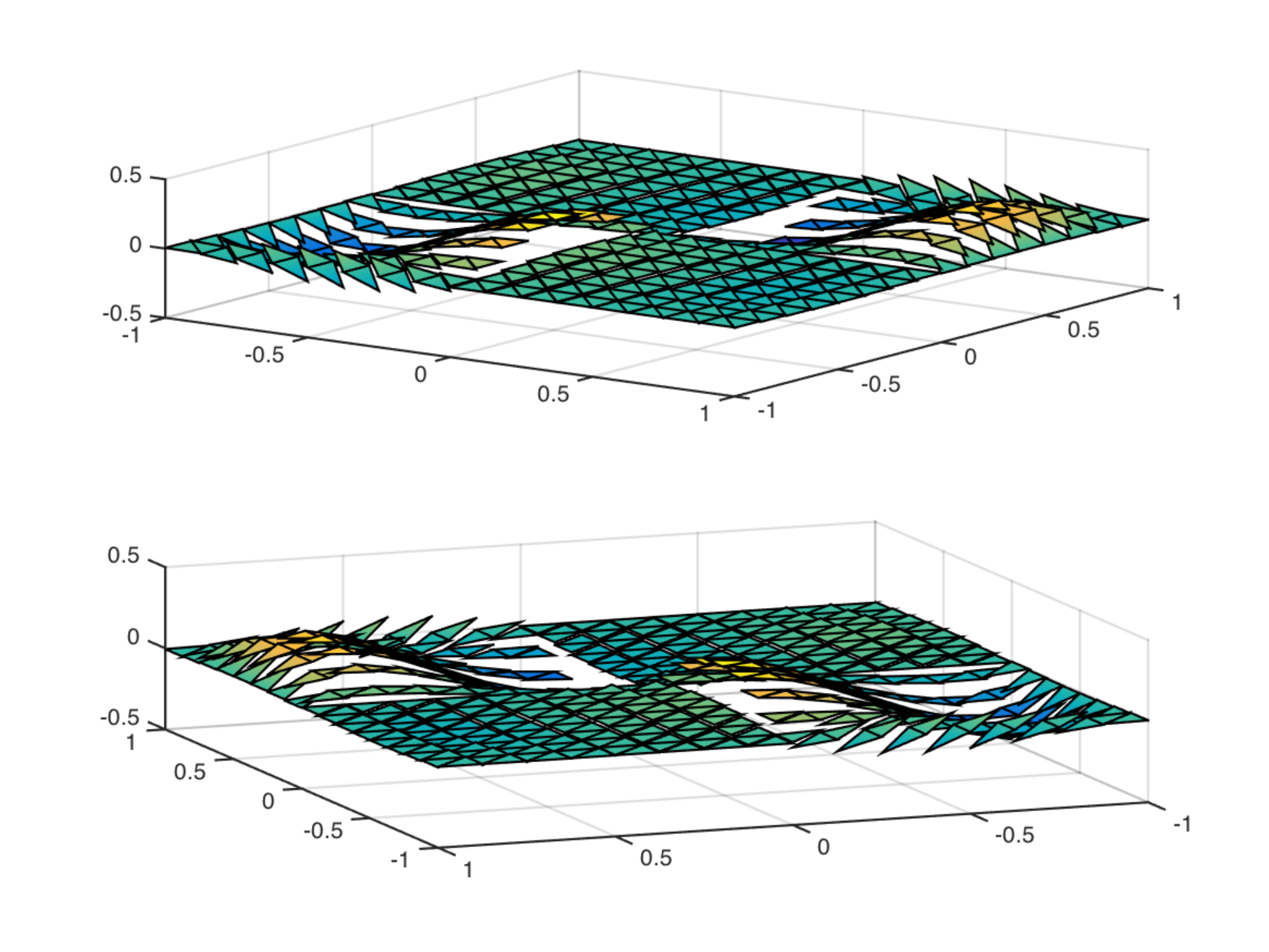} } 
			}                
	\end{subfigure} 
	%
	%
	%
	\caption{\textsc{Example} \ref{Ex Multiscale Continuous Example}, approximate solution for a mesh of size $ h^{-1} = 8 $. The sub-domains are $\Omega_{1} = (-1, 0) \times (-1, 0)\cup  (0, 1) \times  (0, 1)$ and $\Omega_2 =  (-1, 0) \times (0, 1)\cup  (0, 1) \times  (-1, 0)$, see \textsc{Identity} \eqref{Def Geometric Parameters}.  Figure (a) depicts the pressure $p^{h}$ of the approximate solution, it is piecewise constant on the domain $\Omega_{1} $ and piecewise linear affine on the domain $\Omega_2$.
	Figure (b) depicts the flux of the approximate solution $\u^{h}$. On the upper right corner is depicted the $\boldsymbol{x}$-component of the flux, which is continuous across \textbf{horizontal edges} of $\Omega_{1}$ and piecewise constant on the domain $\Omega_2$.	
	On the lower right corner it is displayed
	the $\boldsymbol{y}$-component of the flux, which is continuous across \textbf{vertical edges} of $\Omega_{1}$ and piecewise constant on the domain $\Omega_2$. Observe that the jumps across the interface are captured for both components of the velocity, $ \u_{x} $ has a jump across $ \{0\} \times(-1, 1) $ while $ \u_{y} $ jumps across $ (-1, 1) \times \{0\} $. \label{Fig Multiscale Cont Approximate Solution Numerical Example} }
\end{figure}
A direct calculation shows that $[\u, p]$ defined by \eqref{Eqn Mult Cont Exact Solutions}
is the exact solution of \textsc{Problem} \eqref{Eq porous media strong decomposed}, on the geometric domain described by \eqref{Def Geometric Parameters} with the forcing terms defined in \eqref{Eqn Mult Cont Forcing Terms}. The flux jump $ \flux (x,y) $ occurs because of the jump in the flow resistance coefficient $ a(\cdot) $ described in  \eqref{Eqn Mult Cont Coefficient}; should $ a(\cdot) $ be a continuous function the interface flux term would be null i.e., $ \flux (x,y)  \equiv 0 $. Once more, the boundary conditions \eqref{Eq Drained Condition decomposed}, \eqref{Eq Non-Flux Condition decomposed} and the interface exchange conditions \eqref{Eq normal flux balance}, \eqref{Eq normal stress balance} are satisfied.

The convergence results are displayed in the \textsc{Tables} \ref{Table Mult Pressure Approximation} and \ref{Table Mult Velocity Approximation} below. While the velocity's behavior is the expected one (it agrees with \eqref{Stmt Velocity Numerical Rates of Convergence Ex 1}), the pressure shows mild differences with \eqref{Stmt Pressure Numerical Rates of Convergence Ex 1}
%
%
%
\begin{align}\label{Stmt Pressure Numerical Rates of Convergence Ex 3}
& \Vert \poneh - \pone \Vert_{0, \Omega_{1}} = \mathcal{O}(h^{1.6}) ,&
& \Vert \ptwoh - \ptwo \Vert_{0, \Omega_{2}} = \mathcal{O}(h^{2.1}) , &
& \Vert \ptwoh - \ptwo \Vert_{1, \Omega_{2}} = \mathcal{O}(h^{1}) 
\end{align}
%
%
Finally, the numerical solution for $ h^{-1} =8 $ is depicted in \textsc{Figure}  \ref{Fig Multiscale Cont Approximate Solution Numerical Example}; the choices of grid and display angle were based on optical clarity to illustrate the pressure of \textsc{Example} \ref{Ex Continuous Example} from a different point of view and to get a neat picture of the flux jumps across the interface. 
%
\begin{table}[h!]
\caption{Pressures Convergence Table, \textsc{Example} \ref{Ex Multiscale Continuous Example}}\label{Table Mult Pressure Approximation}
\def\arraystretch{1.4}
\rowcolors{2}{gray!25}{white}
\begin{center}
\begin{tabular}{ c c c c c c c }
    \hline
    \rowcolor{gray!50}
$ h^{-1} $  
& $\Vert  \poneh- p_{1}  \Vert_{ 0, \Omega_{1} } $ 
& $r$  
& $\Vert  \ptwoh- p_{2}  \Vert_{0, \Omega_{2} }$ 
& $r$
& $\Vert  \ptwoh- p_{2}  \Vert_{1, \Omega_{2} }$ 
& $r$ \\ 
    \toprule
$ 1 $ &   0.0246  &   &  0.1252  &  &   0.4376 &   \\
$ 2 $  &  0.0104  &  1.2421   &  0.0464  & 1.4320   & 0.1905  &  1.1998   \\
$ 4 $ &   0.0046  &  1.1769 &  0.0182  & 1.3502    &  0.0903  &   1.0770   \\
$ 8 $  &  0.0013  &    1.8231 &  0.0052  & 1.8074   &  0.0427  &  1.0805    \\ 
$ 16 $  &  0.0003  &     2.1155 &  0.0013  & 2.0000    &  0.0209  &  1.0307 \\ 
$ 32 $  &  0.0001  &     1.5850  &  0.0003  & 2.1155   &  0.0104   &   1.0069  \\ 
    \hline
\end{tabular}
\end{center}
\end{table}
%
%
\begin{table*}[h!]
\caption{Velocities Convergence Table, \textsc{Example} \ref{Ex Multiscale Continuous Example}}\label{Table Mult Velocity Approximation}
\def\arraystretch{1.4}
\rowcolors{2}{gray!25}{white}
\begin{center}
\begin{tabular}{ c c c c c c c }
    \hline
    \rowcolor{gray!50}
$ h^{-1} $  
& $\Vert  \uone^{h}- \uone  \Vert_{ 0, \Omega_{1} } $ 
& $r$  
& $\Vert  \uone^{h}- \uone  \Vert_{\Hdiv(\Omega_{1})}$ 
& $r$
& $\Vert  \utwo^{h}- \utwo  \Vert_{ 0, \Omega_{2} } $ 
& $r$ \\ 
    \toprule
$ 1 $ &   0.1884 &    &   0.1884  &  &   0.2092  & \\
$ 2 $  &   0.1212  &  0.6364   &  0.1212  &  0.6364   &  0.0370   &  2.4993 \\
$ 4 $ &   0.0567  &  1.0960  &  0.0567  &   1.0960  &  0.0177   &  1.0638  \\
$ 8 $  &  0.0292  &   0.9574  &  0.0292  &   0.9574   &   0.0085 &  1.0582   \\ 
$ 16 $  &  0.0148  &  0.9804  &  0.0148  &  0.9804   &  0.0042  &  1.0171   \\ 
$ 32 $  &  0.0074  & 1.0000  &  0.0074   &  1.0000   &  0.0021  &  1.0000   \\ 
    \hline
\end{tabular}
\end{center}
\end{table*}
\end{example}
\begin{example}\label{Ex Multiscale Jump Example}
The purpose of the present example is twofold: illustrate how the method handles problems whose velocities drastically change across the interface but still are continuous functions, this is done in controlled/lab conditions, and suggest a heuristic method to proceed in practice i.e., when the real solution is not known. Such abrupt change takes place because the flow resistance coefficient $ a (\cdot) $, defined in  \textsc{Equation} \eqref{Eqn Mult Cont Coefficient}, has different orders of magnitude within regions $ \Omega_{1} $ and $ \Omega_{2} $. Although the exact solution is continuous on the velocity from the theoretical point of view, because of the multiscaling introduced by $ a( \cdot ) $, it is more convenient/strategic to treat it as discontinuous across the interface as the method does (see \textsc{Figure}  \ref{Fig Multiscale Jump Approximate Solution Numerical Example} below), to avoid numerical instability. In this example, the exact solution is given by (see \textsc{Figure} \ref{Fig Multiscale Jump Exact Solution Numerical Example})
\begin{subequations}\label{Eqn Mult Jump Exact Solutions}
\begin{align}\label{Eqn Mult Jump Exact Pressure}
& p: \Omega \rightarrow \R \, ,&
& p(x, y) = \sin^{2}\Big( \frac{\pi}{2}(x - 1)\Big)
 \sin^{2}\Big( \frac{\pi}{2}(y - 1)\Big),  
\end{align}
\begin{align}\label{Eqn Mult Jump Exact Velocity}
& \u: \Omega\rightarrow \R^{2} \, , &
& \u(x, y) = -\frac{1}{a(x, y)}\,\grad p(x,y),
\end{align}
%
\end{subequations}
Here, the flow resistance coefficient $ a(\cdot) $ is defined by \textsc{Equation} \eqref{Eqn Mult Cont Coefficient}. The forcing terms are
\begin{subequations}\label{Eqn Mult Jump Forcing Terms}
\begin{equation}\label{Eqn Int Mult Jump Forcing Terms}
\begin{split}
\g: \Omega \rightarrow \R^{2}, \qquad\qquad\qquad &
\g = \boldsymbol{0} ,  \\
F: \Omega \rightarrow \R , \qquad\qquad\qquad &
F = -\div \frac{1}{a}\, \grad p ,  
\end{split}
\end{equation}
\begin{equation}\label{Eqn Interface Mult Jump Forcing Terms}
\begin{split}
\stress &: \Gamma \rightarrow \R , 
\qquad\qquad\qquad  
\stress (x, y) = 0 ,\\ 
\flux &: \Gamma \rightarrow \R , \\
\flux (x,y) & = 
- \sin^{2}\Big(\frac{\pi}{2}(x - 1)\Big)\ind_{ (-1,1)\times\{0\} }  
-  \sin^{2}\Big(\frac{\pi}{2}(y - 1)\Big)\ind_{\{0\} \times (-1,1)} 
\end{split}
\end{equation}
\end{subequations}
\begin{figure}[t] 
	\centering
	\begin{subfigure}
	[Pressure Exact Solution.]
		{\resizebox{7.8cm}{8.0cm}
			{\includegraphics[scale = 0.33]{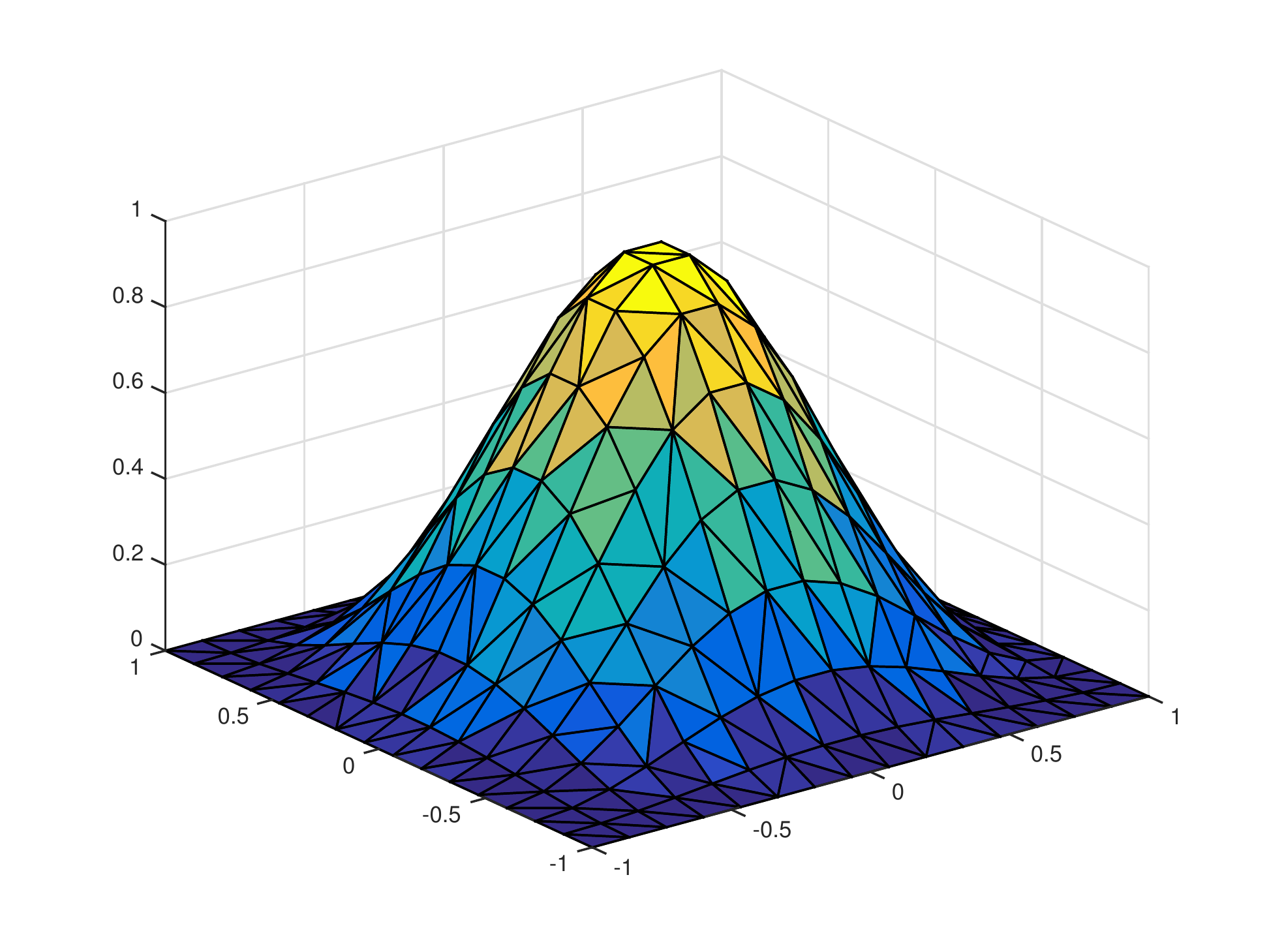} } 
			}
	\end{subfigure} 
	~ 
	\begin{subfigure}[Flux Exact Solution.]
		{\resizebox{7.8cm}{8.0cm}
			{\includegraphics{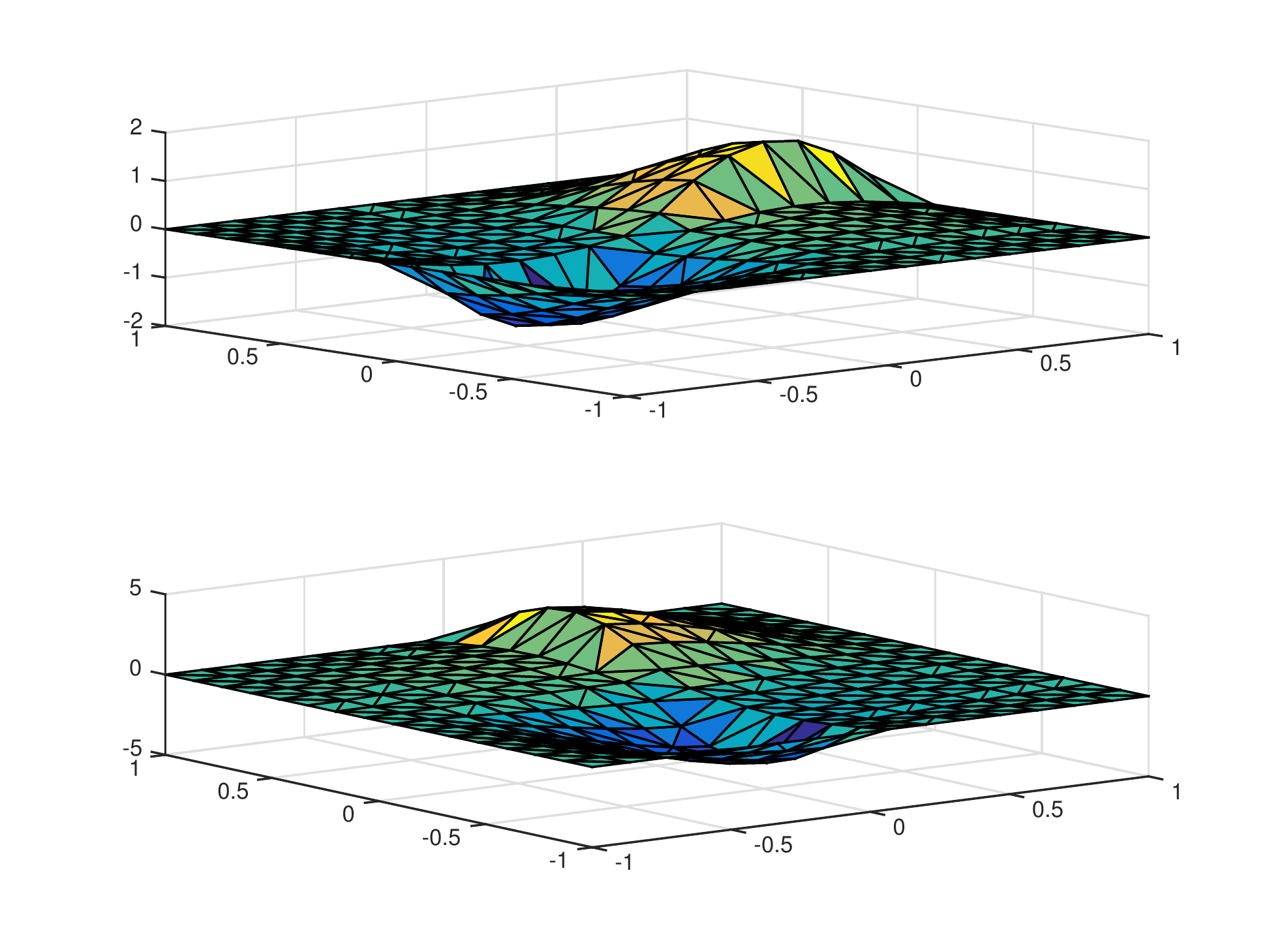} } 
			}                
	\end{subfigure} 
	%
	%
	%
	\caption{\textsc{Example} \ref{Ex Multiscale Jump Example}. Figure (a) depicts the pressure of the exact solution $ p(x, y) = \sin^{2}\big( \frac{\pi}{2}(x - 1)\big)
 \sin^{2}\big( \frac{\pi}{2}(y - 1)\big) $, see \textsc{Equation} \eqref{Eqn Exact Pressure}.
	Figure (b) depicts the flux of the exact solution 
	$ \u = - a^{-1}\grad p 
	$, see \textsc{Equation} \eqref{Eqn Exact Velocity}. On the upper right corner is depicted the $\boldsymbol{x}$-component while the lower right corner displays the $\boldsymbol{y}$-component. Observe the abrupt changes of $ \u_{x} $ across $ \{0\}\times (-1,1) $ and $ \u_{y} $ across $ (-1,1)\times\{0\} $ due to the multiscaling of the flow resistance coefficient $ a(\cdot) $, see \textsc{Equation} \eqref{Eqn Mult Cont Coefficient}. 
	\label{Fig Multiscale Jump Exact Solution Numerical Example} }
\end{figure}
\begin{figure}[h] 
	\centering
	\begin{subfigure}
	[Pressure Approximate Solution. ]
		{\resizebox{7.8cm}{8.0cm}
			{\includegraphics{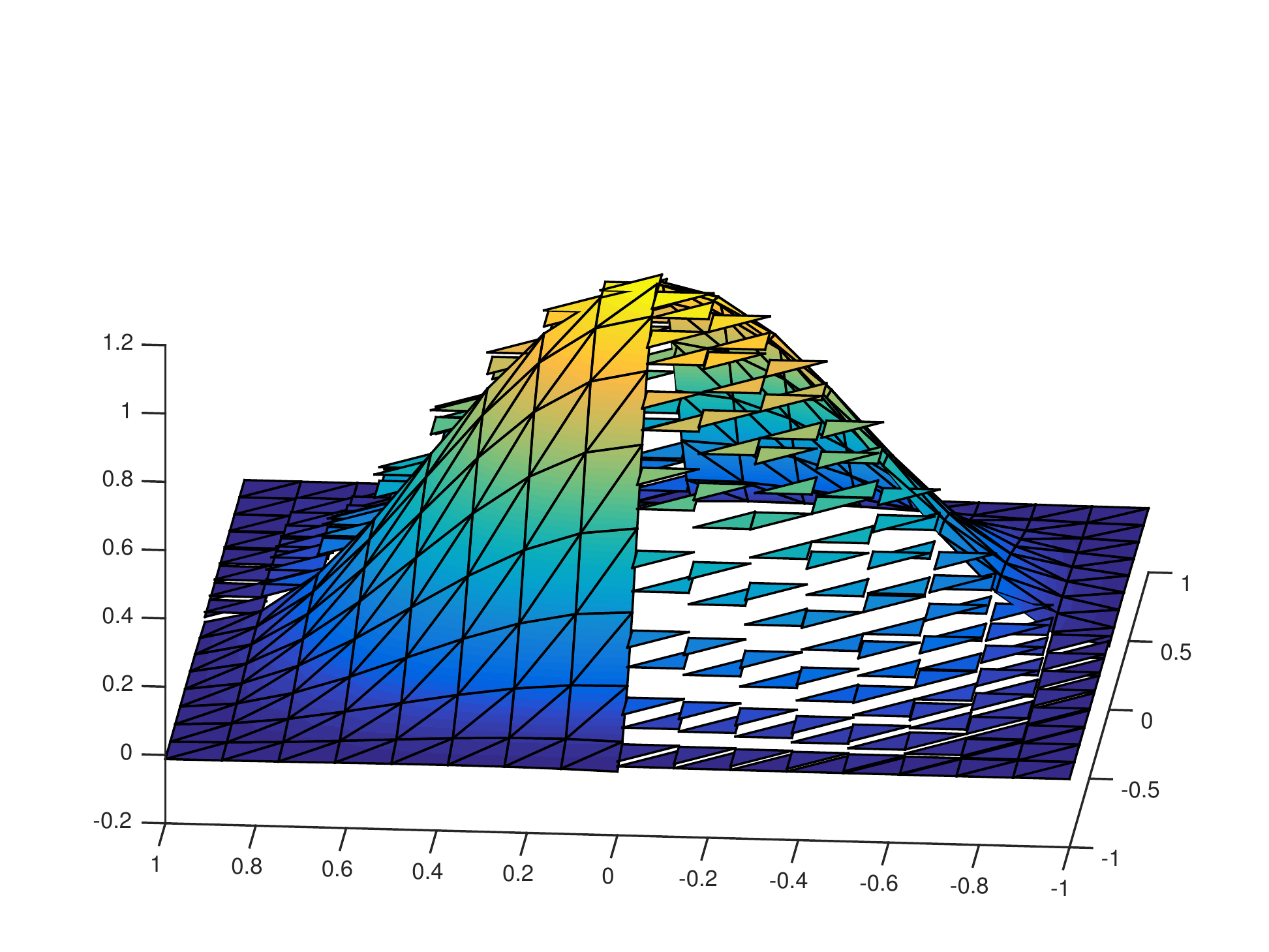} } 
			}
	\end{subfigure} 
	~ 
	\begin{subfigure}[Flux Approximate Solution.]
		{\resizebox{7.8cm}{8.0cm}
			{\includegraphics[scale = 0.33]{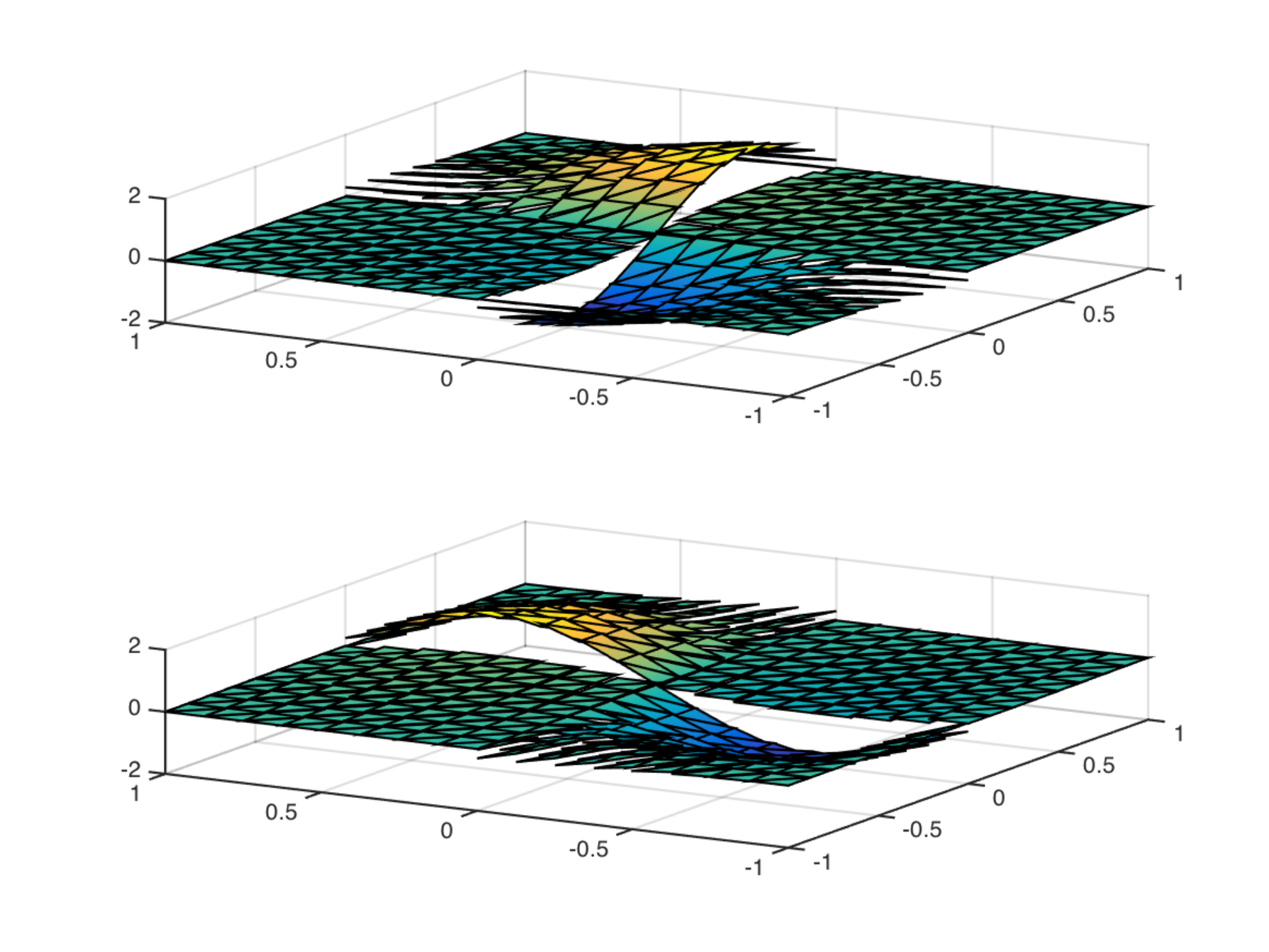} } 
			}                
	\end{subfigure} 
	%
	%
	%
	\caption{\textsc{Example} \ref{Ex Multiscale Jump Example}, approximate solution for a mesh of size $ h^{-1} = 8 $, the sub-domains are $\Omega_{1} = (-1, 0) \times (-1, 0)\cup  (0, 1) \times  (0, 1)$ and $\Omega_2 =  (-1, 0) \times (0, 1)\cup  (0, 1) \times  (-1, 0)$, see \textsc{Identity} \eqref{Def Geometric Parameters}.  Figure (a) depicts the pressure $p^{h}$ of the approximate solution, it is piecewise constant on the domain $\Omega_{1} $ and piecewise linear affine on the domain $\Omega_2$.
	Figure (b) depicts the flux of the approximate solution $\u^{h}$. On the upper right corner is depicted the $\boldsymbol{x}$-component of the flux 
	, which is continuous across \textbf{horizontal edges} of $\Omega_{1}$ and piecewise constant on the domain $\Omega_2$.	
	On the lower right corner we display 
	the $\boldsymbol{y}$-component of the flux, which is continuous across \textbf{vertical edges} of $\Omega_{1}$ and piecewise constant on the domain $\Omega_2$. Notice that the abrupt changes of $ \u_{x} $ across $ \{0\}\times (-1,1) $ and $ \u_{y} $ across $ (-1,1)\times\{0\} $ 
	in the exact solution (see \textsc{Figure} \ref{Fig Multiscale Jump Exact Solution Numerical Example}), are now understood as numerical jumps. \label{Fig Multiscale Jump Approximate Solution Numerical Example} }
\end{figure}
A direct calculation shows that $[\u, p]$ defined by \eqref{Eqn Mult Cont Exact Solutions}
is the exact solution of the \textsc{Problem} \eqref{Eq porous media strong decomposed} on the geometric domain described by \eqref{Def Geometric Parameters} with the forcing terms defined in \eqref{Eqn Mult Cont Forcing Terms}. Once more, the boundary conditions \eqref{Eq Drained Condition decomposed}, \eqref{Eq Non-Flux Condition decomposed} and the interface exchange conditions \eqref{Eq normal flux balance}, \eqref{Eq normal stress balance} are satisfied. 

The interface normal flux forcing term satisfies $ \flux (x,y)  = - \beta(\cdot)\, p\big\vert_{\Gamma} $, for $ \beta(\cdot) \equiv 1 $ (in particular, \textsc{Hypothesis} \ref{Hyp non null local storage coefficient} is verified). Then, the interface normal flux balance condition \textsc{Equation} \eqref{Eq normal flux balance} implies $ \uone\cdot\n\big\vert_{\Gamma} = \utwo\cdot\n\big\vert_{\Gamma} $. Hence, no flux jumps occur despite the change in the order of magnitude between regions, which comes from the flow resistance coefficient introduced in  \eqref{Eqn Mult Cont Coefficient}. 

The convergence results are displayed in the \textsc{Tables} \ref{Table Mult Jump Pressure Approximation} and \ref{Table Mult Jump Velocity Approximation} below. Again, the velocity's behavior agrees with \eqref{Stmt Velocity Numerical Rates of Convergence Ex 1} as expected. However, the pressure $ L^{2}(\Omega_{1}) $-norm differs significantly from the expected one while its $ L^{2}(\Omega_{2}) $-norm differs mildly from the expected one 
\begin{align}\label{Stmt Pressure Numerical Rates of Convergence Ex 4}
& \Vert \poneh - \pone \Vert_{0, \Omega_{1}} = \mathcal{O}(h^{2.2}) ,&
& \Vert \ptwoh - \ptwo \Vert_{0, \Omega_{2}} = \mathcal{O}(h^{2.1}) , &
& \Vert \ptwoh - \ptwo \Vert_{1, \Omega_{2}} = \mathcal{O}(h) . 
\end{align}
The numerical solution for $ h^{-1} = 8 $ is depicted in \textsc{Figure} \ref{Fig Multiscale Jump Approximate Solution Numerical Example}; the choices of grid and display angle were based on optical clarity to illustrate both: the nature of discrete solution and the flux numerical jumps across the interfaces. 
%
\begin{table}[h!]
\caption{Pressures Convergence Table, \textsc{Example} \ref{Ex Multiscale Jump Example}, $ \flux = - p\big\vert_{\Gamma} $}\label{Table Mult Jump Pressure Approximation}
\def\arraystretch{1.4}
\rowcolors{2}{gray!25}{white}
\begin{center}
\begin{tabular}{ c c c c c c c }
    \hline
    \rowcolor{gray!50}
$ h^{-1} $  
& $\Vert  \poneh- p_{1}  \Vert_{ 0, \Omega_{1} } $ 
& $r$  
& $\Vert  \ptwoh- p_{2}  \Vert_{0, \Omega_{2} }$ 
& $r$
& $\Vert  \ptwoh- p_{2}  \Vert_{1, \Omega_{2} }$ 
& $r$ \\ 
    \toprule
$ 1 $ &   0.0624  &   &  0.2574  &  &   1.1991 &   \\
$ 2 $  &  0.0513  &  0.2826  &  0.0746  & 1.7868   & 0.5451  &  1.1374   \\
$ 4 $ &   0.0143  &  1.8429 &  0.0245  & 1.6064    &  0.2799  &   0.9616   \\
$ 8 $  &  0.0037  &    1.9504 &  0.0068  & 1.8492   &  0.1376  &  1.0244    \\ 
$ 16 $  &  0.0009  &     2.0395 &  0.0017  & 2.0000    &  0.0682  &  1.0126 \\ 
$ 32 $  &  0.0002  &     2.1699  &  0.0004  & 2.0875   &  0.0340   &   1.0042  \\ 
    \hline
\end{tabular}
\end{center}
\end{table}
%
%
\begin{table*}[h!]
\caption{Velocities Convergence Table, \textsc{Example} \ref{Ex Multiscale Jump Example}, $ \flux = - p\big\vert_{\Gamma} $}\label{Table Mult Jump Velocity Approximation}
\def\arraystretch{1.4}
\rowcolors{2}{gray!25}{white}
\begin{center}
\begin{tabular}{ c c c c c c c }
    \hline
    \rowcolor{gray!50}
$ h^{-1} $  
& $\Vert  \uone^{h}- \uone  \Vert_{ 0, \Omega_{1} } $ 
& $r$  
& $\Vert  \uone^{h}- \uone  \Vert_{\Hdiv(\Omega_{1})}$ 
& $r$
&  $\Vert  \utwo^{h}- \utwo  \Vert_{ 0, \Omega_{2} } $
& $r$ \\ 
    \toprule
$ 1 $ &   1.0801 &    &   1.0801 &  &   0.3538  & \\
$ 2 $  &   0.3688  &  1.5503   &  0.3688  &  1.5503   &  0.1080   &  1.7119 \\
$ 4 $ &   0.2129  &  0.7927  &  0.2129  &   0.7927  &  0.0558   &  0.9527 \\
$ 8 $  &  0.1125  &   0.9203  &  0.1125  &   0.9203   &   0.0275 &  1.0208   \\ 
$ 16 $  &  0.0571  &  0.9784  &  0.0571  &  0.9784   &  0.0136  &  1.0158   \\ 
$ 32 $  &  0.0287  & 0.9924  &  0.0287   &  0.9924   &  0.0068  &  1.0000   \\ 
    \hline
\end{tabular}
\end{center}
\end{table*}

Next, we present an alternative analysis for the same case. In practice the exact solution is not known, only the forcing terms, namely $ F $, $ \g $ from \textsc{Equation} \eqref{Eqn Int Mult Jump Forcing Terms} but the pressure is not known at the interface i.e., we ignore the normal flux term $ \flux = - p\big\vert_{\Gamma} $. However, this term can be introduced after the first iteration to correct it. In our next numerical experiment the normal flux term in \textsc{Equation} \ref{Eqn Interface Mult Jump Forcing Terms} is replaced by
\begin{equation}\label{Eqn Interface Mult Jump Error Forcing Terms}
%
\flux (x,y) = - \frac{1}{\sqrt{2} } \int_{\Gamma}
\bigg(\sin^{2}\Big(\frac{\pi}{2}(x - 1)\Big)\ind_{ (-1,1)\times\{0\} }  
+  \sin^{2}\Big(\frac{\pi}{2}(y - 1)\Big)\ind_{\{0\} \times (-1,1)}\bigg) \,dS\\
= - \frac{1}{\sqrt{2} }.
%
\end{equation}
The integral above indicates line integral along the interface $ \Gamma$. Notice that this is the first Fourier coefficient of the normal flux term across the interface i.e., the $ L^{2}(\Gamma) $-orthogonal projection of $ \flux $ onto the subspace of constant functions. The numerical solution for $ h^{-1} = 8 $ for this case is displayed in \textsc{Figure} \ref{Fig Multiscale Jump Approximate Solution Numerical Example}; the choices of grid and display angle were based on optical clarity to highlight the errors that the numerical solution contains, both pressure and velocity due to $ \flux $, as well as the flux numerical jumps across the interfaces. The approximation norms are summarized in \textsc{Tables} \ref{Table Mult Jump Error Pressure Approximation} and \ref{Table Mult Jump Error Velocity Approximation}. Clearly, in this case, the convergence rate analysis is pointless since the numerical solution will not converge to the exact solution. However, it makes sense to compute the percentage relative errors in order to have a measure of the attained accuracy.
The relative errors are written on the column to the right of their corresponding absolute errors, as it can be seen after a few steps, the percentage error tends to contract by a half, i.e., $ \mathcal{O}(h) $.  
\begin{figure}[h] 
	\centering
	\begin{subfigure}
	[Pressure Approximate Solution. ]
		{\resizebox{7.8cm}{8.0cm}
			{\includegraphics{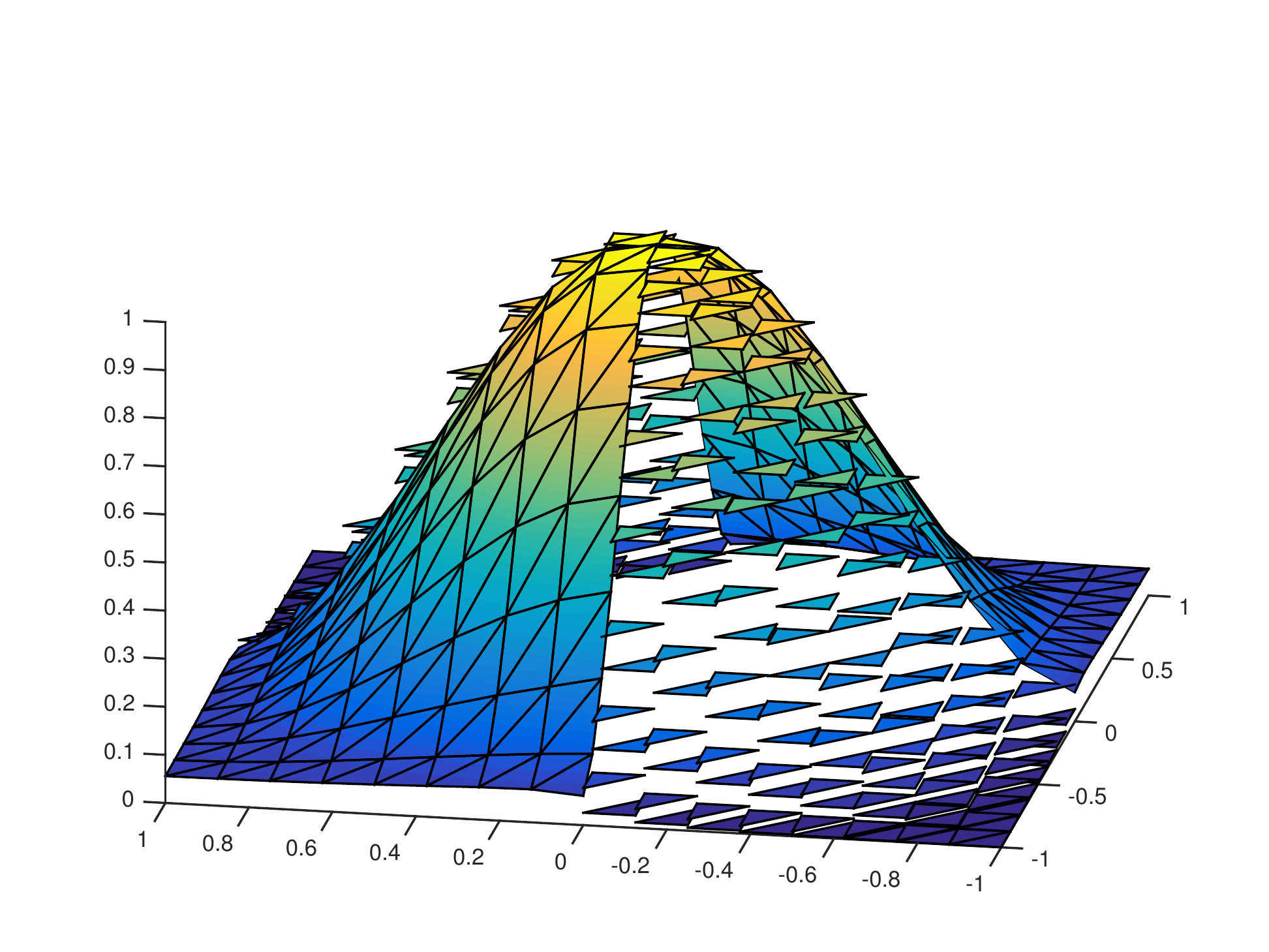} } 
			}
	\end{subfigure} 
	~ 
	\begin{subfigure}[Flux Approximate Solution.]
		{\resizebox{7.8cm}{8.0cm}
			{\includegraphics[scale = 0.33]{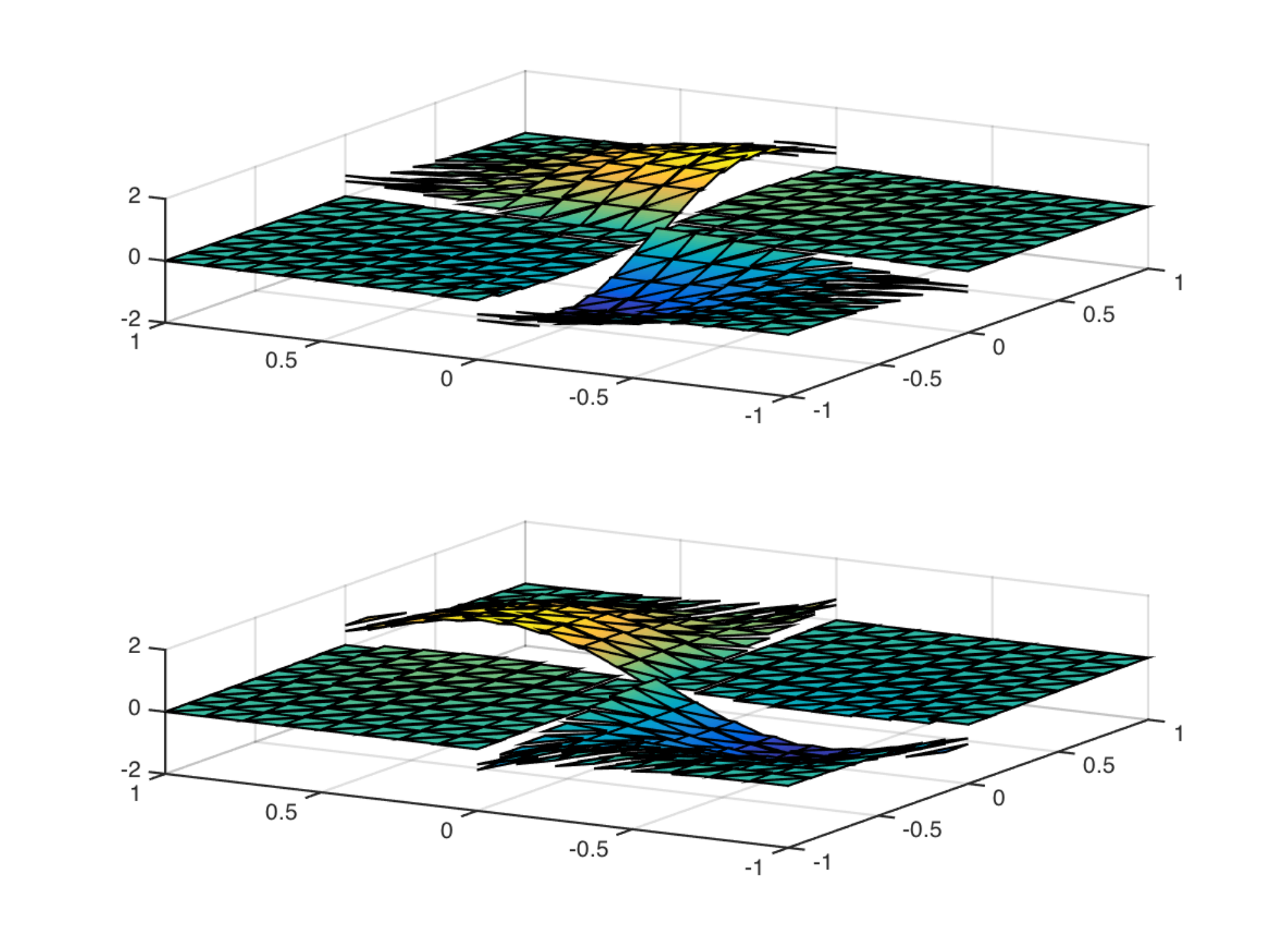} } 
			}                
	\end{subfigure} 
	%
	%
	%
	\caption{\textsc{Example} \ref{Ex Multiscale Jump Example}, approximate solution for a mesh of size $ h^{-1} = 8 $, the sub-domains are $\Omega_{1} = (-1, 0) \times (-1, 0)\cup  (0, 1) \times  (0, 1)$ and $\Omega_2 =  (-1, 0) \times (0, 1)\cup  (0, 1) \times  (-1, 0)$, see \textsc{Identity} \eqref{Def Geometric Parameters}.  Figure (a) depicts the pressure $p^{h}$ of the approximate solution, it is piecewise constant on the domain $\Omega_{1} $ and piecewise linear affine on the domain $\Omega_2$.
	Figure (b) depicts the flux of the approximate solution $\u^{h}$. On the upper right corner is depicted the $\boldsymbol{x}$-component of the flux 
	, which is continuous across \textbf{horizontal edges} of $\Omega_{1}$ and piecewise constant on the domain $\Omega_2$.	
	On the lower right corner we display 
	the $\boldsymbol{y}$-component of the flux, which is continuous across \textbf{vertical edges} of $\Omega_{1}$ and piecewise constant on the domain $\Omega_2$. Again (\textsc{Figure} \ref{Fig Multiscale Jump Approximate Solution Numerical Example}), the abrupt changes of $ \u_{x} $ across $ \{0\}\times (-1,1) $ and $ \u_{y} $ across $ (-1,1)\times\{0\} $ 
	in the exact solution, are now understood as numerical jumps. Notice the errors of this numerical approximation for the boundary conditions for the pressure in $ \Omega_{2} $ and for the flux in $ \Omega_{1} $ (compare with \textsc{Figures} \ref{Fig Multiscale Jump Exact Solution Numerical Example}, and \ref{Fig Multiscale Jump Approximate Solution Numerical Example}), while the weak boundary conditions \eqref{Eq Drained Condition decomposed} and \eqref{Eq Non-Flux Condition decomposed} are satisfied 
	\label{Fig Multiscale Jump Error Solution Numerical Example} }
\end{figure}
%
%
\begin{table}[h!]
\caption{Pressures Convergence Table, \textsc{Example} \ref{Ex Multiscale Jump Example}, $ \flux = - \frac{1}{\sqrt{2} } =  - \frac{1}{\sqrt{2} } \int_{\Gamma} p\big\vert_{\Gamma} \, dS $ }\label{Table Mult Jump Error Pressure Approximation}
\def\arraystretch{1.4}
\rowcolors{2}{gray!25}{white}
\begin{center}
\begin{tabular}{ c c c c c c c }
    \hline
    \rowcolor{gray!50}
$ h^{-1} $  
& $ \Vert  \poneh- p_{1}  \Vert_{ 0, \Omega_{1} } $ 
& 
$ 100 \frac{\Vert  \poneh- p_{1}  \Vert_{ 0, \Omega_{1} }}{\Vert  p_{1}  \Vert_{ 0, \Omega_{1} }} $  
& $ \Vert  \ptwoh- p_{2}  \Vert_{0, \Omega_{2} } $ 
& 
$ 100\frac{ \Vert  \ptwoh- p_{2}  \Vert_{0, \Omega_{2} } }{ \Vert   p_{2}  \Vert_{0, \Omega_{2} } } $
& $\Vert  \ptwoh- p_{2}  \Vert_{1, \Omega_{2} }$ 
& 
$ 100\frac{ \Vert  \ptwoh- p_{2}  \Vert_{1, \Omega_{2} } }{ \Vert   p_{2}  \Vert_{1, \Omega_{2} } } $ 
\\ 
    \toprule
$ 1 $ &   0.1594  &  28.1630 &  0.3785  & 142.7258  &   1.2474 & 198.4346  \\
$ 2 $  &  0.0558  &  5.1081  &  0.1441  & 14.0083   & 0.4945  &  60.9208   \\
$ 4 $ &   0.0414  &  1.9374 &  0.1046  & 4.9625    &  0.3232  &   29.0909   \\
$ 8 $  &  0.0420  &    0.9889 &  0.0880  & 2.0776   &  0.2364  &  14.5850    \\ 
$ 16 $  &  0.0425  &   0.5005 &  0.0822  & 0.9687    &  0.2090  &  7.3318  \\ 
$ 32 $  &  0.0426  &     0.2509  &  0.0801  & 0.4721   &  0.0375   &   3.6772  \\ 
    \hline
\end{tabular}
\end{center}
\end{table}
%
%
\begin{table*}[h!]
\caption{Velocities Convergence Table, \textsc{Example} \ref{Ex Multiscale Jump Example}, $ \flux = - \frac{1}{\sqrt{2} } =  - \frac{1}{\sqrt{2} } \int_{\Gamma} p\big\vert_{\Gamma} \, dS $}\label{Table Mult Jump Error Velocity Approximation}
\def\arraystretch{1.4}
\rowcolors{2}{gray!25}{white}
\begin{center}
\begin{tabular}{ c c c c c c c }
    \hline
    \rowcolor{gray!50}
$ h^{-1} $  
& $ \Vert  \uone^{h}- \uone  \Vert_{ 0, \Omega_{1} } $ 
& Rel. Error 
& $ \Vert  \uone^{h}- \uone  \Vert_{\Hdiv(\Omega_{1})} $ 
& Rel. Error 
& $ \Vert  \utwo^{h}- \utwo  \Vert_{ 0, \Omega_{2} } $ 
& Rel. Error 
\\ 
    \toprule
$ 1 $ &   1.6157 &  106.2303  &   1.6157 & 38.5853  &   0.3439  & 475.9703 \\
$ 2 $  &   0.4454  &  16.3705   &  0.4454  &  4.2843   &  0.0946   &  182.3322 \\
$ 4 $ &   0.3282  &  6.0312  &  0.3282  &   1.5974  &  0.0612   &  89.1516  \\
$ 8 $  &  0.2893  &   2.6582  &  0.2893  &   0.7058   &   0.0439 &  44.9504   \\ 
$ 16 $  &  0.2752  &  1.2644  &  0.2752  &  0.3360   &  0.0384  &  22.6303   \\ 
$ 32 $  &  0.2699  & 0.6200  &  0.2699   &  0.1648   &  0.0375  &  11.3547   \\ 
    \hline
\end{tabular}
\end{center}
\end{table*}
\end{example}
%
%
%
%
%
%
%
%
%
\section{Conclusions and Final Discussion}\label{Sec Conclusions}
%
%
%
%
The present work yields several conclusions summarized below
\begin{enumerate}[(i)]
\item A new conforming primal-dual mixed finite element scheme has been introduced successfully from both points of view: theoretical and numerical.

\item The theoretical analysis of the method includes variational formulation and well-posedness of the continuous problem as well as the choice of finite dimensional spaces, well-posedness (using the LBB theory) and convergence rates for the discrete problem.

\item The method is well-suited for analyzing multiscale porous media fluid flow problems such as oil extraction, groundwater flow and geological fissured systems.

\item The main technical advantages of the method are two: it can handle interface discontinuities which are consistent with the choice of the FEM spaces, see \textsc{Example} \ref{Ex Discontinuous Example}, and it can handle effectively multiscale phenomena since it can easily introduce numerical jumps across the interfaces, see \textsc{Example} \ref{Ex Multiscale Continuous Example}. The latter is numerically convenient even when the exact solution is continuous but it has abrupt changes, see \textsc{Example} \ref{Ex Multiscale Jump Example}. Of course the method can handle regular problems, free of multiple scales and discontinuities, see \textsc{Example} \ref{Ex Continuous Example}.

\item The power of the method lies in the fact that the FEM spaces do not embed strong coupling conditions between regions, on the contrary, they are fully uncoupled and the fluid exchange conditions only hold for the solution (either numerical or theoretical), but not for the test functions.

\item Throughout the pressure convergence tables \ref{Table Pressure Approximation}, \ref{Table Disc Pressure Approximation}, \ref{Table Mult Pressure Approximation} and \ref{Table Mult Jump Pressure Approximation} a substantial superconvergence phenomenon is observed for $ \Vert \poneh - \pone\Vert_{0, \Omega_{1}} $. In \textsc{Tables}  \ref{Table Mult Pressure Approximation} and \ref{Table Mult Jump Pressure Approximation} a mild superconvergence behavior is observed for $ \Vert \ptwoh - \ptwo \Vert_{0, \Omega_{2}} $. It is important to stress that this work made no attempt to present a method with enhanced convergence properties, all the more reason considering that the convergence rate analysis presented in \textsc{Section} \ref{Sec Rate of Convergence} delivers the usual rates of convergence. These observations may come from the regular gridding of the domain or from the particular chosen examples. This will be discussed in future work either by finding examples breaking the superconvergence or developing a new approach to analysis of the convergence rates different from the standard one. 

\item \textsc{Example} \ref{Ex Multiscale Jump Example}, is composed of two parts. The first part is the usual analysis displaying the performance of the method under controlled/lab conditions (\textsc{Tables} \ref{Table Mult Jump Pressure Approximation}, \ref{Table Mult Jump Velocity Approximation}, \textsc{Figure} \ref{Fig Multiscale Jump Approximate Solution Numerical Example}). The second part suggests an iterative heuristic method to attain better numerical results in multiscale problems: start from reasonable (empirical if possible) values of the pressure on the interfaces, use the computed numerical pressure $ \poneh\big\vert_{\Gamma} $ as input for a new iteration and continue in this fashion, until the results attain a desired level of stability from one iteration to the next one. The primal-dual mixed scheme certainly allows to proceed this way, however analyzing is such a method is convergent or under which conditions converges is topic for future work.

\item Finally, the implementation for the 3D porous media problem of the same method should not pose substantial theoretical challenges, but computational ones due to its complexity. The development of such implementation for general domains and grids in a public domain fashion is the topic of future work. 
%
\end{enumerate}
%
%
%
%
%
%
%
\section*{Acknowledgements}
%
%
%
%
The Author wishes to acknowledge Universidad Nacional de Colombia, Sede Medell\'in for its support in this work through the project HERMES 27798. The Author also wishes to thank Professor Carsten Carstensen, from Institut f\"{u}r Mathematik, Humboldt-Universit\"at zu Berlin, Germany, for making freely available his software \textbf{EBmfem.m} and \textbf{fem2d.m}. Without these priceless tools, the implementation presented in \textsc{Section}  \ref{Sec Numerical Example} would have not been possible. Thanks to Professor Bibiana L\'opez Rodr\'iguez from Universidad Nacional de Colombia, Sede Medell\'in, for helping the Author through multiple discussions in the paper's production. Special thanks to Professor Ma\l{}gorzata Peszy\'nska from Oregon State University, whose teachings have guided the Author across all the stages of this work.




\end{document}